\documentclass[a4paper,reqno]{amsart}
\usepackage{amssymb, amsmath, amsthm, mathrsfs, ascmac,color}
\usepackage[pdftex]{graphicx} 
\usepackage[subrefformat=parens]{subcaption}
\usepackage[top=2.5cm, bottom=2.5cm, left=3cm,right=3cm]{geometry}
\usepackage{comment}
\usepackage{harvard}
\usepackage[foot]{amsaddr}
\newtheoremstyle{mystyle}
  {}
  {}
  {\normalfont}
  { }
  {\bfseries}
  {}
  {10pt}
  { }
\theoremstyle{mystyle}
\newtheorem{thm}{Theorem}

\newtheorem{lem}{Lemma}

\newtheorem{rmk}{Remark}
\makeatletter

\@addtoreset{equation}{section}
\makeatother

\newcommand{\dd}{\mathrm d}

\newcommand{\EE}{\mathbb E}

\newcommand{\GG}{\mathscr G_{i-1}^n}

\newcommand{\DeX}{\Delta X_i}
\newcommand{\Xt}{X_{t_i}}
\newcommand{\Xs}{X_{t_{i-1}}}
\newcommand{\tr}{\mathrm{tr}}

\newcommand{\lto}{\longrightarrow}
\newcommand{\pto}{\stackrel{p}{\longrightarrow}}
\newcommand{\dto}{\stackrel{d}{\longrightarrow}}

\newcommand{\TT}{\mathsf T}
\newcommand{\Mi}{\mathcal M_{2,i}}
\newcommand{\Qn}{\mathcal Q_n}

\newcommand{\Dea}{\vartheta_\alpha}
\newcommand{\Deb}{\vartheta_\beta}
\newcommand{\Dek}{\vartheta_{\beta_k}}
\newcommand{\Deko}{\vartheta_{\beta_1}}
\newcommand{\Dekt}{\vartheta_{\beta_2}}
\newcommand{\TUO}{{\mathcal T}_{1,n}^{(1)}}
\newcommand{\TOO}{{\mathcal T}_{1,n}^{(2)}}
\newcommand{\TUT}{{\mathcal T}_{2,n}^{(1)}}
\newcommand{\TOT}{{\mathcal T}_{2,n}^{(2)}}
\newcommand{\TLA}{{\mathcal L}_{1,n}^{(1)}}
\newcommand{\TLB}{{\mathcal L}_{1,n}^{(2)}}
\newcommand{\TLC}{{\mathcal L}_{2,n}^{(1)}}
\newcommand{\TLD}{{\mathcal L}_{2,n}^{(2)}}
\newcommand{\TUA}{{\mathcal U}_{1,n}^{(1)}}
\newcommand{\TUB}{{\mathcal U}_{1,n}^{(2)}}
\newcommand{\TUC}{{\mathcal U}_{2,n}^{(1)}}
\newcommand{\TUD}{{\mathcal U}_{2,n}^{(2)}}
\newcommand{\mma}{{\mathrm m}_1}
\newcommand{\mmb}{{\mathrm m}_2}
\newcommand{\mmc}{{\mathrm m}_3}
\newcommand{\mmd}{{\mathrm m}_4}

\allowdisplaybreaks
\begin{document}
\bibliographystyle{plain}
\title[
	Change point inference 
	in ergodic diffusion processes
]
{
	Change point inference 
	in ergodic diffusion processes
	based on high frequency data
}
\author[Y. Tonaki]{Yozo Tonaki}
\address[Y. Tonaki]{
	Graduate School of Engineering Science, Osaka University, 1-3, 
	Machikaneyama, Toyonaka, Osaka, 560-8531, Japan
}
\author[M. Uchida]{Masayuki Uchida}
\address[M. Uchida]{
	Graduate School of Engineering Science, 
	and Center for Mathematical Modeling and Date Science,
	Osaka University, 1-3, Machikaneyama, Toyonaka, Osaka, 560-8531, Japan
}


\maketitle

\begin{abstract}
We deal with the change point problem in ergodic diffusion processes
based on high frequency data.
Tonaki et al. (2020, 2021) 
studied the change point problem for the ergodic diffusion process model.
However, the change point problem for the drift parameter 
when the diffusion parameter changes is still open.
Therefore, we consider the change detection 
and the change point estimation for the drift parameter 
taking into account that there is a change point in the diffusion parameter. 
Moreover, we examine the performance of the tests and the estimation
with numerical simulations.
\end{abstract}

\section{Introduction}\label{sec1}
We consider a $d$-dimensional diffusion process $\{X_t\}_{t\ge0}$ 
satisfying the stochastic differential equation
\begin{align}\label{sde}
\begin{cases}
\dd X_t=b(X_t,\beta)\dd t+a(X_t,\alpha)\dd W_t,\\
X_0=x_0,
\end{cases}
\end{align}
where parameter space
$\Theta=\Theta_A\times\Theta_B$, which is a  
compact convex subset of 
$\mathbb R^p\times\mathbb R^q$, 
$\theta=(\alpha,\beta)\in\Theta$ is an unknown parameter and
$\{W_t\}_{t\ge0}$ 
is an $r$-dimensional standard Wiener process.
The diffusion coefficient 
$a:\mathbb R^d\times\Theta_A\lto\mathbb R^d\otimes\mathbb R^r$ and
the drift coefficient $b:\mathbb R^d\times\Theta_B\lto\mathbb R^d$
are known
except for the parameter $\theta$. 
We assume that the solution of \eqref{sde} exists, and 
$P_\theta$ and $\EE_\theta$ denote the law of the 
solution and the expectation with respect to $P_\theta$, respectively. 
Let $\{\Xt\}_{i=0}^n$ be discrete observations, 
where $t_i=t_i^n=ih_n$, 
and $\{h_n\}$ is a positive sequence with $h=h_n\lto0$, $T=nh\lto\infty$ and 
$nh^2\lto0$ as $n\to\infty$.

In this paper, we consider the change point problem 
for ergodic diffusion processes.
The change point problem for diffusion processes
based on discrete observations 
has been studied by many researchers,
for non-ergodic, see
De Gregorio and Iacus (2008), 
Iacus and Yoshida (2012) 
and
for ergodic, see
Song and Lee (2009), 
Lee (2011),
Negri and Nishiyama (2017),
Song (2020),
Tonaki et al. (2020,2021)
and reference therein.
De Gregorio and Iacus (2008) 
considered the test for changes in the diffusion parameter 
and the change point estimation in non-ergodic case, 
and 
Iacus and Yoshida (2012)
generalized the estimation method.
Song and Lee (2009),
Lee (2011),
and 
Song (2020)
considered the test for changes in the diffusion parameter in ergodic case.
Furthermore, 
Negri and Nishiyama (2017)
and 
Tonaki et al. (2020)
proposed simultaneous test and adaptive tests 
for changes in the diffusion and drift parameters, respectively,
and
Tonaki et al. (2021)
treated change point estimation 
for the case where there is a change in the diffusion parameter 
and 
for the case where there is no change in the diffusion parameter 
but there is a change in the drift parameter.
The method proposed by Tonaki et al. (2020,2021)
is summarized as follows.
First, test whether there is any change in the diffusion parameter.
If any change is detected, estimate the change point of the diffusion parameter.
If not, assume that there is no change in the diffusion parameter, and
test whether there is any change in the drift parameter.
If any change is detected, estimate the change point of the drift parameter.
This method allows us to estimate the change point in the diffusion parameter 
if a change in the diffusion parameter is detected, 
and to estimate the change point in the drift parameter 
if no change in the diffusion parameter is detected, 
but a change in the drift parameter is detected.

\begin{figure}[h]
\captionsetup{margin=5pt}
\captionsetup[sub]{margin=5pt}
 \begin{minipage}[t]{0.49\linewidth}
  \centering
  \includegraphics[keepaspectratio, width=70mm]{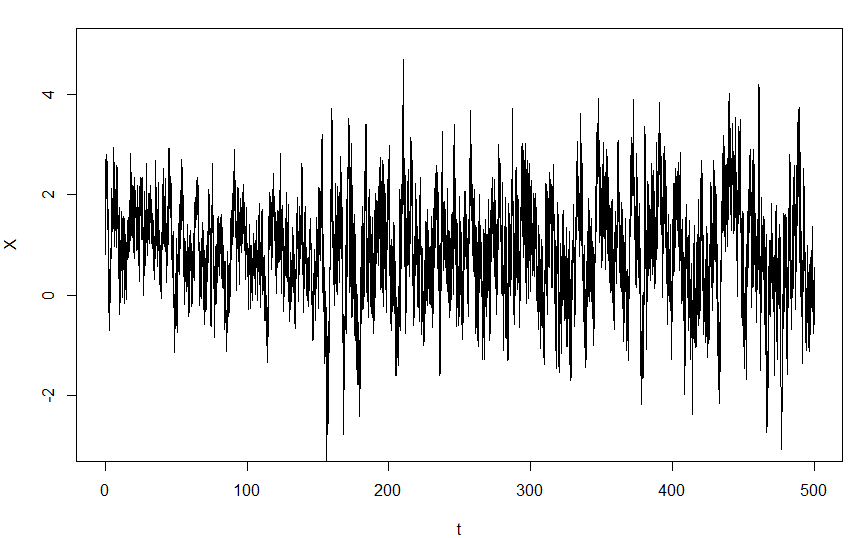}
  \subcaption{
  		$\alpha$ changes from $1$ to $1.5$ at $t=150$,
  		and
  		$(\beta,\gamma)=(1,1)$ does not change.
  }\label{ex1}
 \end{minipage}
 \begin{minipage}[t]{0.49\linewidth}
  \centering
  \includegraphics[keepaspectratio, width=70mm]{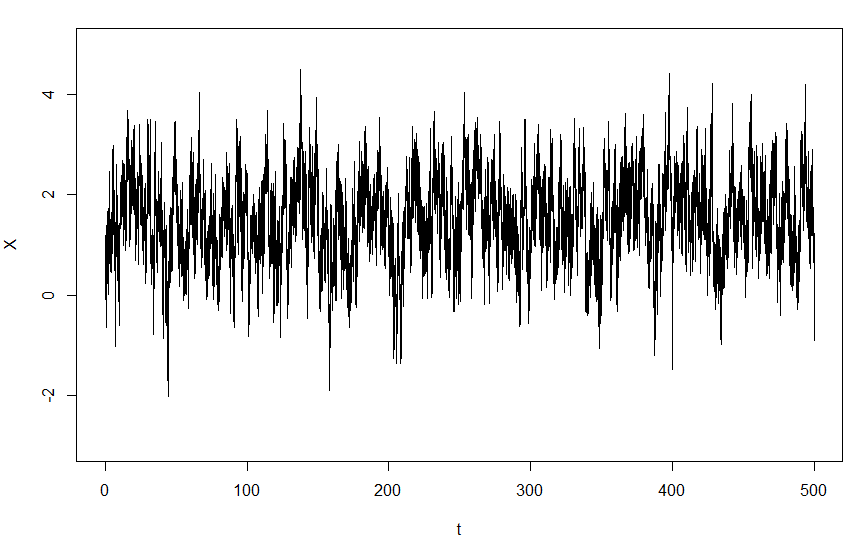}
  \subcaption{
  		$(\alpha,\beta,\gamma)=(1.2,1,1.5)$ does not change.
  }\label{ex2}
 \end{minipage}\\
 \begin{minipage}[t]{0.49\linewidth}
  \centering
  \includegraphics[keepaspectratio, width=70mm]{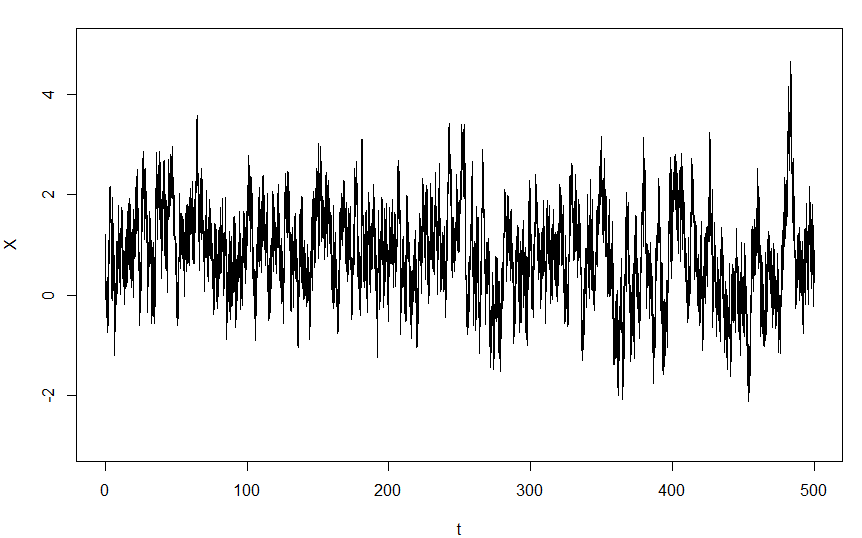}
  \subcaption{
  		$\alpha=1$ does not change, 
  		and
  		$(\beta,\gamma)$ changes from $(1,1)$ to $(0.5,0.5)$ at $t=250$.
  }\label{ex3}
 \end{minipage}
 \begin{minipage}[t]{0.49\linewidth}
  \centering
  \includegraphics[keepaspectratio, width=70mm]{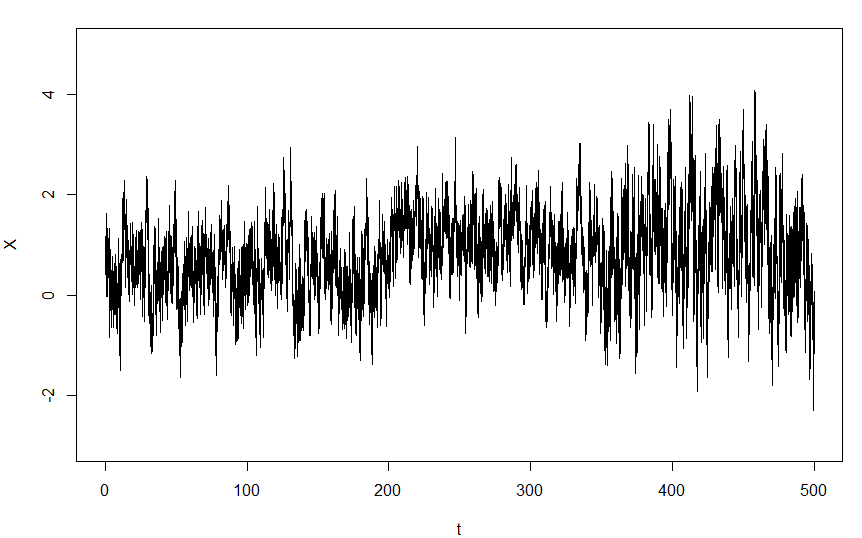}
  \subcaption{
  		$\alpha$ changes from $1$ to $1.5$ at $t=350$,
  		and
  		$(\beta,\gamma)$ changes from $(1.3,0.5)$ to $(1.3,1)$ at $t=200$.
  }\label{ex4}
 \end{minipage}
 \caption{
 		Sample paths of the Ornstein-Uhlenbeck process
 		$\dd X_t=-\beta(X_t-\gamma)\dd t+\alpha\dd W_t$.
 }
 \label{fig_intro}
\end{figure}

The purposes of the change point problem in the ergodic diffusion processes 
are to investigate whether there is a change point 
in the diffusion and drift parameters, and if so, where it is. 
Given the data for the paths as shown in Figure \ref{fig_intro},
using the method of Tonaki et al. (2020,2021),
we can infer that the diffusion parameter changes at $t=150$ and $t=350$ 
for (A) and (D), respectively, 
and that there is no change in the diffusion parameter, 
but the drift parameter changes at $t=250$ for (C). 
Furthermore, for (B), 
it can be inferred that 
there is no change in either the diffusion or drift parameters.
However, Tonaki et al. (2020,2021)
do not provide discussion 
for the change of the drift parameter 
when there is a change in the diffusion parameter such as (A) and (D).
In other words, 
we have no way to infer that 
there is no change point in the drift parameter in (A), 
and that the drift parameter changes at $t=200$ in (D).
Thus, in this paper, 
we propose a method for 
detecting changes in the drift parameter 
and estimating the change point 
when the diffusion parameter changes. 

This paper is organized as follows.
Section \ref{secR} reviews the change point inference for the diffusion parameter.
In Section \ref{sec2}, 
we propose test methods to detect changes in the drift parameter 
when there is a change point in the diffusion parameter,
and
show the asymptotic properties 
under the null hypothesis 
and the consistency of the proposed test statistics.
Moreover, we propose change point estimators for the drift parameter
in Section \ref{sec3}.
In Section \ref{sec4}, we give some examples and simulation studies.
Finally, Section \ref{sec5} is devoted to the proofs of the results 
of Sections \ref{sec2} and \ref{sec3}.

\section{Change point inference for diffusion parameter}\label{secR}
In this section, 
we review the change point detection and estimation 
for the diffusion parameter 
before considering the change point problem for the drift parameter 
when there is a change point in the diffusion parameter.
See Tonaki et al. (2020,2021) for details.

We set the following notations.
\begin{enumerate}
\item[1.]
For a matrix $M$,
$M^\TT$ denotes the transpose of $M$
and $M^{\otimes2}=MM^\TT$.


\item[2.]
Let $A(x,\alpha)=a(x,\alpha)^{\otimes2}$ and $\DeX=\Xt-\Xs$.

\item[3.]
For $k=1,2, \ldots$, and $\epsilon\in(0,1)$,
$\{\boldsymbol B_k^0(s)\}_{0\le s\le 1}$ denotes
a $k$-dimensional Brownian bridge,
and
$w_k(\epsilon)$ denotes 
the upper-$\epsilon$ point  of 
$\displaystyle \sup_{0\le s\le 1}\|\boldsymbol B_k^0(s)\|$, that is,
\begin{align*}
P\biggl(\sup_{0\le s\le 1}\|\boldsymbol B_k^0(s)\|>w_k(\epsilon)\biggr)=\epsilon.
\end{align*}

\item[4.]
Let $\{\mathbb W(s)\}_{s\ge0}$ be a two-sided standard Wiener process.

\item[5.]
Let $C^{k,\ell}_{\uparrow}(\mathbb R^d\times\Theta)$ be 
the space of all functions $f$ satisfying the following conditions.

\begin{enumerate}
\item[(i)] 
$f$ is continuously differentiable with respect to $x\in\mathbb R^d$ up to
order $k$ for all $\theta\in\Theta$,
\item[(ii)]  
$f$ and all its $x$-derivatives up to order $k$ are 
$\ell$ times continuously differentiable with respect to $\theta\in\Theta$,
\item[(iii)]  
$f$ and all derivatives are of polynomial growth in $x\in\mathbb R^d$ 
uniformly in $\theta\in\Theta$, i.e.,  
$g$ is of polynomial growth in $x\in\mathbb R^d$ uniformly 
in $\theta\in\Theta$ if, for some $C>0$, 
\begin{align*}
\sup_{\theta\in\Theta}\|g(x,\theta)\|\le C(1+\|x\|)^C.
\end{align*}
\end{enumerate}

\end{enumerate}

We make the following assumptions.
\begin{enumerate}
\renewcommand{\labelenumi}{{\textbf{[C\arabic{enumi}]}}}

\item 
There exists a constant $C>0$ such that for any $x,y\in\mathbb R^d$, 
\begin{align*}
\sup_{\alpha\in\Theta_A}\|a(x,\alpha)-a(y,\alpha)\|
+\sup_{\beta\in\Theta_B}\|b(x,\beta)-b(y,\beta)\|\le C\|x-y\|.
\end{align*}

\item   
$\displaystyle\sup_t\EE_{\theta}\bigl[\|X_t\|^k\bigr]<\infty$
for all $k\ge0$ and $\theta\in\Theta$.

\item $\displaystyle\inf_{x,\alpha}\det A(x,\alpha)>0$.

\item 
$a\in C^{4,4}_{\uparrow}(\mathbb R^d\times\Theta_A)$
and
$b\in C^{4,4}_{\uparrow}(\mathbb R^d\times\Theta_B)$.

\item 
The solution of \eqref{sde} is ergodic with its invariant measure $\mu_\theta$ 
such that 
\begin{align*}
\int_{\mathbb R^d}\|x\|^k\dd\mu_\theta(x)<\infty
\end{align*}
for all $k\ge0$ and $\theta\in\Theta$,
and for all measurable function $f$,  
\begin{align*}
\int_{\mathbb R^d}f(x)\dd \mu_{\theta_n}(x)
\lto
\int_{\mathbb R^d}f(x)\dd \mu_{\theta_0}(x)
\end{align*}
as $\theta_n\lto\theta_0$.

\item
If there exist $m\ge1$, 
$\theta_0,\ldots,\theta_{m-1}\in\mathrm{Int}\, \Theta$, 
$\tau_1,\ldots,\tau_{m-1}\in(0,1)$
such that
\begin{align*}
X_t
=
\begin{cases}
X_t(\theta_0), &t\in[0,\tau_1 T),\\
X_t(\theta_1), &t\in[\tau_1 T,\tau_2 T),\\
\quad\vdots\\
X_t(\theta_{m-1}), &t\in[\tau_{m-1} T,T],
\end{cases}
\end{align*}
then for any $f\in C^{1,1}_{\uparrow}(\mathbb R^d\times\Theta)$, 
$j=1,\ldots,m-1$,
and $\delta\in(1,2)$ with $nh^\delta\lto\infty$,
\begin{align*}
\max_{[n^{1/\delta}]\le k\le [n\tau_{j+1}]-[n\tau_j]}
\left|
\frac1k\sum_{i=[n\tau_j]+1}^{[n\tau_j]+k}f(\Xs,\theta_j)
-\int_{\mathbb R^d}f(x,\theta_j)\dd\mu_{\theta_j}(x)
\right|\pto0.
\end{align*}
\end{enumerate}

\begin{rmk}
If for all $j=1,2,\ldots,m-1$, 
$\{X_t^{(j)}\}_{t\ge0}$, $X_t^{(j)}=X_t(\theta_j)$ are stationary, 
then \textbf{[C6]} is satisfied from Lemma 4.3 of Song and Lee (2009).
\end{rmk}

\subsection{Test for a change in diffusion parameter}\label{secR-1}
First, we consider the change detection for the diffusion parameter.
For simplicity, we assume that there is a change point 
under the alternative hypothesis, that is, 
we consider the following hypothesis testing problem.

$H_0^\alpha: \alpha$ does not change over $[0,T]$
\ v.s.\   
$H_1^\alpha:$ There exists $\tau_*^\alpha\in(0,1)$ such that 
\begin{align*}
\alpha
=
\begin{cases}
\alpha_{1}^*, &t\in[0,\tau_*^\alpha T),\\
\alpha_{2}^*, &t\in[\tau_*^\alpha T,T],
\end{cases}
\end{align*}
where $\alpha_1^*,\alpha_2^*\in\mathrm{Int}\, \Theta_A$
and $\alpha_1^*\neq\alpha_2^*$.

We consider the following two cases.
\begin{description}
\item[Case A]
The parameters $\alpha_1^*$ and $\alpha_2^*$ depend on $n$, 

\item[Case B]
The parameters $\alpha_1^*$ and $\alpha_2^*$ are fixed and not depend on $n$.
\end{description}
In Case A, we assume the following convergences hold.
\begin{enumerate}
\item[(i)]
$\Dea=|\alpha_1^*-\alpha_2^*|\lto0$ and $n\Dea^2\lto\infty$ as $n\to\infty$.

\item[(ii)]
$\alpha_1^*$ and $\alpha_2^*$ converge to $\alpha_0\in\mathrm{Int}\,\Theta_A$
as $n\to\infty$.

\end{enumerate}
We make the following assumptions.
\begin{enumerate}
\renewcommand{\labelenumi}{\textbf{[D\arabic{enumi}]}}

\item Under $H_0^\alpha$, there exists an estimator $\hat\alpha$ such that 
$\sqrt n(\hat\alpha-\alpha)=O_p(1)$.

\item
Under $H_1^\alpha$,
there exists an estimator $\hat\alpha$ such that 
$\vartheta_\alpha^{-1}(\hat\alpha-\alpha_0)=O_p(1)$
in Case A.

\item 
$\vartheta_\alpha^{-1}(\alpha_k^*-\alpha_0)$ converges to $c_k\in\mathbb R^p$ 
as $n\to\infty$ for $k=1,2$.
Under $H_1^\alpha$,
\begin{align*}
\int_{\mathbb R^d}
\left[\tr\bigl(
A^{-1}\partial_{\alpha^\ell} A(x,\alpha_0)
\bigr)
\right]_\ell
\dd\mu_{\alpha_0}(x)(c_1-c_2)\neq0
\end{align*}
in Case A.

\item Under $H_1^\alpha$, there exist $\alpha'\in\mathrm{Int}\,\Theta_A$ and 
an estimator $\hat\alpha$ such that $\hat\alpha-\alpha'=o_p(1)$
in Case B.

\end{enumerate}
Let
\begin{align*}
F(\alpha)=\int_{\mathbb R^d}
\tr\bigl(
A^{-1}(x,\alpha')A(x,\alpha)
\bigr)\dd\mu_{\alpha}(x).
\end{align*}
\begin{enumerate}
\renewcommand{\labelenumi}{\textbf{[D\arabic{enumi}]}}
\setcounter{enumi}{4}
\item
Under $H_1^\alpha$, $F(\alpha_1^*)\neq F(\alpha_2^*)$  
in Case B.

\end{enumerate}
\begin{rmk}
See Uchida and Yoshida (2012,2014), Kaino and Uchida (2018), and Tonaki et al. (2020) 
for constructing an estimator $\hat\alpha$ 
that satisfies \textbf{[D1]}, 
Tonaki et al. (2021) for \textbf{[D2]}, 
and Uchida and Yoshida (2011) and Tonaki et al. (2020) for \textbf{[D4]}.
\end{rmk}
Let 
\begin{align*}
\hat\eta_i
&=\tr\biggl(A^{-1}(\Xs,\hat\alpha)
\frac{(\Delta X_i)^{\otimes2}}{h}\biggr).
\end{align*}
The test statistic for the diffusion parameter is as follows.
\begin{align*}
{\mathcal T}_n^\alpha=\frac{1}{\sqrt{2dn}}\max_{1\le k\le n}
\left|\sum_{i=1}^k\hat\eta_i-\frac{k}{n}\sum_{i=1}^n\hat\eta_i\right|.
\end{align*}
The following theorem is the result for the asymptotic distribution 
under the null hypothesis and the consistency of the test $\mathcal T_n^\alpha$.
\begin{thm}[Tonaki et al., 2020,2021]\label{th1_TKU}
Suppose that \textbf{[C1]}-\textbf{[C6]} hold.
\begin{enumerate}
\item[(1)]
If \textbf{[D1]} is satisfied,
then 
$\mathcal T_n^\alpha\dto\displaystyle\sup_{0\le s\le 1}|\boldsymbol B_1^0(s)|$
under $H_0^\alpha$.

\item[(2)]
If \textbf{[D2]} and \textbf{[D3]} are satisfied,
then for $\epsilon\in(0,1)$,
$P(\mathcal T_n^\alpha>w_1(\epsilon))\lto1$ under $H_1^\alpha$ in Case A.

\item[(3)]
If  \textbf{[D4]} and \textbf{[D5]} are satisfied,
then for $\epsilon\in(0,1)$,
$P(\mathcal T_n^\alpha>w_1(\epsilon))\lto1$ under $H_1^\alpha$ in Case B.

\end{enumerate}
\end{thm}

\subsection{Estimation for a change in diffusion parameter}\label{secR-2}

If $H_0^\alpha$ is rejected, 
we next consider the change point estimation problem 
for the following stochastic differential equation.
\begin{align*}
X_t
=
\begin{cases}
\displaystyle
X_0+\int_0^t b(X_s,\beta)\dd s+\int_0^t a(X_s,\alpha_1^*) \dd W_s, 
& t\in[0,\tau_*^\alpha T),\\
\displaystyle
X_{\tau_*^\alpha T}+\int_{\tau_*^\alpha T}^t b(X_s,\beta)\dd s
+\int_{\tau_*^\alpha T}^t a(X_s,\alpha_2^*) \dd W_s, & t\in[\tau_*^\alpha T,T],
\end{cases}
\end{align*}
where $\alpha_1^*,\alpha_2^*\in\mathrm{Int}\, \Theta_A$, $\alpha_1^*\neq\alpha_2^*$,
and $\tau_*^\alpha\in(0,1)$ is unknown.

We make the following assumptions.
\begin{enumerate}
\renewcommand{\labelenumi}{\textbf{[D\arabic{enumi}]}}
\setcounter{enumi}{5}

\item 
There exist estimators $\hat\alpha_k$ ($k=1,2$) such that  
$\sqrt n(\hat\alpha_k-\alpha_k^*)=O_p(1)$.

\item
$h/\Dea^2\lto\infty$ as $n\to\infty$, and $\Dea^{-1}(\alpha_k^*-\alpha_0)=O(1)$
in Case A.
\end{enumerate}
Let 
\begin{align*}
\Gamma^\alpha(x,\alpha_1,\alpha_2) 
=\tr\bigl(A^{-1}(x,\alpha_1)A(x,\alpha_2)-I_d\bigr)
-\log\det A^{-1}(x,\alpha_1)A(x,\alpha_2),
\end{align*}
and
$Q(x,\theta)$ be the coefficient of $h^2$ 
of $\EE_\theta[(\DeX)^{\otimes2}|\GG]$, 
where $\GG=\sigma\bigl[\{W_s\}_{s\le t_i^n}\bigr]$.

\begin{enumerate}
\renewcommand{\labelenumi}{\textbf{[D\arabic{enumi}]}}
\setcounter{enumi}{7}
\item
$\displaystyle
\inf_x \Gamma^\alpha(x,\alpha_1^*,\alpha_2^*)>0$.
\item
There exists a constant $C>0$ such that
\begin{enumerate}
\item
$\displaystyle\sup_{x,\alpha_k}
\bigl|
\partial_{(\alpha_1,\alpha_2)}\Gamma^\alpha(x,\alpha_1,\alpha_2)
\bigr|<C$,
\item 
$\displaystyle\sup_{x,\alpha_k}
\left|
\left[
\tr\Bigl(
\bigl(A^{-1}(x,\alpha_1)-A^{-1}(x,\alpha_2)\bigr)
\partial_{\alpha^{\ell}} A(x,\alpha_3)
\Bigr)
\right]_{\ell=1}^p
\right|<C$,
\item 
$\displaystyle\sup_{x,\theta}|Q(x,\theta)|<C$.
\end{enumerate}
\end{enumerate}
\begin{rmk}
See Section 3 in Tonaki et al. (2021) 
for how to construct estimators $\hat\alpha_1$ and $\hat\alpha_2$ 
satisfying \textbf{[D6]}. 
Tonaki et al. (2021)
assumed $T\Dea\lto0$ 
in addition to condition \textbf{[D7]}, 
but there is no need to assume it under \textbf{[C6]}.
The assumption that
$h/\Dea^2\lto\infty$ in \textbf{[D7]} is not necessary when $a(x,\alpha)=\alpha$.
\end{rmk}
Let
\begin{align*}
&F_i(\alpha)
=
\tr\left(A^{-1}(\Xs,\alpha)\frac{(\DeX)^{\otimes 2}}{h}\right)
+\log\det A(\Xs,\alpha),\\
&\Phi_n(\tau:\alpha_1,\alpha_2)
=
\sum_{i=1}^{[n\tau]}F_i(\alpha_1)+\sum_{i=[n\tau]+1}^nF_i(\alpha_2).
\end{align*}
The change point estimator for the diffusion parameter is as follows.
\begin{align*}
\hat\tau_n^\alpha
=\underset{\tau\in[0,1]}{\mathrm{argmin}}\,\Phi_n(\tau:\hat\alpha_1,\hat\alpha_2).
\end{align*}
In Case A, let for $v\in\mathbb R$,
\begin{align*}
e_\alpha
&=\lim_{n\to\infty}\Dea^{-1}(\alpha_1^*-\alpha_2^*),\\
\Xi^\alpha(x,\alpha)
&=
\Bigl[
\tr\bigl(
A^{-1}\partial_{\alpha^{\ell_1}}A
A^{-1}\partial_{\alpha^{\ell_2}}A(x,\alpha)
\bigr)
\Bigr]_{\ell_1,\ell_2=1}^p,\\
\mathcal J_\alpha
&=\frac12e_\alpha^\TT 
\int_{\mathbb R^d}
\Xi^\alpha(x,\alpha_0)
\dd\mu_{\alpha_0}(x)
e_\alpha,\\
\mathbb F(v)
&=-2\mathcal J_\alpha^{1/2}\mathbb W(v)+\mathcal J_\alpha|v|.
\end{align*}
The following results give the asymptotic behavior 
of the estimator $\hat\tau_n^\alpha$. 
\begin{thm}[Tonaki et al., 2021]\label{th2_TKU}
Suppose that \textbf{[C1]}-\textbf{[C6]} and \textbf{[D6]} hold.
\begin{enumerate}
\item[(1)]
Under  \textbf{[D7]}, 
$n\Dea^2(\hat\tau_n^\alpha-\tau_*^\alpha)
\dto\underset{v\in\mathbb R}{\mathrm{argmin}}\,\mathbb F(v)$
in Case A.

\item[(2)]
Under  \textbf{[D8]} and \textbf{[D9]}, 
$n(\hat\tau_n^\alpha-\tau_*^\alpha)=O_p(1)$
in Case B. 

\item[(3)]
Under  \textbf{[D8]},  \textbf{[D9]}(a) and (b), 
$n^{\epsilon_1}(\hat\tau_n^\alpha-\tau_*^\alpha)=o_p(1)$
for $\epsilon_1\in[0,\frac12)$
in Case B. 

\end{enumerate}
\end{thm}

\begin{rmk}
Since the $1$-dimensional Ornstein-Uhlenbeck process defined by
$\dd X_t=-\beta(X_t-\gamma)\dd t+\alpha\dd W_t$ 
($\alpha,\beta>0$, $\gamma\in\mathbb R$)
does not satisfy \textbf{[D9]}(3), but satisfies \textbf{[D9]}(1) and (2), 
we can estimate $\tau_*^\alpha$ by Theorem \ref{th2_TKU} (3) in this model. 
In contrast, the Hyperbolic diffusion model defined by
$\dd X_t=(\beta-\gamma X_t/\sqrt{1+X_t^2})\dd t+\alpha\dd W_t$ 
($\alpha>0$, $\beta\in\mathbb R$, $|\beta|<\gamma$) 
satisfies \textbf{[D9]}, and therefore $\tau_*^\alpha$ can be estimated 
by Theorem \ref{th2_TKU} (2) in this model.
\end{rmk}

\subsection{
Test and estimation for the drift parameter 
when $\boldsymbol{H_0^\alpha}$ is not rejected
}

If $H_0^\alpha$ is not rejected, i.e., 
no change in the diffusion parameter is detected, 
the next step is to investigate the change in the drift parameter. 
Specifically, 
the existence of change points in drift parameter is investigated,
and if any change is detected, the change point is estimated. 
In this case, see Subsection 2.2 of Tonaki et al. (2020) 
for the change detection method 
and Subsection 2.2 of Tonaki et al. (2021) 
for the change point estimation method.

\section{Change point detection for drift parameter}\label{sec2}

The purpose of this paper is to infer the change point in the drift parameter 
when there is a change point in the diffusion parameter. 
In the following, we consider the situation that $H_0^\alpha$ is rejected 
in the hypothesis testing problem 
in Subsection \ref{secR-1}, 
and the change point of the diffusion parameter is estimated 
in Subsection \ref{secR-2}. 

\subsection{Test for changes in drift parameter}\label{sec2-1}
In this subsection, 
we consider the change detection for the drift parameter 
when there is a change point in the diffusion parameter, that is, 
we treat the following stochastic differential equation.
\begin{align*}
X_t
=
\begin{cases}
\displaystyle
X_0+\int_0^t b(X_s,\beta_1)\dd s+\int_0^t a(X_s,\alpha_1^*) \dd W_s, 
& t\in[0,\tau_*^\alpha T),\\
\displaystyle
X_{\tau_*^\alpha T}+\int_{\tau_*^\alpha T}^t b(X_s,\beta_2)\dd s
+\int_{\tau_*^\alpha T}^t a(X_s,\alpha_2^*) \dd W_s, & t\in[\tau_*^\alpha T,T],
\end{cases}
\end{align*}
where $\alpha_1^*,\alpha_2^*\in\mathrm{Int}\, \Theta_A$, 
$\alpha_1^*\neq\alpha_2^*$,
$\tau_*^\alpha\in(0,1)$ is unknown,
and
$\beta_1$ and $\beta_2$ may change, or be equal.

We consider the following two hypothesis testing problems.
\begin{center}
$H_0^{(1)}: \beta_1$ does not change over $[0, \tau_*^\alpha T]$
\quad v.s.\quad   
$H_1^{(1)}:$ not $H_0^{(1)}$
\end{center}
and
\begin{center}
$H_0^{(2)}: \beta_2$ does not change over $[\tau_*^\alpha T,T]$
\quad v.s.\quad   
$H_1^{(2)}:$ not $H_0^{(2)}$.
\end{center}
We make the following assumptions.
\begin{enumerate}
\renewcommand{\labelenumi}{{\textbf{[E\arabic{enumi}]}}}

\item
There exists $\epsilon_1\in(0,1)$ such that
$n^{\epsilon_1}(\hat\tau_n^\alpha-\tau_*^\alpha)=o_p(1)$.

\item 
For $k=1,2$, there exists an estimator $\hat\beta_k$ such that  
$\sqrt{T}(\hat\beta_k-\beta_k)=O_p(1)$
under $H_0^{(k)}$. 
\end{enumerate}
\begin{rmk}
Let $\hat\alpha_1$, $\hat\alpha_2$ be estimators of $\alpha_1^*$, $\alpha_2^*$. 
Take into account that 
if $n|\alpha_1^*-\alpha_2^*|^2\lto\infty$, then
$n|\hat\alpha_1-\hat\alpha_2|^2\lto\infty$ in probability,
in other words,
the probability of $n|\hat\alpha_1-\hat\alpha_2|^2<1$ converges to zero. 
According to Theorem \ref{th2_TKU},
we can choose $\epsilon_1\in(0,1)$ in \textbf{[E1]}, 
for example, 
$n^{\epsilon_1}=(n^{0.5}\land n|\hat\alpha_1-\hat\alpha_2|^2)^{0.9}$, i.e., 
$\epsilon_1=0.45\land (0.9+1.8\log_n|\hat\alpha_1-\hat\alpha_2|)$ for large $n$.
See 
Kessler (1997), Uchida and Yoshida (2012,2014),
Kamatani and Uchida (2015) and Kaino and Uchida (2018) 
for constructing an estimator $\hat\beta_k$ 
that satisfies \textbf{[E2]}. 
\end{rmk}
Set
\begin{align*}
\underline{\tau}_n=\hat\tau_n^\alpha-n^{-\epsilon_1},\quad
\overline{\tau}_n=\hat\tau_n^\alpha+n^{-\epsilon_1}.
\end{align*}

We consider $r=d$.
Let
\begin{align*}
\check\xi_{k,i}=1_d^\TT a^{-1}(\Xs,\hat\alpha_k)(\DeX-h b(\Xs,\hat\beta_k))
\end{align*}
for $k=1,2$.
The test statistics for the drift parameter are as follows.
\begin{align*}
\TUO
&=
\frac{1}{\sqrt{d\underline{\tau}_n T}}
\max_{1\le k\le [n\underline{\tau}_n]}
\left|
\sum_{i=1}^k\check\xi_{1,i}
-\frac{k}{[n\underline{\tau}_n]}\sum_{i=1}^{[n\underline{\tau}_n]}\check\xi_{1,i}
\right|,\\
\TOO
&=
\frac{1}{\sqrt{d(1-\overline{\tau}_n) T}}
\max_{1\le k\le n-[n\overline{\tau}_n]}
\left|
\sum_{i=[n\overline{\tau}_n]+1}^{[n\overline{\tau}_n]+k}\check\xi_{2,i}
-\frac{k}{n-[n\overline{\tau}_n]}
\sum_{i=[n\overline{\tau}_n]+1}^{n}\check\xi_{2,i}
\right|.
\end{align*}

\begin{thm}\label{th1}
Suppose that \textbf{[C1]}-\textbf{[C6]}, 
\textbf{[D6]}, \textbf{[E1]} and \textbf{[E2]} 
hold.
Then, 
\begin{align}
\TUO
&\dto
\sup_{0\le s\le 1}|\boldsymbol B_1^0(s)|
\ 
\text{ under } H_0^{(1)},
\label{eq-th1-1}
\\
\TOO
&\dto
\sup_{0\le s\le 1}|\boldsymbol B_1^0(s)|
\ 
\text{ under } H_0^{(2)}.
\label{eq-th1-2}
\end{align}

\end{thm}


$\TUO$ and $\TOO$ are simple test statistics, 
but for the 1-dimensional Ornstein-Uhlenbeck process defined by
$\dd X_t=-\beta(X_t-\gamma)\dd t+\alpha\dd W_t$ 
($\alpha,\beta>0$, $\gamma\in\mathbb R$), 
if $\beta$ changes and $\gamma$ does not change, 
these tests do not satisfy the conditions for the consistency to hold 
as well as the test statistic proposed by Tonaki et al. (2020). 
Therefore, we consider other test statistics.
Let 
\begin{align*}
\check\zeta_{k,i}
&=\partial_\beta b(\Xs,\hat\beta_k)^\TT 
A^{-1}(\Xs,\hat\alpha_k)\bigl(\DeX-hb(\Xs,\hat\beta_k)\bigr),
\\
\mathcal I_{1,n}
&=\frac1{[n\underline{\tau}_n]}\sum_{i=1}^{[n\underline{\tau}_n]}
\partial_\beta b(\Xs,\hat\beta_1)^\TT 
A^{-1}(\Xs,\hat\alpha_1)
\partial_\beta b(\Xs,\hat\beta_1),
\\
\mathcal I_{2,n}
&=\frac1{n-[n\overline{\tau}_n]}\sum_{i=[n\overline{\tau}_n]+1}^n
\partial_\beta b(\Xs,\hat\beta_2)^\TT 
A^{-1}(\Xs,\hat\alpha_2)
\partial_\beta b(\Xs,\hat\beta_2)
\end{align*}
for $k=1,2$. The test statistics for the drift parameter are as follows.
\begin{align*}
\TUT
&=
\frac{1}{\sqrt{\underline{\tau}_n T}}
\max_{1\le k\le [n\underline{\tau}_n]}
\left\|
\mathcal I_{1,n}^{-1/2}
\left(
\sum_{i=1}^k\check\zeta_{1,i}
-\frac{k}{[n\underline{\tau}_n]}\sum_{i=1}^{[n\underline{\tau}_n]}\check\zeta_{1,i}
\right)
\right\|,\\
\TOT
&=
\frac{1}{\sqrt{(1-\overline{\tau}_n)T}}
\max_{1\le k\le n-[n\overline{\tau}_n]}
\left\|
\mathcal I_{2,n}^{-1/2}
\left(
\sum_{i=[n\overline{\tau}_n]+1}^{[n\overline{\tau}_n]+k}\check\zeta_{2,i}
-\frac{k}{n-[n\overline{\tau}_n]}
\sum_{i=[n\overline{\tau}_n]+1}^{n}\check\zeta_{2,i}
\right)
\right\|.
\end{align*}
We additionally make the following assumptions 
with respect to the smoothness of the drift coefficient $b$.
\begin{enumerate}
\renewcommand{\labelenumi}{{\textbf{[E\arabic{enumi}]}}}
\setcounter{enumi}{2}
\item There exists an integer $\mma\ge 3$ such that
$nh^{\mma/(\mma-1)}\lto\infty$ and 
$b\in C^{4,\mma+1}_{\uparrow}(\mathbb R^d\times\Theta_B)$.
\item 
There exists an integer $\mmb \ge3$ such that 
$b\in C_{\uparrow}^{4,\mmb+1}(\mathbb R^d\times\Theta_B)$ and
$\partial_{\beta^{\ell_{\mmb+1}}}\cdots\partial_{\beta^{\ell_1}}b(x,\beta)=0$ for 
$1\le \ell_1,\ldots,\ell_{\mmb+1}\le q$.

\end{enumerate}

\begin{thm}\label{th2}
Suppose that \textbf{[C1]}-\textbf{[C6]}, 
\textbf{[D6]}, 
\textbf{[E1]}-\textbf{[E3]} 
hold.
Then,
\begin{align}
\TUT
&\dto\sup_{0\le s\le 1}\|\boldsymbol B_q^0(s)\|
\ 
\text{ under } H_0^{(1)},
\label{eq-th1-3}
\\
\TOT
&\dto\sup_{0\le s\le 1}\|\boldsymbol B_q^0(s)\|
\ 
\text{ under } H_0^{(2)}.
\label{eq-th1-4}
\end{align}
Furthermore, \eqref{eq-th1-3} and \eqref{eq-th1-4} 
still hold even if we replace \textbf{[E3]} with \textbf{[E4]}.
\end{thm}


\begin{rmk}
When the drift parameter changes at the same point as the diffusion parameter, 
these tests are unable to detect changes in the drift parameter. 
In other words, 
even if the null hypotheses $H_0^{(1)}$ and $H_0^{(2)}$ are not rejected, 
it is possible that
the drift parameter changes 
at the same point as the diffusion parameter.
See Subsection \ref{sec2-3} for the case where 
neither $H_0^{(1)}$ nor $H_0^{(2)}$ is rejected.
\end{rmk}

\subsection{Consistency of tests}\label{sec2-2}
In this subsection, we consider the consistency of the proposed tests. 
For simplicity, we assume that there is a change point 
under the alternative hypothesis, that is, 
we consider the following two hypothesis testing problems.

$H_0^{(1)}: \beta_1$ does not change over $[0, \tau_*^\alpha T]$
\ v.s.\   
$H_1^{(1)}:$ There exists $\tau_*^\beta\in(0,\tau_*^\alpha)$ such that 
\begin{align*}
\beta_1
=
\begin{cases}
\beta_{1,1}^*, &t\in[0,\tau_*^\beta T),\\
\beta_{1,2}^*, &t\in[\tau_*^\beta T,\tau_*^\alpha T],
\end{cases}
\end{align*}

$H_0^{(2)}: \beta_2$ does not change over $[\tau_*^\alpha T,T]$
\ v.s.\   
$H_1^{(2)}:$ There exists $\tau_*^\beta\in(\tau_*^\alpha,1)$ such that 
\begin{align*}
\beta_2
=
\begin{cases}
\beta_{2,1}^*, &t\in[\tau_*^\alpha T,\tau_*^\beta T),\\
\beta_{2,2}^*, &t\in[\tau_*^\beta T,T],
\end{cases}
\end{align*}
where $\beta_{k,1}^*,\beta_{k,2}^*\in\mathrm{Int}\,\Theta_B$, 
$\beta_{k,1}^*\neq\beta_{k,2}^*$ for $k=1,2$.

For $k=1,2$, we make the following assumptions.
\begin{enumerate}
\renewcommand{\labelenumi}{\textbf{[F\arabic{enumi}]}}
\item 
There exist $\beta_k'\in\Theta_B$ and 
an estimator $\hat\beta_k$
such that  
$\hat\beta_k-\beta_k'=o_p(1)$ under $H_1^{(k)}$.

\item 
There exist $\beta_k'\in\Theta_B$ and 
an estimator $\hat\beta_k$
such that  
$\sqrt{T} (\hat\beta_k-\beta_k')=O_p(1)$ under $H_1^{(k)}$.

\end{enumerate}

Let $\alpha_k^*\lto\alpha_k^{(0)}\in\mathrm{Int}\,\Theta_A$ as $n\to\infty$,
where we note that $\alpha_k^*$ may depend on $n$. 
Let
\begin{align*}
\mathcal G_{k,\ell}
&=\int_{\mathbb R^d}1_d^\TT a^{-1}(x,\alpha_k^{(0)})
\bigl(
b(x,\beta_{k,\ell}^*)-b(x,\beta_k')
\bigr)
\dd\mu_{(\alpha_k^{(0)},\beta_{k,\ell}^*)}(x),\\
\mathcal H_{k,\ell}
&=\int_{\mathbb R^d}\partial_\beta b(x,\beta_k')^\TT A^{-1}(x,\alpha_k^{(0)})
\bigl(b(x,\beta_{k,\ell}^*)-b(x,\beta_k')\bigr)
\dd\mu_{(\alpha_k^{(0)},\beta_{k,\ell}^*)}(x).
\end{align*}

\begin{enumerate}
\renewcommand{\labelenumi}{{\textbf{[F\arabic{enumi}]}}}
\setcounter{enumi}{2}

\item 
$\mathcal G_{k,1}\neq\mathcal G_{k,2}$ 
under $H_1^{(k)}$.

\item 
$\mathcal H_{k,1}\neq\mathcal H_{k,2}$ 
under $H_1^{(k)}$. 

\end{enumerate}

\begin{enumerate}
\renewcommand{\labelenumi}{\textbf{[G\arabic{enumi}]}}

\item 
$\Dek=|\beta_{k,1}^*-\beta_{k,2}^*|$
depends on $n$, and $\Dek\lto0$, $T\Dek^2\lto\infty$
as $n\to\infty$ under $H_1^{(k)}$.

\item
There exists $\beta_k^{(0)}\in\mathrm{Int}\,\Theta_B$ such that
$\Dek^{-1}(\beta_{k,\ell}^*-\beta_k^{(0)})\lto d_{k,\ell}\in\mathbb R^q$
as $n\to\infty$ for $\ell=1,2$.

\item 
There exist 
$\beta_k'$ with $\beta_k'-\beta_k^{(0)}=o(1)$ and 
an estimator $\hat\beta_k$
such that  
$\sqrt T(\hat\beta_k-\beta_k')=O_p(1)$ under $H_1^{(k)}$.

\item 
$\displaystyle
\int_{\mathbb R^d}
1_d^\TT a^{-1}(x,\alpha_k^{(0)})\partial_\beta b(x,\beta_k^{(0)})
\dd\mu_{(\alpha_k^{(0)},\beta_k^{(0)})}(x)(d_{k,1}-d_{k,2})
\neq0$
under $H_1^{(k)}$.

\item 
For $\epsilon_1\in(0,1)$ in \textbf{[E1]},
$n^{\epsilon_1}\Dek\lto\infty$.

\item
There exists an integer $\mmc\ge3$ such that 
$n^{-\mmc}h^{-(\mmc+1)}=O(1)$, 
$h^{-1/2}\Dek^{\mmc}\lto0$ 
and $b\in C_{\uparrow}^{4,\mmc+1}(\mathbb R^d\times\Theta_B)$.

\end{enumerate}
\begin{rmk}
See Uchida and Yoshida (2011) and Tonaki et al. (2020,2021) 
for constructing an estimator $\hat\beta_k$ 
that satisfies \textbf{[F1]}, \textbf{[F2]} or \textbf{[G3]}. 
\end{rmk}

\begin{thm}\label{th3}
Suppose that \textbf{[C1]}-\textbf{[C5]}, \textbf{[D6]} and \textbf{[E1]} hold.
If any one of the following conditions is satisfied
\begin{enumerate}
\renewcommand{\labelenumi}{(\alph{enumi})}
\item
\textbf{[F1]} and \textbf{[F3]},
\item  
\textbf{[G1]}-\textbf{[G4]},
\end{enumerate}
then for $\epsilon\in(0,1)$, 
$P\bigl(\TUO>w_1(\epsilon)\bigr)\lto1$ under $H_1^{(1)}$,
and
$P\bigl(\TOO>w_1(\epsilon)\bigr)\lto1$ under $H_1^{(2)}$.
\end{thm}
\begin{thm}\label{th4}
Suppose that \textbf{[C1]}-\textbf{[C5]}, \textbf{[D6]} and \textbf{[E1]} hold.
If any one of the following conditions is satisfied
\begin{enumerate}
\renewcommand{\labelenumi}{(\alph{enumi})}
\item
\textbf{[E4]}, \textbf{[F2]} and \textbf{[F4]},
\item
\textbf{[E5]}, \textbf{[F1]} and \textbf{[F4]},
\item
\textbf{[G1]}-\textbf{[G3]}, \textbf{[G5]} and \textbf{[G6]},
\end{enumerate}
then for $\epsilon\in(0,1)$,
$P\bigl(\TUT>w_q(\epsilon)\bigr)\lto1$ under $H_1^{(1)}$,
and 
$P\bigl(\TOT>w_q(\epsilon)\bigr)\lto1$ under $H_1^{(2)}$.
\end{thm}

\begin{rmk}
(a) of Theorem \ref{th3}, (a) and (b) of Theorem \ref{th4} are the conditions 
to satisfy the consistency when the difference of change does not depend on $n$, 
i.e., Case B described in Section \ref{sec3}, 
and (b) of Theorem \ref{th3} and (c) of Theorem \ref{th4} are the conditions 
to satisfy the consistency when the difference of change depends on $n$ and shrinks, 
i.e., Case A described in Section \ref{sec3}.
\end{rmk}

\subsection{
Change in diffusion and drift parameters at the same point
}\label{sec2-3}

Since $\TUO$ and $\TOO$ (or, $\TUT$ and $\TOT$)
are tests for the change of the drift parameter  
in $[0,\underline{\tau}_nT]$ and $[\overline{\tau}_nT,T]$, respectively, 
neither test can detect the change
when the drift parameter changes in $[\underline{\tau}_nT, \overline{\tau}_nT]$, 
i.e., $\tau_*^\beta=\tau_*^\alpha$.
Therefore, in this subsection, 
we consider how to investigate 
whether the drift parameter changes 
at the same point as the diffusion parameter.
In other words, 
we consider a method for detecting changes in the drift parameter 
for the following stochastic differential equation.
\begin{align*}
X_t
=
\begin{cases}
\displaystyle
X_0+\int_0^t b(X_s,\beta_1^*)\dd s+\int_0^t a(X_s,\alpha_1^*) \dd W_s, 
& t\in[0,\tau_*^\alpha T),\\
\displaystyle
X_{\tau_*^\alpha T}+\int_{\tau_*^\alpha T}^t b(X_s,\beta_2^*)\dd s
+\int_{\tau_*^\alpha T}^t a(X_s,\alpha_2^*) \dd W_s, & t\in[\tau_*^\alpha T,T],
\end{cases}
\end{align*}
where 
$\alpha_1^*,\alpha_2^*\in\mathrm{Int}\,\Theta_A$, 
$\beta_1^*,\beta_2^*\in\mathrm{Int}\,\Theta_B$,
$\alpha_1^*\neq\alpha_2^*$, 
and $\beta_1^*$ and $\beta_2^*$ may be equal or not.

If neither $H_0^{(1)}$ nor $H_0^{(2)}$ is rejected, 
we construct the estimators $\check\beta_1$ and $\check\beta_2$ 
for $\beta_1^*$ and $\beta_2^*$
with data from the intervals $[0,\underline{\tau}_n T]$ and $[\overline{\tau}_n T,T]$,
respectively.
Here, note that the estimators $\check\beta_1$ and $\check\beta_2$ 
can be constructed to satisfy the following.
\begin{align*}
\sqrt{T}(\check\beta_1-\beta_1^*)=O_p(1),
\quad
\sqrt{T}(\check\beta_2-\beta_2^*)=O_p(1).
\end{align*}
Then, we have
\begin{align*}
\sqrt{T}|\beta_1^*-\beta_2^*|
\le
\sqrt{T}|\check\beta_1-\beta_1^*|
+\sqrt{T}|\check\beta_2-\beta_2^*|
+\sqrt{T}|\check\beta_1-\check\beta_2|
=O_p(1)+\sqrt{T}|\check\beta_1-\check\beta_2|
\end{align*}
and
\begin{align*}
\sqrt{T}|\check\beta_1-\check\beta_2|
\le
\sqrt{T}|\check\beta_1-\beta_1^*|
+\sqrt{T}|\check\beta_2-\beta_2^*|
+\sqrt{T}|\beta_1^*-\beta_2^*|
=O_p(1)+\sqrt{T}|\beta_1^*-\beta_2^*|.
\end{align*}
Thus, $\sqrt{T}|\check\beta_1-\check\beta_2|=O_p(1)$ is equivalent to
$\sqrt{T}|\beta_1^*-\beta_2^*|=O(1)$.
Note that
if $\sqrt{T}|\check\beta_1-\check\beta_2|\lto\infty$, 
then $\sqrt{T}|\check\beta_1-\check\beta_2|\neq O_p(1)$,
and
if $\sqrt{T}|\beta_1^*-\beta_2^*|$ is monotone, then
$\sqrt{T}|\beta_1^*-\beta_2^*|\neq O(1)$ 
is equivalent to $\sqrt{T}|\beta_1^*-\beta_2^*|\lto\infty$.
Hence, we have the following assertions.
\begin{enumerate}
\item[(1)]
If $\sqrt{T}|\hat\beta_1-\hat\beta_2|=O_p(1)$, 
then $\sqrt{T}|\beta_1^*-\beta_2^*|= O(1)$,

\item[(2)]
If $\sqrt{T}|\check\beta_1-\check\beta_2|\lto\infty$, 
then $\sqrt{T}|\beta_1^*-\beta_2^*|\lto\infty$.
 
\end{enumerate}
This implies that 
if $\sqrt{T}|\check\beta_1-\check\beta_2|$ is sufficiently large, 
then we infer that the drift parameter changes at $\tau_*^\alpha T$.
Here we note that $\tau_*^\alpha T$ is the same point in time
at which the diffusion parameter changes.


\begin{rmk}
The assertion (2) implies that 
this method can detect any change in the degree that 
the proposed test statistics $\TUO$ and $\TUT$ can detect it.
In other words, the change in the drift parameter  
that satisfies the assumption \textbf{[G1]} 
can be detected by the test $\TUO$ or $\TOO$ 
if the change does not occur at the same time as the diffusion parameter, 
and can also be detected by this method if the change occurs at the same time.
As we saw above, 
we can theoretically determine whether the drift parameter changes 
at the same time as the diffusion parameter 
by investigating $\sqrt{T}|\check\beta_1-\check\beta_2|$,
but it would be difficult to 
determine whether the drift parameter changes simultaneously 
with the diffusion parameter in practice.
See the numerical simulations in Section \ref{sec4}.
\end{rmk}


\section{Change point estimation for drift parameter}\label{sec3}
In this section, 
we consider the change point estimation for the drift parameter 
when there is a change point in the diffusion parameter. 
For simplicity, 
we assume that there is a change point 
in the diffusion and drift parameters, respectively.
Namely,
there exist $\tau_*^\alpha, \tau_*^\beta\in(0,1)$ such that 
\begin{align*}
\alpha^*
=
\begin{cases}
\alpha_1^*, &t\in[0,\tau_*^\alpha T),\\
\alpha_2^*, &t\in[\tau_*^\alpha T, T],
\end{cases}\quad
\beta^*
=
\begin{cases}
\beta_1^*, &t\in[0,\tau_*^\beta T),\\
\beta_2^*, &t\in[\tau_*^\beta T, T],
\end{cases}
\end{align*}
where $\alpha_1^*,\alpha_2^*\in \mathrm{Int}\,\Theta_A$, 
$\beta_1^*,\beta_2^*\in \mathrm{Int}\,\Theta_B$,
$\alpha_1^*\neq\alpha_2^*$, $\beta_1^*\neq\beta_2^*$ and
$\tau_*^\alpha\neq\tau_*^\beta$.

If $\tau_*^\beta<\tau_*^\alpha$, 
then \eqref{sde} can be expressed as follows. 
\begin{align*}
X_t=
\begin{cases}
\displaystyle
X_0+\int_0^t b(X_s,\beta_1^*)\dd s+\int_0^t a(X_s,\alpha_1^*)\dd W_s, 
&t\in[0,\tau_*^\beta T),\\
\displaystyle
X_{\tau_*^\beta T}+\int_{\tau_*^\beta T}^t b(X_s,\beta_2^*)\dd s
+\int_{\tau_*^\beta T}^t a(X_s,\alpha_1^*)\dd W_s,
&t\in[\tau_*^\beta T,\tau_*^\alpha T),\\
\displaystyle
X_{\tau_*^\alpha T}+\int_{\tau_*^\alpha T}^t b(X_s,\beta_2^*)\dd s
+\int_{\tau_*^\alpha T}^t a(X_s,\alpha_2^*)\dd W_s,
&t\in[\tau_*^\alpha T, T].
\end{cases}
\end{align*}
On the other hand, if $\tau_*^\alpha<\tau_*^\beta$, 
then \eqref{sde} can be expressed as 
\begin{align*}
X_t=
\begin{cases}
\displaystyle
X_0+\int_0^t b(X_s,\beta_1^*)\dd s+\int_0^t a(X_s,\alpha_1^*)\dd W_s, 
&t\in[0,\tau_*^\alpha T),\\
\displaystyle
X_{\tau_*^\alpha T}+\int_{\tau_*^\alpha T}^t b(X_s,\beta_1^*)\dd s
+\int_{\tau_*^\alpha T}^t a(X_s,\alpha_2^*)\dd W_s,
&t\in[\tau_*^\alpha T,\tau_*^\beta T),\\
\displaystyle
X_{\tau_*^\beta T}+\int_{\tau_*^\beta T}^t b(X_s,\beta_2^*)\dd s
+\int_{\tau_*^\beta T}^t a(X_s,\alpha_2^*)\dd W_s,
&t\in[\tau_*^\beta T, T].
\end{cases}
\end{align*}

We consider the asymptotic properties of the proposed estimators, 
such as consistency and asymptotic distribution, in the following two cases.

\begin{description}
\item[Case A]
The parameters $\beta_1^*$ and $\beta_2^*$ depend on $n$, 

\item[Case B]
The parameters $\beta_1^*$ and $\beta_2^*$ are fixed and not depend on $n$. 
\end{description}
We make the following assumptions.
\begin{enumerate}
\renewcommand{\labelenumi}{\textbf{[H\arabic{enumi}]}}

\item
The test $\TUO$ or $\TUT$ detects
a change in the drift parameter. 

\item
The test $\TOO$ or $\TOT$ detects
a change in the drift parameter. 
 
\item 
There exist estimators 
$\hat\beta_k$ ($k=1,2$)
such that $\sqrt{T}(\hat\beta_k-\beta_k^*)=O_p(1)$.

\end{enumerate}

\begin{enumerate}
\renewcommand{\labelenumi}{\textbf{[A\arabic{enumi}]}}

\item
$\Deb=|\beta_1^*-\beta_2^*|$ depends on $n$, and 
\begin{align*}
\Deb\lto0,\quad
T\Deb^2\lto\infty
\end{align*}
as $n\to\infty$.
Furthermore, $\Deb^{-1}(\beta_k^*-\beta_0)=O(1)$ holds 
for some $\beta_0\in\mathrm{Int}\, \Theta_B$ and $k=1,2$.

\item There exists an integer $\mmd\ge 3$ such that 
$nh^{\mmd/(\mmd-1)}\lto\infty$, $h^{-1/2}\Deb^{\mmd-1}\lto0$ 
and
$b\in C_{\uparrow}^{4,\mmd+1}(\mathbb R^d\times\Theta_B)$.
\end{enumerate}
Let
$\alpha_k^*\lto\alpha_k^{(0)}\in\mathrm{Int}\,\Theta_A$ as $n\to\infty$,
and
\begin{align*}
\Gamma^\beta(x,\alpha,\beta_1,\beta_2)
= 
\tr
\bigl[
A^{-1}(x,\alpha)(b(x,\beta_1)-b(x,\beta_2))^{\otimes2}
\bigr].
\end{align*}
\begin{enumerate}
\renewcommand{\labelenumi}{\textbf{[B\arabic{enumi}]}}
\item
$\displaystyle
\inf_x\Gamma^\beta(x,\alpha_1^{(0)},\beta_1^*,\beta_2^*)>0$.

\item
$\displaystyle
\inf_x\Gamma^\beta(x,\alpha_2^{(0)},\beta_1^*,\beta_2^*)>0$.

\item
There exists a constant $C>0$ such that

\begin{enumerate}
\item

$\displaystyle\sup_{x,\alpha,\beta_k}
\left|
\partial_{(\alpha,\beta_1,\beta_2)}
\Gamma^\beta(x,\alpha,\beta_1,\beta_2)
\right|
<C$, 
\item 
$\displaystyle
\sup_{x,\alpha,\beta_k}
\left|
\Bigl[
\tr
\Bigl(
A^{-1}(x,\alpha)
\partial_{\beta^{\ell}}b(x,\beta_1)
\bigl(
b(x,\beta_2)-b(x,\beta_3)
\bigr)^\TT
\Bigr)
\Bigr]_{\ell=1}^q
\right|<C$. 
\end{enumerate}
\end{enumerate}

\begin{rmk}
See Section 3 in Tonaki et al. (2021) 
for the construction of estimators that satisfy \textbf{[H3]}.
Tonaki et al. (2021)
assumed $T\Deb^4\lto0$ 
in addition to condition \textbf{[A1]}, 
but there is no need to assume it under \textbf{[C6]}.
\end{rmk}

We consider $\tau_*^\beta<\tau_*^\alpha$.
Suppose that \textbf{[H1]} holds. 
Then, we set
\begin{align*}
&G_i(\beta|\alpha)
=
\tr\left(A^{-1}(\Xs,\alpha)
\frac{(\DeX-hb(\Xs,\beta))^{\otimes 2}}{h}\right),
\\
&\Psi_{1,n}(\tau:\beta_1,\beta_2|\alpha)
=
\sum_{i=1}^{[n\tau]}G_i(\beta_1|\alpha)
+\sum_{i=[n\tau]+1}^{[n\underline{\tau}_n]}G_i(\beta_2|\alpha)
\end{align*}
and propose
\begin{align*}
\hat\tau_{1,n}^\beta
=\underset{\tau\in[0,\underline{\tau}_n]}{\mathrm{argmin}}\,
\Psi_{1,n}(\tau:\hat\beta_1,\hat\beta_2|\hat\alpha_1)
\end{align*}
as an estimator of $\tau_*^\beta$.

We consider $\tau_*^\alpha<\tau_*^\beta$.
Suppose that \textbf{[H2]} holds. 
Then, we set
\begin{align*}
\Psi_{2,n}(\tau:\beta_1,\beta_2|\alpha)
=
\sum_{i=[n\overline{\tau}_n]+1}^{[n\tau]}G_i(\beta_1|\alpha)
+\sum_{i=[n\tau]+1}^n G_i(\beta_2|\alpha)
\end{align*}
and propose
\begin{align*}
\hat\tau_{2,n}^\beta
=\underset{\tau\in[\overline{\tau}_n,1]}{\mathrm{argmin}}\,
\Psi_{2,n}(\tau:\hat\beta_1,\hat\beta_2|\hat\alpha_2)
\end{align*}
as an estimator of $\tau_*^\beta$.

In Case A, we set for $k=1,2$ and $v\in\mathbb R$ , 
\begin{align*}
e_\beta
&=\lim_{n\to\infty}\Deb^{-1}(\beta_1^*-\beta_2^*),\\
\Xi^\beta(x,\alpha,\beta)
&=
\Bigl[
\partial_{\beta^{\ell_1}}b(x,\beta)^\TT
A^{-1}(x,\alpha)\partial_{\beta^\ell_2}b(x,\beta)
\Bigr]_{\ell_1,\ell_2=1}^q,\\
\mathcal J_{k,\beta}
&=e_\beta^\TT 
\int_{\mathbb R^d}
\Xi^\beta(x,\alpha_k^{(0)},\beta_0)
\dd\mu_{(\alpha_k^{(0)},\beta_0)}(x)
e_\beta,\\
\mathbb G_k(v)
&=-2\mathcal J_{k,\beta}^{1/2}\mathbb W(v)+\mathcal J_{k,\beta}|v|.
\end{align*}
\begin{thm}\label{th5}
Let $\tau_*^\beta<\tau_*^\alpha$.
Suppose that \textbf{[C1]}-\textbf{[C6]}, 
\textbf{[D6]}, \textbf{[E1]}, \textbf{[H1]} and \textbf{[H3]} hold.
\begin{enumerate}

\item[(1)]
Under  \textbf{[A1]} and \textbf{[A2]}, in Case A, 
\begin{align*}
T\Deb^2(\hat\tau_{1,n}^\beta-\tau_*^\beta)
\dto
\underset{v\in\mathbb R}{\mathrm{argmin}}\,
\mathbb G_1(v).
\end{align*}

\item[(2)]
Under \textbf{[B1]} and \textbf{[B3]}, in Case B,
\begin{align*}
T(\hat\tau_{1,n}^\beta-\tau_*^\beta)=O_p(1).
\end{align*}

\end{enumerate}
\end{thm}
\begin{thm}\label{th6}
Let $\tau_*^\alpha<\tau_*^\beta$.
Suppose that \textbf{[C1]}-\textbf{[C6]}, 
\textbf{[D6]}, \textbf{[E1]}, \textbf{[H2]} and \textbf{[H3]} hold.
\begin{enumerate}

\item[(1)]
Under \textbf{[A1]} and \textbf{[A2]}, in Case A, 
\begin{align*}
T\Deb^2(\hat\tau_{2,n}^\beta-\tau_*^\beta)
\dto
\underset{v\in\mathbb R}{\mathrm{argmin}}\,
\mathbb G_2(v).
\end{align*}

\item[(2)]
Under \textbf{[B2]} and \textbf{[B3]}, in Case B,
\begin{align*}
T(\hat\tau_{2,n}^\beta-\tau_*^\beta)=O_p(1).
\end{align*}

\end{enumerate}
\end{thm}

\begin{rmk}\label{rmk11}
In the case of the 1-dimensional Ornstein-Uhlenbeck process defined by
$\dd X_t=-\beta(X_t-\gamma)\dd t+\alpha\dd W_t$ 
($\alpha,\beta>0$, $\gamma\in\mathbb R$), 
\textbf{[B1]}-\textbf{[B3]} are satisfied 
if $\beta$ does not change and $\gamma$ changes.
Therefore, Theorems \ref{th5} and \ref{th6} (2) 
enable us to estimate the change point of the drift parameter in Case B. 
Moreover, in the case of the hyperbolic diffusion model defined by
$\dd X_t=(\beta-\gamma X_t/\sqrt{1+X_t^2})\dd t+\alpha\dd W_t$ 
($\alpha>0$, $\beta\in\mathbb R$, $|\beta|<\gamma$), 
since \textbf{[B1]}-\textbf{[B3]} hold, 
we can estimate 
the change point in time of the drift parameter in Case B.
On the other hand, in Case A, 
the change point of the drift parameter can be estimated 
in both cases of the 1-dimensional Ornstein-Uhlenbeck process 
and the hyperbolic diffusion model.
\end{rmk}

\section{Examples and simulation results}\label{sec4}
We consider the following stochastic differential equation 
with a change point in the diffusion parameter.
\begin{align*}
X_t
=
\begin{cases}
\displaystyle
X_0+\int_0^t b(X_s,\beta)\dd s+\int_0^t a(X_s,\alpha_1^*) \dd W_s, 
& t\in[0,\tau_*^\alpha T),\\
\displaystyle
X_{\tau_*^\alpha T}+\int_{\tau_*^\alpha T}^t b(X_s,\beta)\dd s
+\int_{\tau_*^\alpha T}^t a(X_s,\alpha_2^*) \dd W_s, & t\in[\tau_*^\alpha T,T],
\end{cases}
\end{align*}
where $\alpha_1^*\neq\alpha_2^*$,
and
$\beta$ may change in $[0,T]$. 

In this section, 
we consider the following three situations 
and confirm the results of Sections \ref{sec2} and \ref{sec3} 
by numerical simulations.

\begin{enumerate}

\item[(i)]
The drift parameter $\beta$ does not change over $[0,T]$.

\item[(ii)]
The drift parameter $\beta$ changes from $\beta_1^*$ to $\beta_2^*$ 
at $\tau_*^\beta T$,
where $\tau_*^\beta\in(0,1)$, $\tau_*^\beta\neq\tau_*^\alpha$.

\item[(iii)]
The drift parameter $\beta$ changes from $\beta_1^*$ to $\beta_2^*$ 
at the same point as the diffusion parameter.

\end{enumerate}
We perform numerical simulations with the following procedures.

\begin{description}

\item[All situations] 
Perform the following six steps in all situations (i)-(iii).

\begin{enumerate}
\item
Test for a change in the diffusion parameter in $[0,T]$ 
to check whether there is a change point or not.
See Subsection \ref{secR-1} for the test statistic.

\item
If a change is detected in (1), 
choose $\tau_1,\tau_2\in(0,1)$ such that 
the test detects the change in $[\tau_1T,\tau_2T]$ 
to estimate $\alpha_1^*$ and $\alpha_2^*$.
See Section 3 of Tonaki et al. (2021).

\item
Construct estimators $\hat\alpha_1$ and $\hat\alpha_2$ 
for $\alpha_1^*$ and $\alpha_2^*$ 
from $[0,\tau_1T]$ and $[\tau_2T,T]$, respectively.

\item
Estimate $\tau_*^\alpha$ 
with the estimators $\hat\alpha_1$ and $\hat\alpha_2$ of (3).
Let $\hat\tau_n^\alpha$ be the estimator of $\tau_*^\alpha$.
See Subsection \ref{secR-2} for the estimator of $\tau_*^\alpha$.

\item
Choose $\epsilon_1>0$ such that 
$n^{\epsilon_1}=(n^{0.5}\land n|\hat\alpha_1-\hat\alpha_2|^2)^{0.9}$,  
and 
set ${\underline\tau}_n=\hat\tau_n^\alpha-n^{-\epsilon_1}$ 
and ${\overline\tau}_n=\hat\tau_n^\alpha+n^{-\epsilon_1}$.

\item
Test for a change in the drift parameter 
in $[0,{\underline\tau}_n T]$ and $[{\overline\tau}_n T,T]$
to check whether there is a change point or not.
See Subsection \ref{sec2-1} for the test statistics.

\end{enumerate}

\item[Situations (i) and (iii)]
After (1)-(6), perform the following step in situations (i) and (iii).

\begin{enumerate}
\item[(7)]
If neither of the tests in (6) detects a change, 
construct estimators $\check\beta_1$ and $\check\beta_2$ 
from $[0,{\underline\tau}_n T]$ and $[{\overline\tau}_n T,T]$, respectively, 
and investigate $\sqrt T|\check\beta_1-\check\beta_2|$
to check whether the drift parameter changes at the same point 
as the diffusion parameter. See Subsection \ref{sec2-3}.

\end{enumerate}

\item[Situation (ii)]
After (1)-(6), perform the following step in situation (ii).

\begin{enumerate}

\item[(7)]
If a change is detected in (6), estimate $\tau_*^\beta$. 
Note that the estimators $\hat\beta_1$ and $\hat\beta_2$ 
for $\beta_1^*$ and $\beta_2^*$ are constructed in the same way as (2)-(3).
See Section \ref{sec3} for the estimator of $\tau_*^\beta$.

\end{enumerate}

\end{description}

All simulations are conducted at significance level $0.05$ and 
the corresponding critical values are  
obtained from the following: the Brownian bridge is 
generated by taking $10^4$ points 
on the interval $[0,1]$, 
and the maximum value of its norm is recorded. 
This is repeated $10^4$ times. As a result, we have
\begin{align*}
P\left(\sup_{0\le s\le 1}| { \boldsymbol B_1^0(s)} |> 1.3617\right)=0.05,\quad
P\left(\sup_{0\le s\le 1}\| { \boldsymbol B_2^0(s)} \|> 1.5736\right)=0.05,
\end{align*}
i.e., the corresponding critical values are $1.3617$ and $1.5736$ for 
the 
1-dimensional and 2-dimensional Brownian bridges, respectively.

\subsection{Model 1}
We consider the $1$-dimensional Ornstein-Uhlenbeck process defined by 
\begin{align*}
\dd X_t=-\beta(X_t-\gamma)\dd t+\alpha\dd W_t,\quad X_0=x_0,
\end{align*}
where $\alpha,\beta>0$ and $\gamma\in\mathbb R$.
For simulations of the test statistics and the estimator, 
we study the following stochastic differential equation
\begin{align*}
X_t
=
\begin{cases}
\displaystyle
X_0-\int_0^t \beta(X_s-\gamma)\dd s+\alpha_1^* W_t, 
& t\in[0,\tau_*^\alpha T),\\
\displaystyle
X_{\tau_*^\alpha T}-\int_{\tau_*^\alpha T}^t \beta(X_s-\gamma)\dd s
+\alpha_2^* (W_t-W_{\tau_*^\alpha T}), & t\in[\tau_*^\alpha T,T],
\end{cases}
\end{align*}
where $x_0=2$, $\tau_*^\alpha=0.8$, $\alpha_1^*=1$, $\alpha_2^*=1.2$, 
$\boldsymbol\beta=(\beta,\gamma)$.

We consider the following three situations.
\begin{enumerate}
\item[(i)]
The drift parameter $\boldsymbol\beta=(1,2)$
does not change over $[0,T]$.

\item[(ii)]
The drift parameter $\boldsymbol\beta$ changes 
from $\boldsymbol\beta_1^*=(1,2-\Deb)$ 
to $\boldsymbol\beta_2^*=(1,2)$ 
at $\tau_*^{\boldsymbol\beta}=0.4$ (Case A). 

\item[(iii)]
The drift parameter $\boldsymbol\beta$ changes 
from $\boldsymbol\beta_1^*=(1,2-\Deb)$ to $\boldsymbol\beta_2^*=(1,2)$ 
at $\tau_*^{\boldsymbol\beta}=\tau_*^\alpha=0.8$.

\end{enumerate}
The number of iteration is 1000.
We set that the sample size of the data $\{\Xt\}_{i=0}^n$ is $n=10^6$ or $10^7$, 
$h=n^{-0.52}$, $T=nh=n^{0.48}$, $nh^2=n^{-0.04}$,
$\Deb=n^{-0.1}$.

We test for changes in the diffusion parameter in the interval $[0,T]$. 
The results show that in all situations (i)-(iii), 
the change is detected in all $1000$ iterations.
In order to estimate the parameters before and after the change, 
we test for the change in the diffusion parameter in the interval $[0.125T,0.875T]$. 
The results indicate that the change is detected in all $1000$ iterations 
for all situations (i)-(iii). 
Therefore, we estimate $\alpha_1^*$ from $[0,0.125T]$ and 
$\alpha_2^*$ from $[0.875T,T]$,
and estimate $\tau_*^\alpha$ using the estimators $\hat\alpha_1$ and $\hat\alpha_2$.
The estimates of $\alpha_1^*$, $\alpha_2^*$ and $\tau_*^\alpha$ 
are reported in Table \ref{tab1}.
In this case, 
we chose $\epsilon_1=0.45$ for all iterations. 

\begin{table}[h]
\captionsetup{margin=5pt}
\caption{
Mean and standard deviation of the estimators. 
True values: 
$\alpha_1^*=1$, $\alpha_2^*=1.2$, $\tau_*^\alpha=0.8$.
}
\begin{center}
\begin{tabular*}{1.0\textwidth}{@{\extracolsep{\fill}}ccccccc}\hline
$n$ & $T$ & $h$ && $\hat\alpha_1$ & $\hat\alpha_2$ & $\hat\tau_n^\alpha$ 
\rule[0mm]{0cm}{4mm}\\\hline 
$10^6$ & 758.58 & $7.59\times10^{-4}$ & 
(i)& 1.00018 & 1.20018 & 0.79878    \rule[0mm]{0cm}{4mm}\\ 
&&&& (0.00207) & (0.00235) & (0.00016)  \\
&&& 
(ii)& 1.00018 & 1.20018 & 0.79879    \rule[0mm]{0cm}{4mm}\\ 
&&&& (0.00207) & (0.00235) & (0.00015)  \\
&&& 
(iii)& 1.00018 & 1.20018 & 0.79879    \rule[0mm]{0cm}{4mm}\\ 
&&&& (0.00207) & (0.00235) & (0.00015)  \\\hline
$10^7$ & 2290.87 & $2.29\times10^{-4}$ &
(i)& 1.00005 & 1.20010 & 0.79961   \rule[0mm]{0cm}{4mm}\\ 
&&&& (0.00062) & (0.00078) & (0.00006)  \\
&&& 
(ii)& 1.00005 & 1.20010 & 0.79961    \rule[0mm]{0cm}{4mm}\\ 
&&&& (0.00062) & (0.00078) & (0.00006)  \\
&&& 
(iii)& 1.00005 & 1.20010 & 0.79961    \rule[0mm]{0cm}{4mm}\\ 
&&&& (0.00062) & (0.00078) & (0.00006)  \\\hline
\end{tabular*}
\label{tab1}
\end{center}
\end{table}

Next, we test for changes in the drift parameter in the intervals 
$[0,\underline{\tau}_n T]$ and $[\overline{\tau}_n T,T]$. 
The results of the tests for changes in the drift parameter 
are summarized in Table \ref{tab2} 
and Figures \ref{model1i_test}, \ref{model1ii_test} and \ref{model1iii_test}.
From Table \ref{tab2}, Figures \ref{model1i_test} and \ref{model1iii_test},
we can see that in (i) and (iii), 
the proportions of the test statistics that 
exceed the critical values are close to the significance level $0.05$, 
and the distribution of the test statistics almost correspond with 
the null distribution. 
This implies that the test statistics have good performance.
Therefore, in (i) and (iii), we construct estimators 
$\check{\boldsymbol\beta}_1=(\check\beta_1,\check\gamma_1)$ 
and $\check{\boldsymbol\beta}_2=(\check\beta_2,\check\gamma_2)$ 
from the intervals $[0,\underline{\tau}_n T]$ and $[\overline{\tau}_n T,T]$,
respectively 
when the test statistics $\TUO$ and $\TOO$ do not exceed the critical value, 
and investigate 
$\sqrt T\|\check{\boldsymbol\beta}_1-\check{\boldsymbol\beta}_2\|$. 
The results of the estimates of $\boldsymbol\beta_1^*$ and $\boldsymbol\beta_2^*$ 
in (i) and (iii) are summarized in 
Table \ref{tab3} and Figure \ref{model1i_test_M},
and Table \ref{tab5} and Figure \ref{model1iii_test_M}, respectively. 
From Figure \ref{model1i_test_M},
$\sqrt T\|\check{\boldsymbol\beta}_1-\check{\boldsymbol\beta}_2\|$ 
does not diverge when increasing from $n=10^6$ to $n=10^7$ in (i).
In this case, it appears that
$\sqrt T\|\check{\boldsymbol\beta}_1-\check{\boldsymbol\beta}_2\|$
is bounded in probability.
Meanwhile,
from Figure \ref{model1iii_test_M}, 
$\sqrt T\|\check{\boldsymbol\beta}_1-\check{\boldsymbol\beta}_2\|$ 
diverges when increasing from $n=10^6$ to $n=10^7$ in (iii).
According to Subsection \ref{sec2-3}, 
we can infer that the drift parameter changes at the same time 
as the diffusion parameter in (iii).
However, as we can see by comparing 
Figures \ref{model1i_test_M} (B) and \ref{model1iii_test_M} (B), 
it would be difficult to determine whether 
the drift parameter changes at the same point as the diffusion parameter 
when $n=10^7$. 
In this case, it would be possible to determine 
whether there is a change when $n=10^9$, but this is not realistic.

\begin{table}[h]
\captionsetup{margin=5pt}
\caption{
Proportions over the corresponding critical value.
}
\begin{center}
\begin{tabular*}{1.0\textwidth}{@{\extracolsep{\fill}}cccccccc}\hline
$n$ & $T$ & $h$ && $\TUO$ & $\TUT$ & $\TOO$ & $\TOT$ 
\rule[0mm]{0cm}{5mm}\\\hline 
$10^6$ & 758.58 & $7.59\times10^{-4}$ 
& (i)
& 0.040 & 0.034 & 0.045 & 0.048   \rule[0mm]{0cm}{4mm}\\ 
&&& (ii)
& 0.784 & 0.704 & 0.045 & 0.046    \rule[0mm]{0cm}{4mm}\\ 
&&& (iii)
& 0.040 & 0.034 & 0.046 & 0.048    \rule[0mm]{0cm}{4mm}\\\hline 
$10^7$ & 2290.87 & $2.29\times10^{-4}$ 
& (i)
& 0.048 & 0.053 & 0.040 & 0.046    \rule[0mm]{0cm}{4mm}\\ 
&&& (ii)
& 0.981 & 0.944 & 0.040 & 0.045    \rule[0mm]{0cm}{4mm}\\ 
&&& (iii)
& 0.048 & 0.053 & 0.040 & 0.046    \rule[0mm]{0cm}{4mm}\\\hline
\end{tabular*}
\label{tab2}
\end{center}
\end{table}

\begin{figure}[h]
\captionsetup{margin=5pt}
\captionsetup[sub]{margin=5pt}
 \begin{minipage}[t]{0.46\linewidth}
  \centering
  \includegraphics[keepaspectratio, width=67mm]{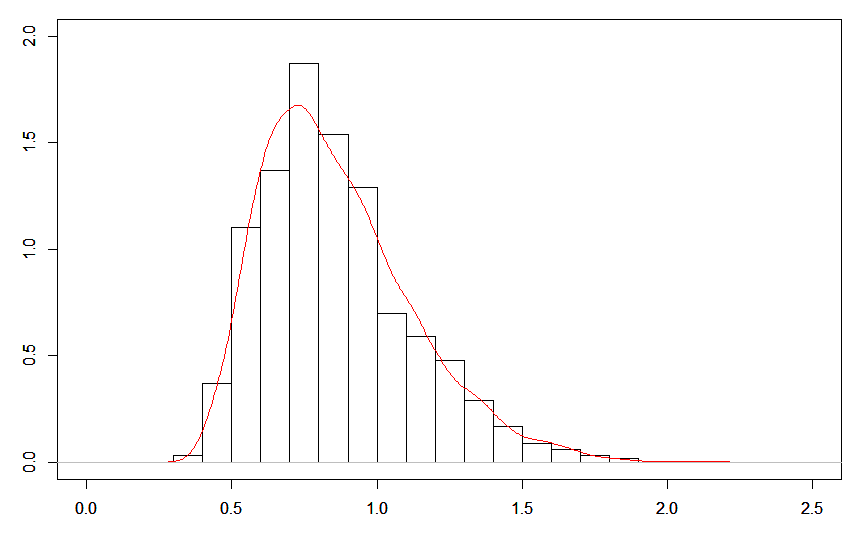}
  \subcaption{
  		Histogram of $\TUO$ with $n=10^7$.
  }
 \end{minipage}
 \begin{minipage}[t]{0.46\linewidth}
  \centering
  \includegraphics[keepaspectratio, width=67mm]{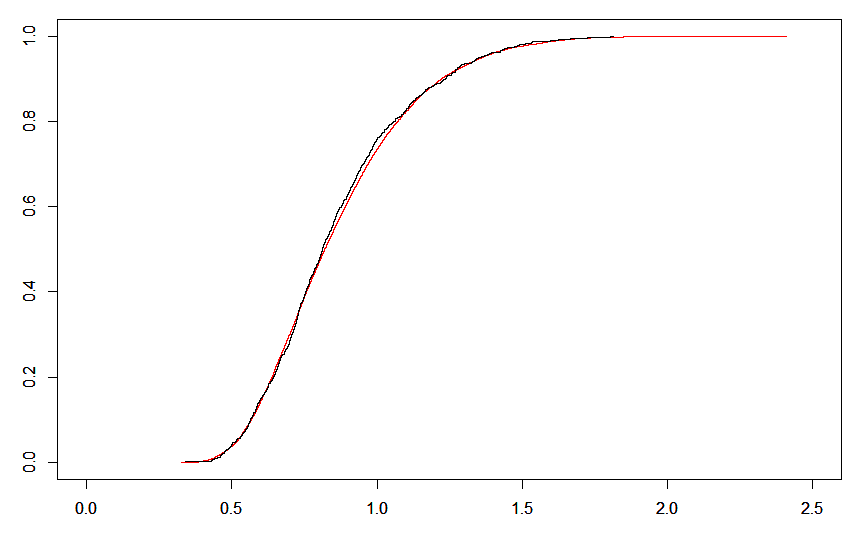}
  \subcaption{
  		EDF of $\TUO$ with $n=10^7$. 
  }
 \end{minipage}\\
 \begin{minipage}[t]{0.46\linewidth}
  \centering
  \includegraphics[keepaspectratio, width=67mm]{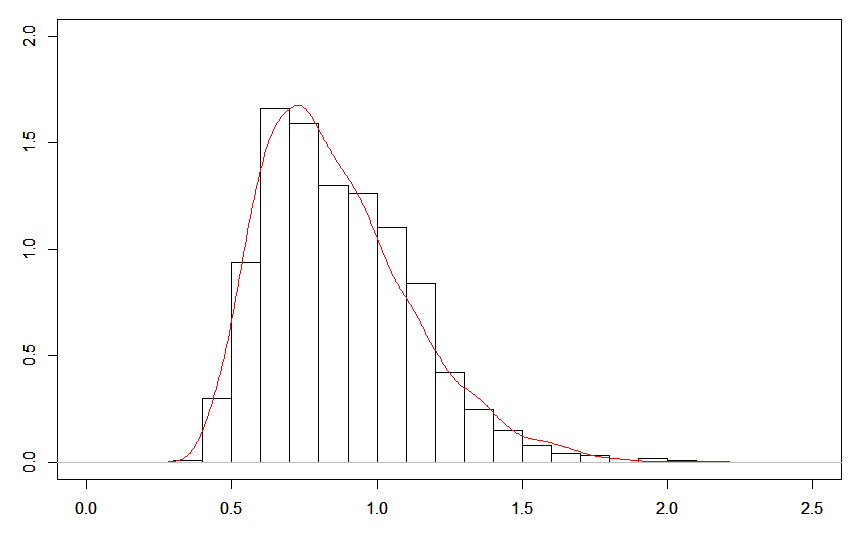}
  \subcaption{
  		Histogram of $\TOO$ with $n=10^7$.
  }
 \end{minipage}
 \begin{minipage}[t]{0.46\linewidth}
  \centering
  \includegraphics[keepaspectratio, width=67mm]{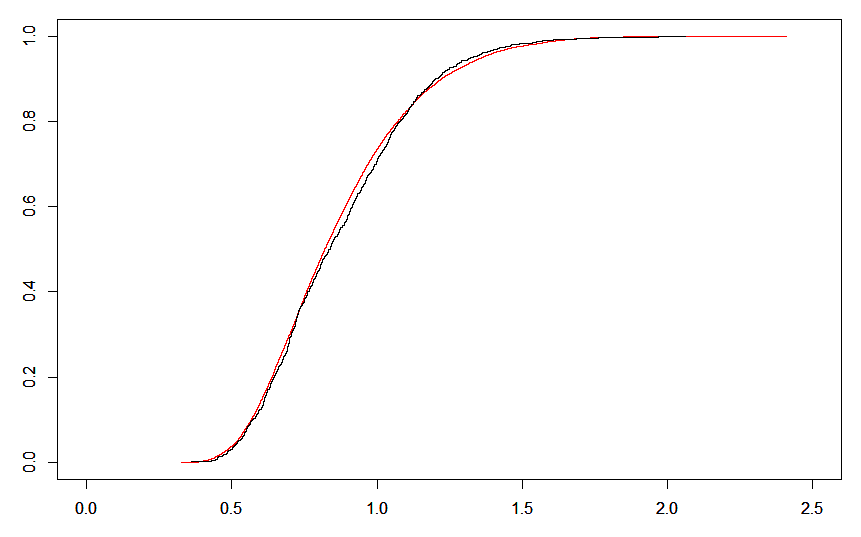}
  \subcaption{
  		EDF of $\TOO$ with $n=10^7$. 
  }
 \end{minipage}
 \caption{
 		Histogram (black line) versus theoretical density function (red line) 
 		and
 		empirical distribution function (black line) versus 
 		theoretical distribution function (red line) in (i).
 }
 \label{model1i_test}
\end{figure}

\begin{table}[h]
\captionsetup{margin=5pt}
\caption{
Mean and standard deviation of the estimators in (i). 
True values: $\beta^*=1$, $\gamma^*=2$.
}
\begin{center}
\begin{tabular*}{1.0\textwidth}{@{\extracolsep{\fill}}ccccccc}\hline
$n$ & $T$ & $h$ & $\check\beta_1$ & $\check\gamma_1$ & $\check\beta_2$ & 
$\check\gamma_2$ 
\rule[0mm]{0cm}{4mm}\\\hline 
$10^6$ & 758.58 & $7.59\times10^{-4}$ 
& 1.00700 & 1.99850 & 1.02695 & 1.99989    \rule[0mm]{0cm}{4mm}\\ 
&&& (0.05871) & (0.04084) & (0.12165) & (0.09591)  \\
$10^7$ & 2290.87 & $2.29\times10^{-4}$ 
& 1.00211 & 2.00024 & 1.00837 & 1.99797    \rule[0mm]{0cm}{4mm}\\ 
&&& (0.03321) & (0.02422) & (0.06541) & (0.05436)  \\\hline
\end{tabular*}
\label{tab3}
\end{center}
\end{table}

\begin{figure}[h]
\captionsetup{margin=5pt}
\captionsetup[sub]{margin=5pt}
 \begin{minipage}[t]{0.46\linewidth}
  \centering
  \includegraphics[keepaspectratio, width=67mm]
  {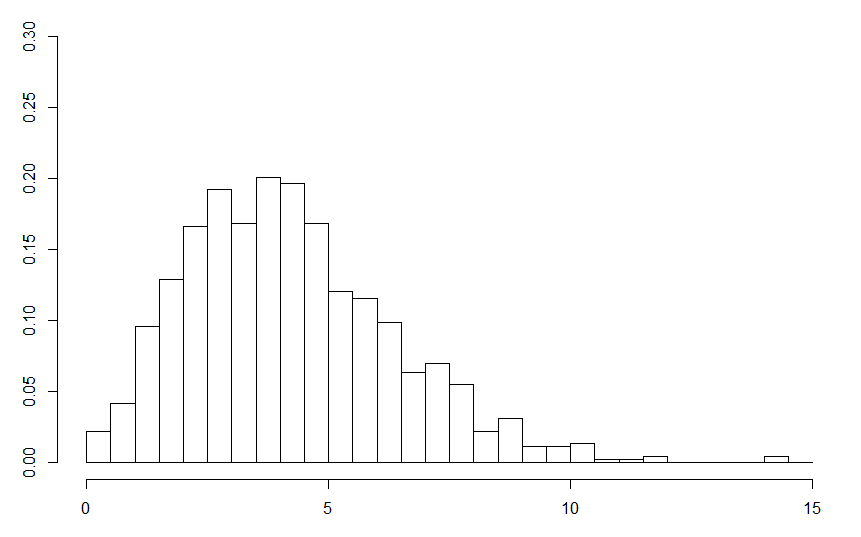}
  \subcaption{
  		$n=10^6$. 
  }
 \end{minipage}
 \begin{minipage}[t]{0.46\linewidth}
  \centering
  \includegraphics[keepaspectratio, width=67mm]
  {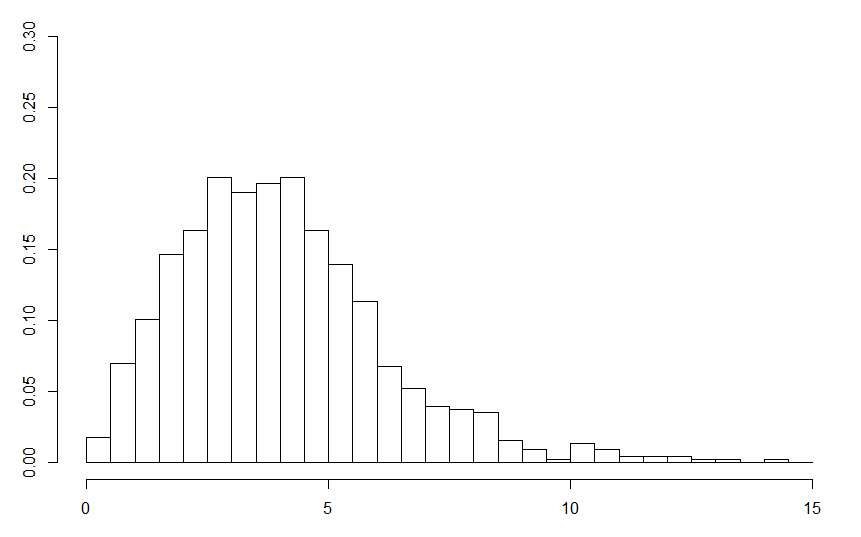}
  \subcaption{
  		$n=10^7$. 
  }
 \end{minipage}
 \caption{
  		Histogram of 
  		$\sqrt T\|\check{\boldsymbol\beta}_1-\check{\boldsymbol\beta}_2\|$  
 		in (i).
 }
 \label{model1i_test_M}
\end{figure}

\begin{figure}[h]
\captionsetup{margin=5pt}
\captionsetup[sub]{margin=5pt}
 \begin{minipage}[t]{0.46\linewidth}
  \centering
  \includegraphics[keepaspectratio, width=67mm]
  {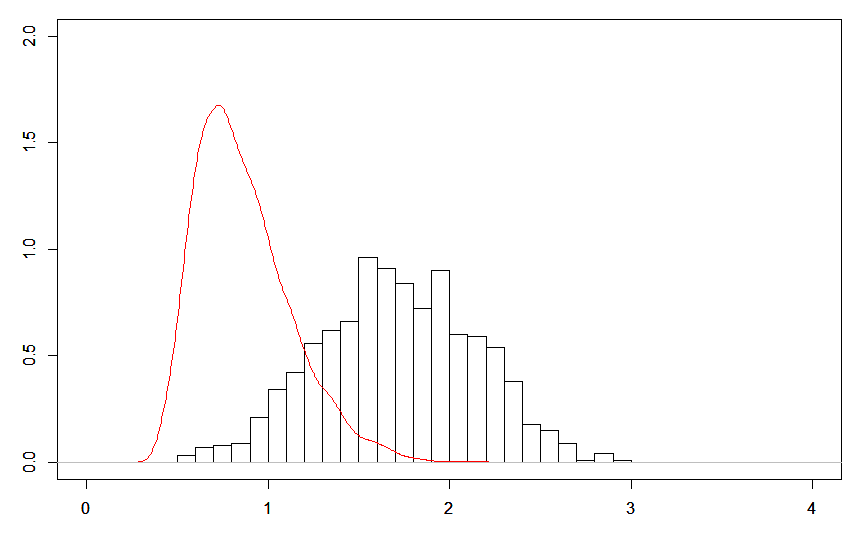}
  \subcaption{
  		Histogram of $\TUO$ with $n=10^6$.
  }
 \end{minipage}
 \begin{minipage}[t]{0.46\linewidth}
  \centering
  \includegraphics[keepaspectratio, width=67mm]
  {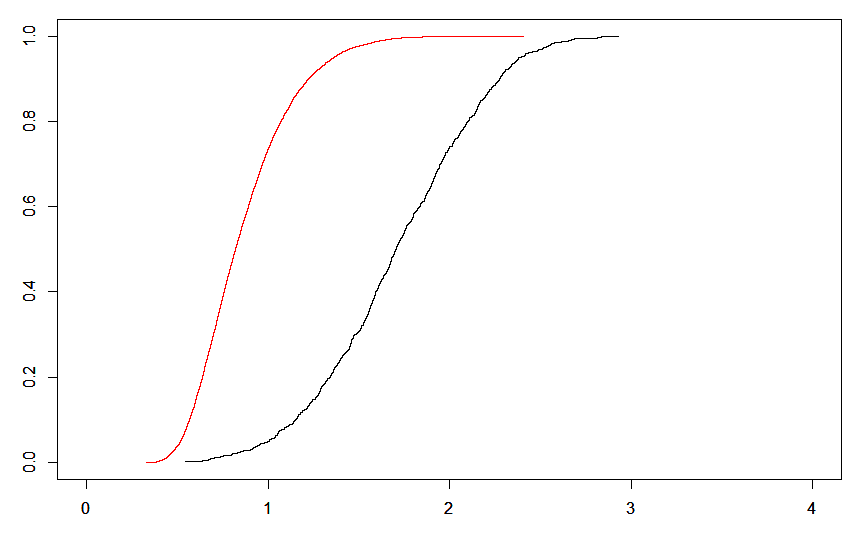}
  \subcaption{
  		EDF of $\TUO$ with $n=10^6$. 
  }
 \end{minipage}\\
 \begin{minipage}[t]{0.46\linewidth}
  \centering
  \includegraphics[keepaspectratio, width=67mm]
  {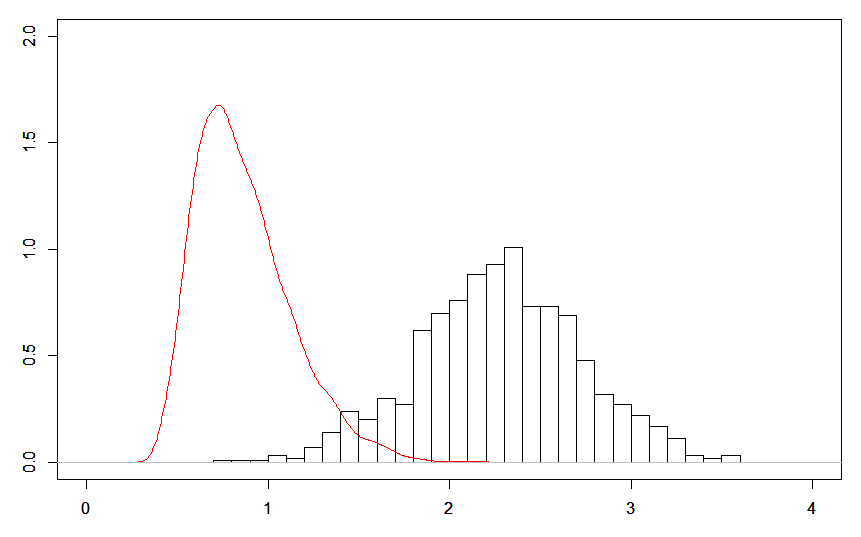}
  \subcaption{
  		Histogram of $\TUO$ with $n=10^7$.
  }
 \end{minipage}
 \begin{minipage}[t]{0.46\linewidth}
  \centering
  \includegraphics[keepaspectratio, width=67mm]
  {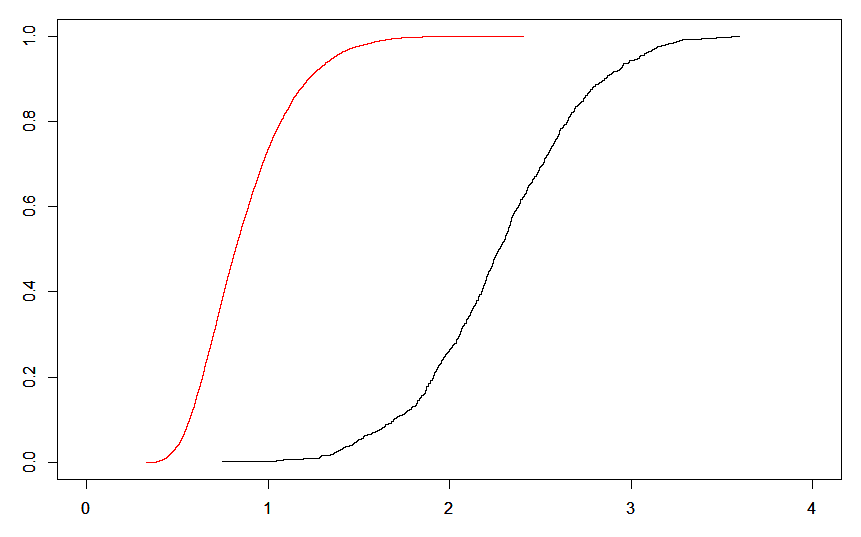}
  \subcaption{
  		EDF of $\TUO$ with $n=10^7$. 
  }
 \end{minipage}\\
 \begin{minipage}[t]{0.46\linewidth}
  \centering
  \includegraphics[keepaspectratio, width=67mm]
  {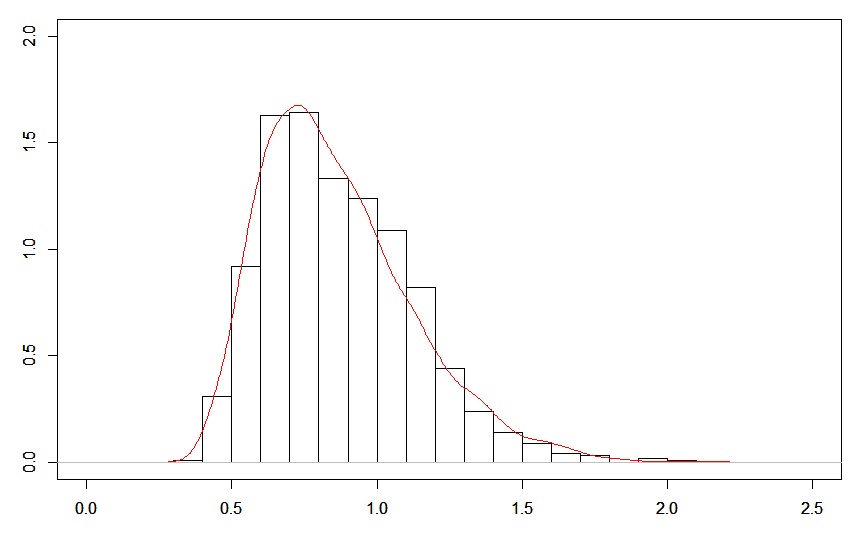}
  \subcaption{
  		Histogram of $\TOO$ with $n=10^7$.
  }
 \end{minipage}
 \begin{minipage}[t]{0.46\linewidth}
  \centering
  \includegraphics[keepaspectratio, width=67mm]
  {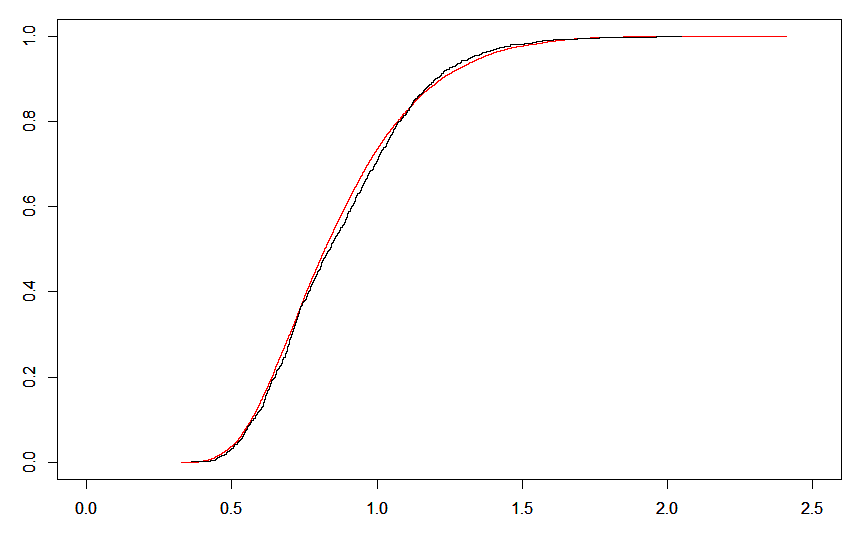}
  \subcaption{
  		EDF of $\TOO$ with $n=10^7$. 
  }
 \end{minipage}
 \caption{
 		Histogram (black line) versus theoretical density function (red line) 
 		and
 		empirical distribution function (black line) versus 
 		theoretical distribution function (red line) in (ii).
 }
 \label{model1ii_test}
\end{figure}

On the other hand, 
according to Table \ref{tab2} and (A)-(D) of Figure \ref{model1ii_test},
it can be seen that
the proportions of the test statistics $\TUO$ and $\TUT$ that exceed the critical value 
approach $1.000$ as $n$ increases, 
and the distribution of the test statistics diverges in (ii).
We can also see from (E)-(F) of Figure \ref{model1ii_test} 
that the distribution of the test statistic $\TOO$
almost corresponds with the null distribution.
Therefore, we estimate the drift parameters before and after the change
by the same procedure as Steps (2)-(3),  
and also estimate the change point of the drift parameter
when the test statistic $\TUO$ exceeds the critical value. 
Here, we construct the estimators by looking for the intervals with no change point.
Specifically, we first test for changes in the drift parameter 
in $[0.25\underline{\tau}_nT,0.75\underline{\tau}_nT]$. 
If the change is detected, we construct 
$\hat{\boldsymbol\beta}_1$ from $[0,0.25\underline{\tau}_nT]$
and $\hat{\boldsymbol\beta}_2$ from $[0.75\underline{\tau}_nT,\underline{\tau}_nT]$.
If no change is detected, we next expand the test interval to 
$[0.125\underline{\tau}_nT,0.875\underline{\tau}_nT]$, 
$[0.0625\underline{\tau}_nT,0.9375\underline{\tau}_nT]$, 
and $[0.01\underline{\tau}_nT,0.99\underline{\tau}_nT]$ 
and if the change is detected in the expanded interval, 
we estimate $\boldsymbol\beta_1^*$ and $\boldsymbol\beta_2^*$ 
using the data in the intervals that are not used in the test.
The results of these estimates are shown 
in Table \ref{tab4} and Figure \ref{model1ii_estimation}.
We can see that
the distribution of the estimator almost corresponds with 
the asymptotic distribution and the estimator has good performance
from Figure \ref{model1ii_estimation}. 
In this simulation, 
we considered the situation that the difference 
between $\gamma_1^*$ and $\gamma_2^*$ shrinks.
As we mentioned in Remark \ref{rmk11}, 
the change point can also be estimated when the difference is fixed.


{\begin{table}[h]
\captionsetup{margin=5pt}
\caption{
Mean and standard deviation of the estimators in (ii). 
True values: $\beta^*=1$, $\gamma_2^*=2$, $\tau_*^{\boldsymbol\beta}=0.4$,  
$\gamma_1^*\approx1.7488$ and $1.8005$ for $n=10^6$ and $10^7$, respectively.
}
\begin{center}
\begin{tabular*}{1.0\textwidth}{@{\extracolsep{\fill}}cccccccc}\hline
$n$ & $T$ & $h$ & $\hat\beta_1$ & $\hat\gamma_1$ & $\hat\beta_2$ & 
$\hat\gamma_2$ & $\hat\tau_n^{\boldsymbol\beta}$
\rule[0mm]{0cm}{4mm}\\\hline 
$10^6$ & 758.58 & $7.59\times10^{-4}$ 
& 1.09318 & 1.73533 & 1.11949 & 2.00802 & 0.40789  \rule[0mm]{0cm}{4mm}\\ 
&&& (0.331866) & (0.14934) & (0.46379) & (0.17179) & (0.12373) \\
$10^7$ & 2290.87 & $2.29\times10^{-4}$ 
& 1.01099 & 1.79835 & 1.02131 & 2.00017 & 0.40229  \rule[0mm]{0cm}{4mm}\\ 
&&& (0.09569) & (0.05623) & (0.09199) & (0.06371) & (0.07034) \\\hline
\end{tabular*}
\label{tab4}
\end{center}
\end{table}


\begin{figure}[h]
\captionsetup{margin=5pt}
\captionsetup[sub]{margin=5pt}
 \begin{minipage}[t]{0.46\linewidth}
  \centering
  \includegraphics[keepaspectratio, width=67mm]
  {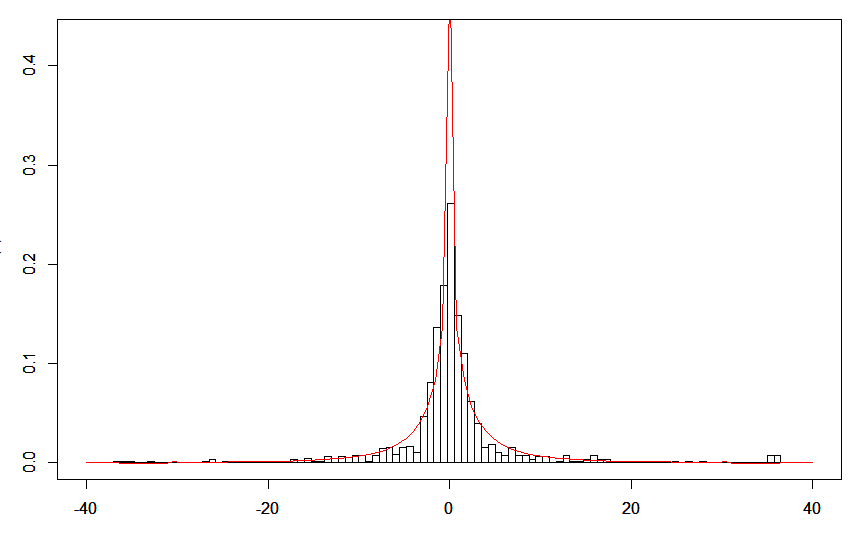}
  \subcaption{
  		Histogram of 
  		$T\Deb^2(\hat\tau_n^\beta-\tau_*^\beta)$.  
  }
 \end{minipage}
 \begin{minipage}[t]{0.46\linewidth}
  \centering
  \includegraphics[keepaspectratio, width=67mm]
  {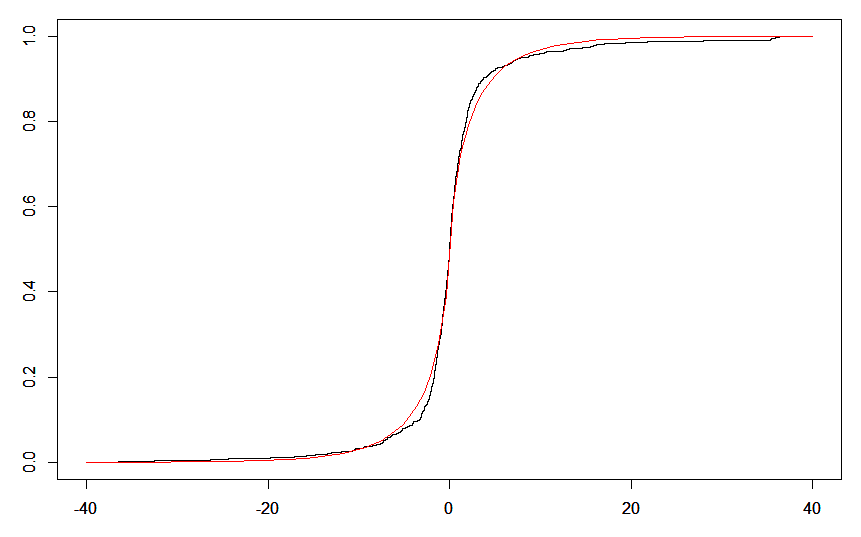}
  \subcaption{
  		EDF of 
  		$T\Deb^2(\hat\tau_n^\beta-\tau_*^\beta)$.  
  }
 \end{minipage}
 \caption{
 		Histogram (black line) versus theoretical density function (red line) 
 		and
 		empirical distribution function (black line) versus 
 		theoretical distribution function (red line) with $n=10^7$ in (ii).
 }
 \label{model1ii_estimation}
\end{figure}
}


\begin{table}[h]
\captionsetup{margin=5pt}
\caption{
Mean and standard deviation of the estimators in (iii). 
True values: $\beta^*=1$, $\gamma_2^*=2$, 
$\gamma_1^*\approx1.74881$ and $1.80047$ 
for $n=10^6$ and $10^7$, respectively.
}
\begin{center}
\begin{tabular*}{1.0\textwidth}{@{\extracolsep{\fill}}ccccccc}\hline
$n$ & $T$ & $h$ & $\check\beta_1$ & $\check\gamma_1$ & $\check\beta_2$ & 
$\check\gamma_2$ 
\rule[0mm]{0cm}{4mm}\\\hline 
$10^6$ & 758.58 & $7.59\times10^{-4}$ 
& 1.00714 & 1.74722 & 1.02731 & 1.99986    \rule[0mm]{0cm}{4mm}\\ 
&&& (0.05874) & (0.04078) & (0.12144) & (0.09612)  \\
$10^7$ & 2290.87 & $2.29\times10^{-4}$ 
& 1.00215 & 1.80070 & 1.00840 & 1.99795    \rule[0mm]{0cm}{4mm}\\ 
&&& (0.03331) & (0.02423) & (0.06540) & (0.05432)  \\\hline
\end{tabular*}
\label{tab5}
\end{center}
\end{table}


\begin{figure}[h]
\captionsetup{margin=5pt}
\captionsetup[sub]{margin=5pt}
 \begin{minipage}[t]{0.46\linewidth}
  \centering
  \includegraphics[keepaspectratio, width=67mm]{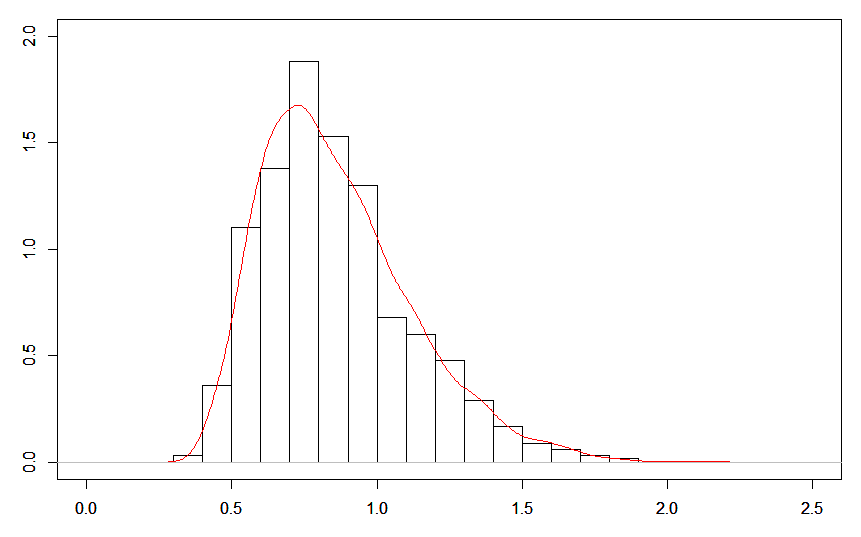}
  \subcaption{
  		Histogram of $\TUO$ with $n=10^7$.
  }
 \end{minipage}
 \begin{minipage}[t]{0.46\linewidth}
  \centering
  \includegraphics[keepaspectratio, width=67mm]{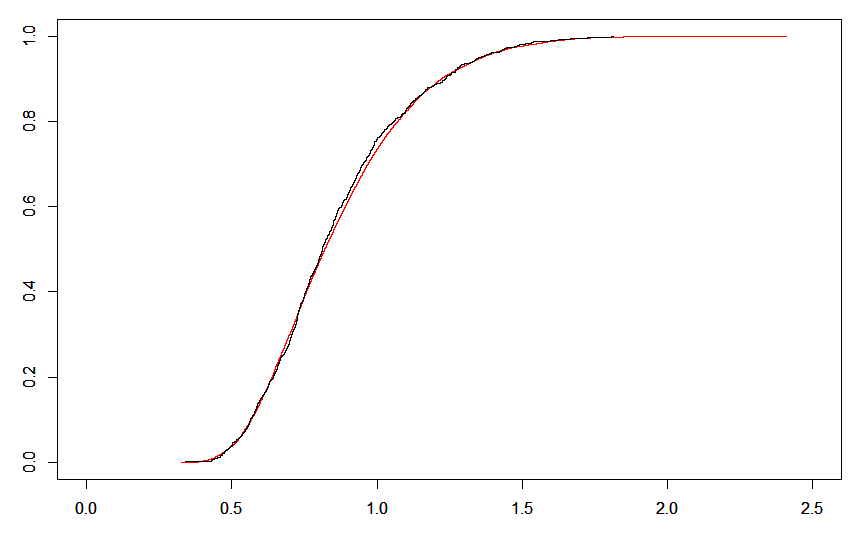}
  \subcaption{
  		EDF of $\TUO$ with $n=10^7$. 
  }
 \end{minipage}\\
 \begin{minipage}[t]{0.46\linewidth}
  \centering
  \includegraphics[keepaspectratio, width=67mm]{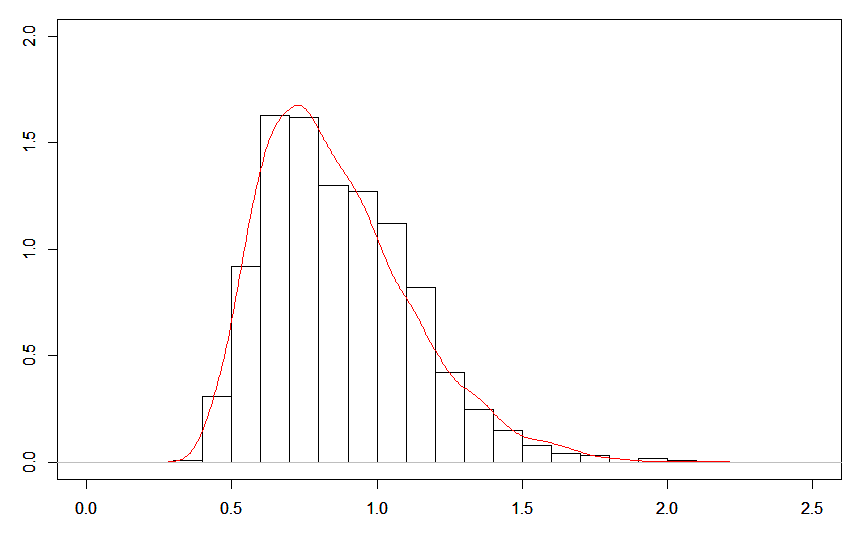}
  \subcaption{
  		Histogram of $\TOO$ with $n=10^7$.
  }
 \end{minipage}
 \begin{minipage}[t]{0.46\linewidth}
  \centering
  \includegraphics[keepaspectratio, width=67mm]{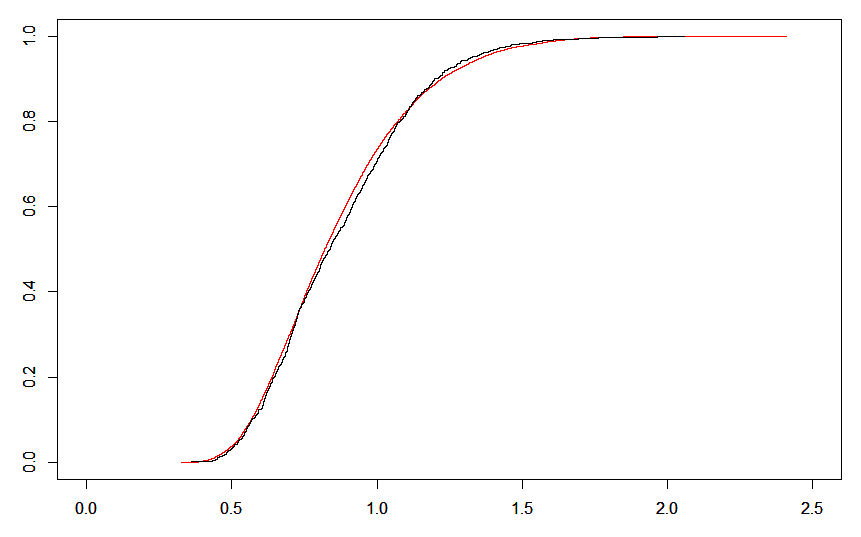}
  \subcaption{
  		EDF of $\TOO$ with $n=10^7$. 
  }
 \end{minipage}
 \caption{
 		Histogram (black line) versus theoretical density function (red line) 
 		and
 		empirical distribution function (black line) versus 
 		theoretical distribution function (red line) in (iii).
 }
 \label{model1iii_test}
\end{figure}

\clearpage

\begin{figure}[h]
\captionsetup{margin=5pt}
\captionsetup[sub]{margin=5pt}
 \begin{minipage}[t]{0.46\linewidth}
  \centering
  \includegraphics[keepaspectratio, width=67mm]
  {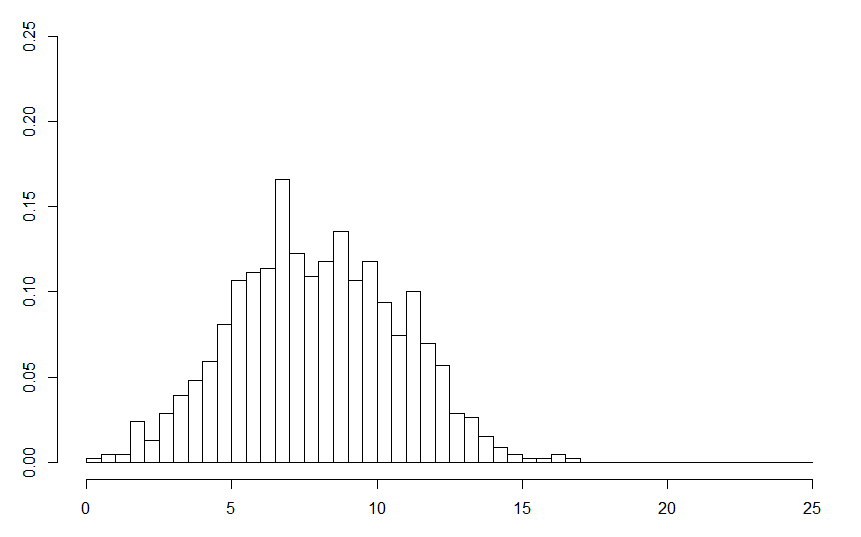}
  \subcaption{
  		$n=10^6$. 
  }
 \end{minipage}
 \begin{minipage}[t]{0.46\linewidth}
  \centering
  \includegraphics[keepaspectratio, width=67mm]
  {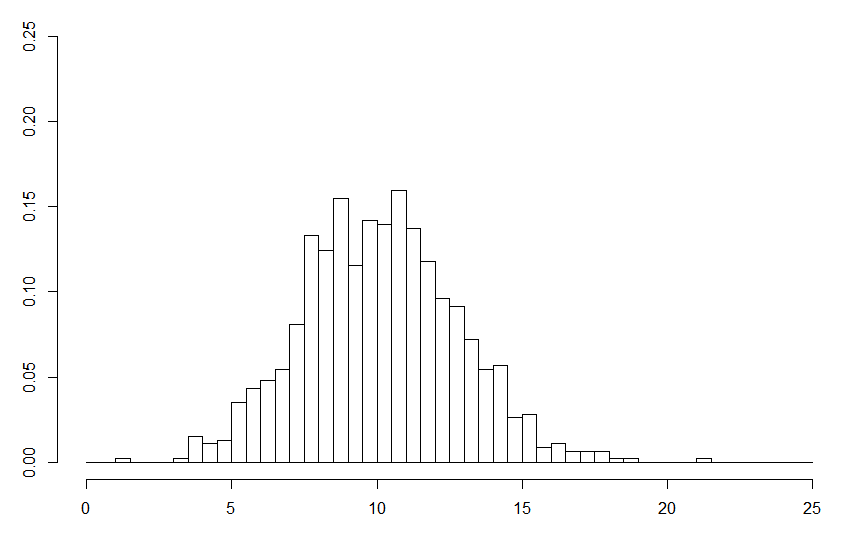}
  \subcaption{
  		$n=10^7$. 
  }
 \end{minipage}
 \caption{
  		Histogram of 
  		$\sqrt T\|\check{\boldsymbol\beta}_1-\check{\boldsymbol\beta}_2\|$  
 		in (iii).
 }
 \label{model1iii_test_M}
\end{figure}


\subsection{Model 2}
We consider the hyperbolic diffusion model defined by
\begin{align*}
\dd X_t=\biggl(\beta-\frac{\gamma X_t}{\sqrt{1+X_t^2}}\biggr)\dd t
+\alpha\dd W_t,\quad X_0=x_0,
\end{align*}
where $\alpha>0$, $\beta\in\mathbb R$ and $|\beta|<\gamma$.
In order to investigate the asymptotic performance of 
the test statistics and the estimator, 
we treat the following stochastic differential equation
\begin{align*}
X_t
=
\begin{cases}
\displaystyle
X_0+\int_0^t \biggl(\beta-\frac{\gamma X_s}{\sqrt{1+X_s^2}}\biggr)\dd s
+\alpha_1^* W_t, 
& t\in[0,\tau_*^\alpha T),\\
\displaystyle
X_{\tau_*^\alpha T}
+\int_{\tau_*^\alpha T}^t \biggl(\beta-\frac{\gamma X_s}{\sqrt{1+X_s^2}}\biggr)\dd s
+\alpha_2^* (W_t-W_{\tau_*^\alpha T}), & t\in[\tau_*^\alpha T,T],
\end{cases}
\end{align*}
where $x_0=1$, $\tau_*^\alpha=0.4$, $\alpha_1^*=1+n^{-0.36}$, $\alpha_2^*=1$, 
$\boldsymbol\beta=(\beta,\gamma)$.

We consider the following three situations.
\begin{enumerate}
\item[(i)]
The drift parameter $\boldsymbol\beta=(1,2)$
does not change over $[0,T]$.

\item[(ii)]
The drift parameter $\boldsymbol\beta$ changes 
from $\boldsymbol\beta_1^*=(1,2)$ to $\boldsymbol\beta_2^*=(0.5,2)$ 
at $\tau_*^{\boldsymbol\beta}=0.7$ (Case B). 

\item[(iii)]
The drift parameter $\boldsymbol\beta$ changes 
from $\boldsymbol\beta_1^*=(1,2)$ to $\boldsymbol\beta_2^*=(0.5,2)$ 
at $\tau_*^{\boldsymbol\beta}=\tau_*^\alpha=0.4$. 

\end{enumerate}
The number of iteration is 1000.
We set that the sample size of the data $\{\Xt\}_{i=0}^n$ is $n=10^6$ or $10^7$, 
$h=n^{-0.625}$, $T=nh=n^{0.375}$, $nh^2=n^{-0.25}$.

\begin{table}[h]
\captionsetup{margin=5pt}
\caption{
Mean and standard deviation of the estimators. 
True values: $\alpha_2^*=1$, $\tau_*^\alpha=0.4$, 
$\alpha_1^*\approx 1.00692$ and $1.00302$
for $n=10^6$ and $10^7$, respectively.
}
\begin{center}
\begin{tabular*}{1.0\textwidth}{@{\extracolsep{\fill}}ccccccc}\hline
$n$ & $T$ & $h$ && $\hat\alpha_1$ & $\hat\alpha_2$ & $\hat\tau_n^\alpha$ 
\rule[0mm]{0cm}{4mm}\\\hline 
$10^6$ & 177.83 & $1.78\times10^{-4}$ & 
(i)& 1.00702 & 1.00002 & 0.39858    \rule[0mm]{0cm}{4mm}\\ 
&&&& (0.00170) & (0.00172) & (0.06628)  \\
&&& 
(ii)& 1.00712 & 0.99992 & 0.39286    \rule[0mm]{0cm}{4mm}\\ 
&&&& (0.00176) & (0.00152) & (0.05996)  \\
&&& 
(iii)& 1.00702 & 1.00003 & 0.39922    \rule[0mm]{0cm}{4mm}\\ 
&&&& (0.00172) & (0.00173) & (0.06767)  \\\hline
$10^7$ & 421.70 & $4.22\times10^{-4}$ &
(i)& 1.00304 & 1.00001 & 0.39855   \rule[0mm]{0cm}{4mm}\\ 
&&&& (0.00045) & (0.00050) & (0.02576)  \\
&&& 
(ii)& 1.00304 & 1.00003 & 0.39871    \rule[0mm]{0cm}{4mm}\\ 
&&&& (0.00045) & (0.00047) & (0.02728)  \\
&&& 
(iii)& 1.00304 & 1.00002 & 0.39855    \rule[0mm]{0cm}{4mm}\\ 
&&&& (0.00045) & (0.00051) & (0.02635)  \\\hline
\end{tabular*}
\label{tab6}
\end{center}
\end{table}

We test for changes in the diffusion parameter in the interval $[0,T]$
in 1000 iterations. 
In all situations (i)-(iii), 
the change was detected $990$ times when $n=10^6$ and
$1000$ times when $n=10^7$.
If the change in the diffusion parameter is detected, 
the next step is to estimate the parameters $\alpha_1^*$ and $\alpha_2^*$ 
in the same way to estimate $\boldsymbol\beta_1^*$ and $\boldsymbol\beta_2^*$ 
in situation (ii) of model 1, 
and estimate $\tau_*^\alpha$ using the estimators 
$\hat\alpha_1$ and $\hat\alpha_2$.
The estimates of $\alpha_1^*$, $\alpha_2^*$ and $\tau_*^\alpha$ are shown 
in Table \ref{tab6}.
In this case, 
we chose $\epsilon_1=0.9+1.8\log_n|\hat\alpha_1-\hat\alpha_2|$ for all iterations. 


Next, we test for changes in the drift parameter in the intervals 
$[0,\underline{\tau}_n T]$ and $[\overline{\tau}_n T,T]$. 
Table \ref{tab7} and 
Figures \ref{model2i_test}, \ref{model2ii_test} and \ref{model2iii_test}
show the simulation results of the tests for changes in the drift parameter.
In (i) and (iii), 
it can be seen that
the test statistics have good performance 
from Table \ref{tab7}, Figures \ref{model2i_test} and \ref{model2iii_test}.
Hence, in (i) and (iii), we construct
$\check{\boldsymbol\beta}_1$ and $\check{\boldsymbol\beta}_2$ 
from the intervals $[0,\underline{\tau}_n T]$ and $[\overline{\tau}_n T,T]$,
respectively 
when the test statistics $\TUO$ and $\TOO$ do not exceed the critical value, 
and investigate 
$\sqrt T\|\check{\boldsymbol\beta}_1-\check{\boldsymbol\beta}_2\|$. 
The results of the estimates of $\boldsymbol\beta_1^*$ and $\boldsymbol\beta_2^*$ 
in (i) and (iii) are summarized in 
Table \ref{tab8} and Figure \ref{model2i_test_M},
and Table \ref{tab10} and Figure \ref{model2iii_test_M}, respectively.
It can be seen that 
$\sqrt T\|\check{\boldsymbol\beta}_1-\check{\boldsymbol\beta}_2\|=O_p(1)$ in (i),
and $\sqrt T\|\check{\boldsymbol\beta}_1-\check{\boldsymbol\beta}_2\|$ 
diverges in (iii)
from Figures \ref{model2i_test_M} and \ref{model2iii_test_M}, respectively.

\begin{table}[h]
\captionsetup{margin=5pt}
\caption{
Proportions of the corresponding critical values exceeded.
}
\begin{center}
\begin{tabular*}{1.0\textwidth}{@{\extracolsep{\fill}}cccccccc}\hline
$n$ & $T$ & $h$ && $\TUO$ & $\TUT$ & $\TOO$ & $\TOT$ 
\rule[0mm]{0cm}{5mm}\\\hline 
$10^6$ & 177.83 & $1.78\times10^{-4}$  
& (i)
& 0.035 & 0.043 & 0.060 & 0.051  
\rule[0mm]{0cm}{4mm}\\ 
&&&& ($35/990$) & ($43/990$) & ($59/990$) & ($50/990$)  \\
&&& (ii)
& 0.034 & 0.060 & 0.510 & 0.414    \rule[0mm]{0cm}{4mm}\\ 
&&&& ($34/990$) & ($59/990$) & ($505/990$) & ($410/990$)  \\
&&& (iii)
& 0.043 & 0.071 & 0.052 & 0.065    \rule[0mm]{0cm}{4mm}\\ 
&&&& ($43/990$) & ($70/990$) & ($51/990$) & ($64/990$)  \\\hline
$10^7$ & 421.70 & $4.22\times10^{-4}$ 
& (i)
& 0.038 & 0.040 & 0.038 & 0.031    \rule[0mm]{0cm}{4mm}\\ 
&&& (ii)
& 0.040 & 0.040 & 0.941 & 0.887    \rule[0mm]{0cm}{4mm}\\ 
&&& (iii)
& 0.043 & 0.049 & 0.044 & 0.048    \rule[0mm]{0cm}{4mm}\\\hline
\end{tabular*}
\label{tab7}
\end{center}
\end{table}

\begin{figure}[h]
\captionsetup{margin=5pt}
\captionsetup[sub]{margin=5pt}
 \begin{minipage}[t]{0.46\linewidth}
  \centering
  \includegraphics[keepaspectratio, width=67mm]{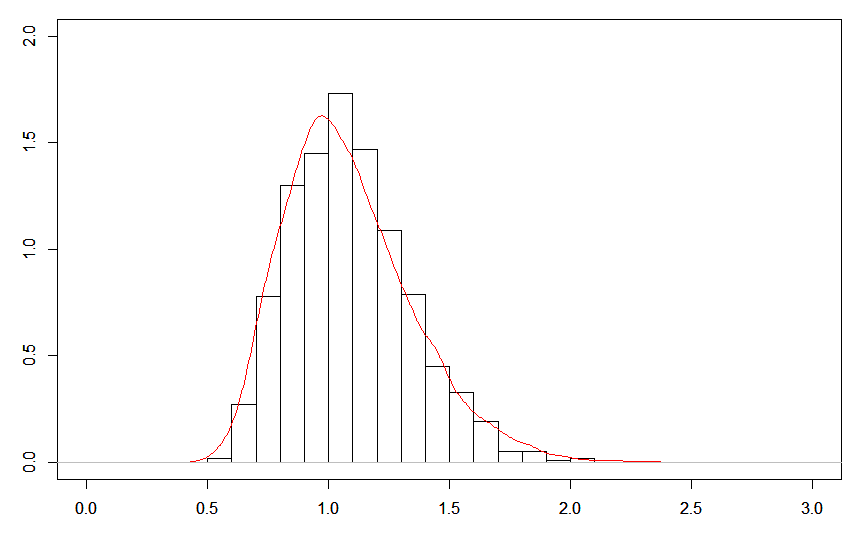}
  \subcaption{
  		Histogram of $\TUT$ with $n=10^7$.
  }
 \end{minipage}
 \begin{minipage}[t]{0.46\linewidth}
  \centering
  \includegraphics[keepaspectratio, width=67mm]{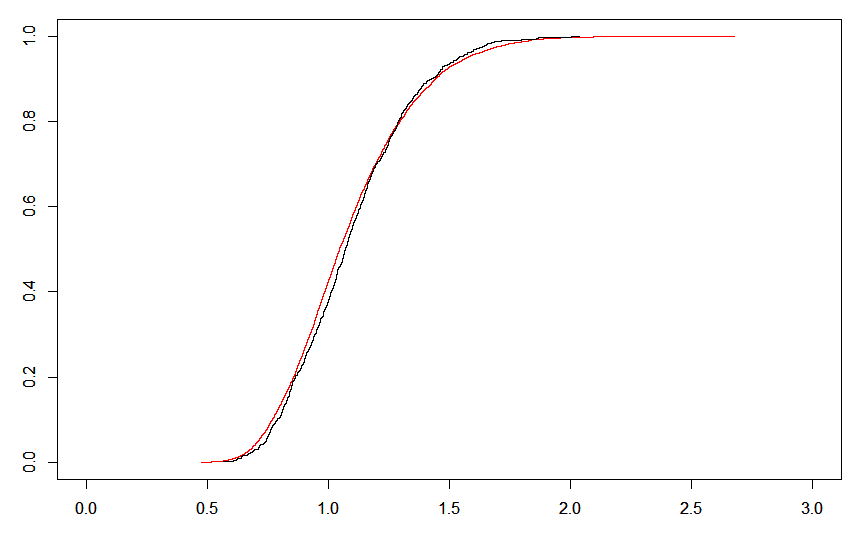}
  \subcaption{
  		EDF of $\TUT$ with $n=10^7$. 
  }
 \end{minipage}\\
 \begin{minipage}[t]{0.46\linewidth}
  \centering
  \includegraphics[keepaspectratio, width=67mm]{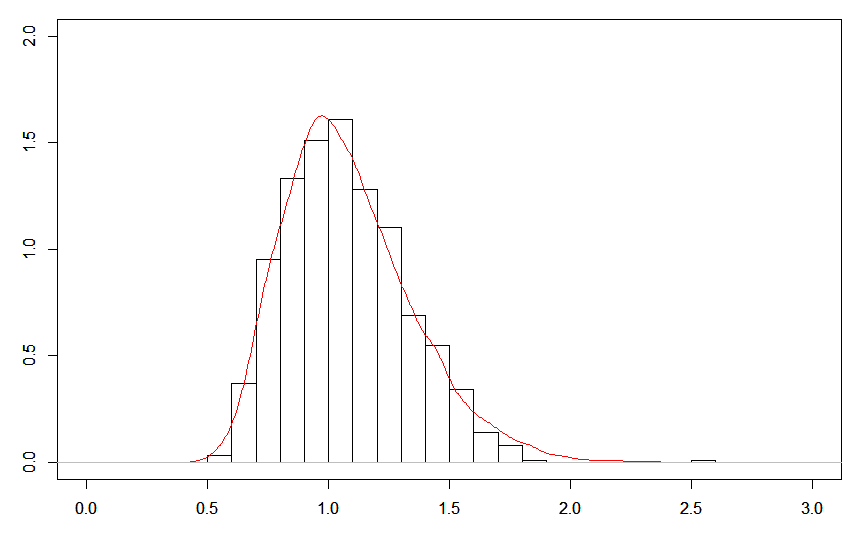}
  \subcaption{
  		Histogram of $\TOT$ with $n=10^7$.
  }
 \end{minipage}
 \begin{minipage}[t]{0.46\linewidth}
  \centering
  \includegraphics[keepaspectratio, width=67mm]{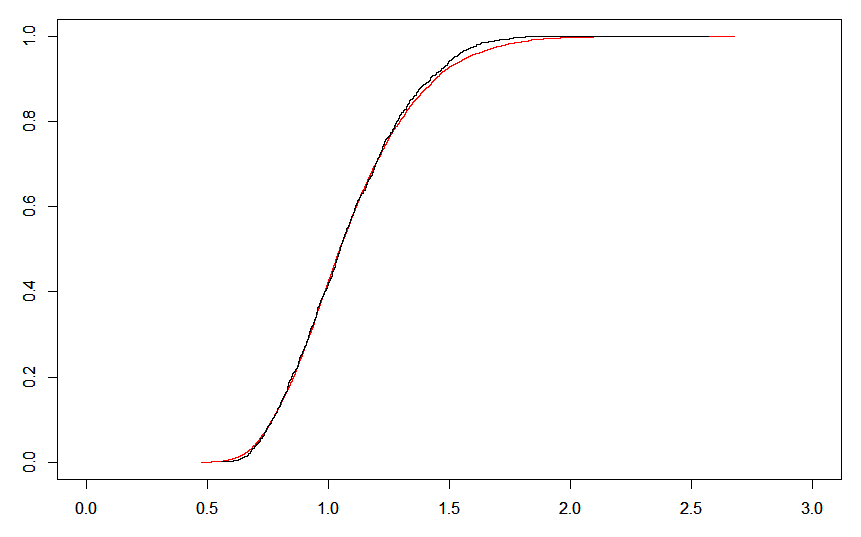}
  \subcaption{
  		EDF of $\TOT$ with $n=10^7$. 
  }
 \end{minipage}
 \caption{
 		Histogram (black line) versus theoretical density function (red line) 
 		and
 		empirical distribution function (black line) versus 
 		theoretical distribution function (red line) in (i).
 }
 \label{model2i_test}
\end{figure}

\begin{figure}[h]
\captionsetup{margin=5pt}
\captionsetup[sub]{margin=5pt}
 \begin{minipage}[t]{0.46\linewidth}
  \centering
  \includegraphics[keepaspectratio, width=67mm]
  {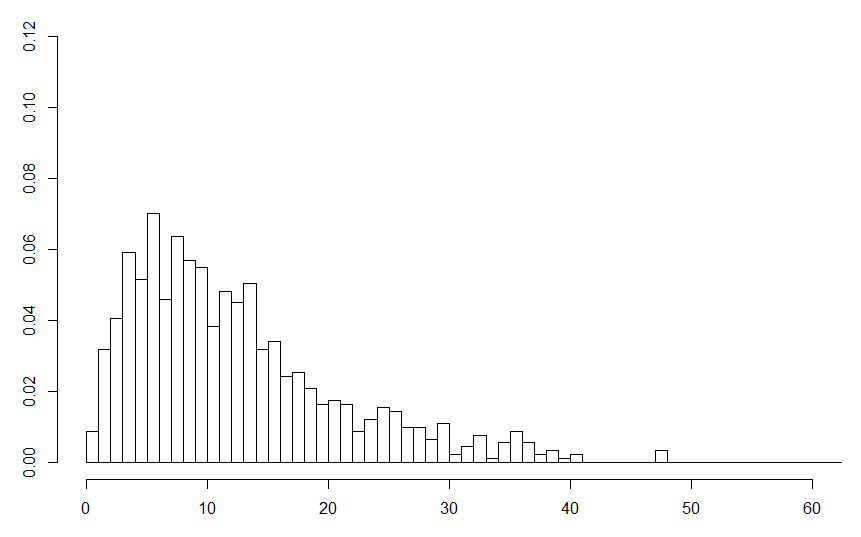}
  \subcaption{
  		$n=10^6$. 
  }
 \end{minipage}
 \begin{minipage}[t]{0.46\linewidth}
  \centering
  \includegraphics[keepaspectratio, width=67mm]
  {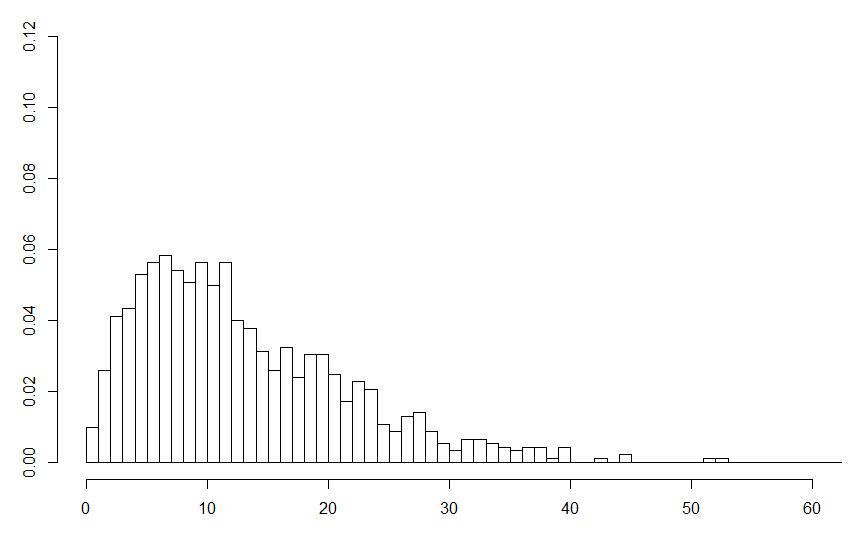}
  \subcaption{
  		$n=10^7$. 
  }
 \end{minipage}
 \caption{
  		Histogram of 
  		$\sqrt T\|\check{\boldsymbol\beta}_1-\check{\boldsymbol\beta}_2\|$  
 		in (i).
 }
 \label{model2i_test_M}
\end{figure}

\begin{table}[h]
\captionsetup{margin=5pt}
\caption{
Mean and standard deviation of the estimators in (i). 
True values: $\beta^*=1$, $\gamma^*=2$.
}
\begin{center}
\begin{tabular*}{1.0\textwidth}{@{\extracolsep{\fill}}ccccccc}\hline
$n$ & $T$ & $h$ & $\check\beta_1$ & $\check\gamma_1$ & $\check\beta_2$ & 
$\check\gamma_2$ 
\rule[0mm]{0cm}{4mm}\\\hline 
$10^6$ & 177.83 & $1.78\times10^{-4}$  
& 1.09928 & 2.18729 & 1.05088 & 2.09803    \rule[0mm]{0cm}{4mm}\\ 
&&& (0.53315) & (0.72917) & (0.34763) & (0.48807)  \\
$10^7$ & 421.70 & $4.22\times10^{-4}$ 
& 1.01788 & 2.04133 & 1.01649 & 2.03230    \rule[0mm]{0cm}{4mm}\\ 
&&& (0.13699) & (0.21138) & (0.11101) & (0.17223)  \\\hline
\end{tabular*}
\label{tab8}
\end{center}
\end{table}

From Table \ref{tab7} and (A)-(B) of Figure \ref{model2ii_test},
we see that in (ii), the distribution of the test statistic $\TUT$ 
almost corresponds with the null distribution.
Moreover, 
it can be seen from Table \ref{tab7} 
and (C)-(F) of Figure \ref{model2ii_test}
that
the proportions of the test statistics $\TOO$ and $\TOT$ that exceed the critical value 
approach $1.000$ as $n$ increases, 
and the distribution of the test statistics $\TOT$ diverges in (ii).
Therefore, in (ii), we estimate the drift parameters before and after the change, 
and also estimate the change point
when the test statistic $\TOO$ exceeds the critical value. 
Here, we construct the estimators 
$\hat{\boldsymbol\beta}_1$ and $\hat{\boldsymbol\beta}_2$ 
in the same way as in situation (ii) of model 1.
The results of these estimates are shown 
in Table \ref{tab9} and Figure \ref{model2ii_estimation}.
From Figure \ref{model2ii_estimation},
we can see that
the distribution of the estimator does not diverge when increasing 
from $n=10^6$ to $n=10^7$, which 
implies that the estimator has good performance.


\begin{figure}[h]
\captionsetup{margin=5pt}
\captionsetup[sub]{margin=5pt}
 \begin{minipage}[t]{0.46\linewidth}
  \centering
  \includegraphics[keepaspectratio, width=67mm]
  {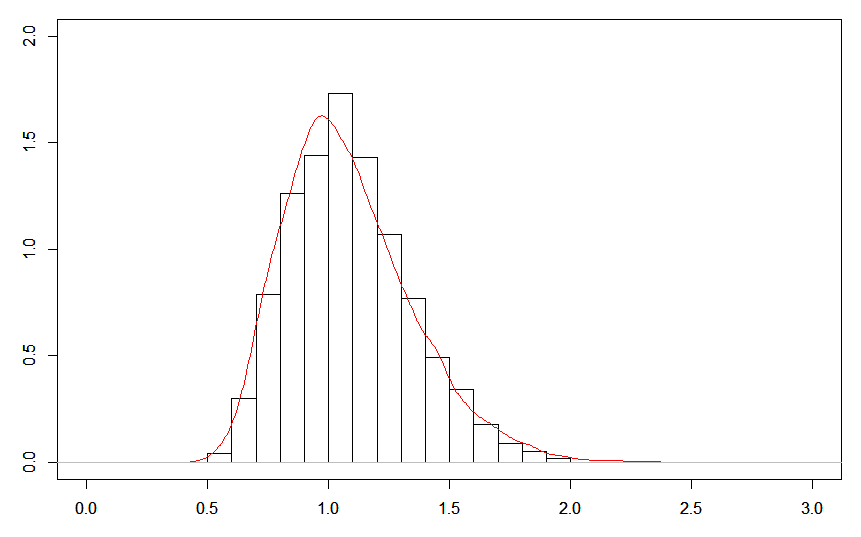}
  \subcaption{
  		Histogram of $\TUT$ with $n=10^7$.
  }
 \end{minipage}
 \begin{minipage}[t]{0.46\linewidth}
  \centering
  \includegraphics[keepaspectratio, width=67mm]
  {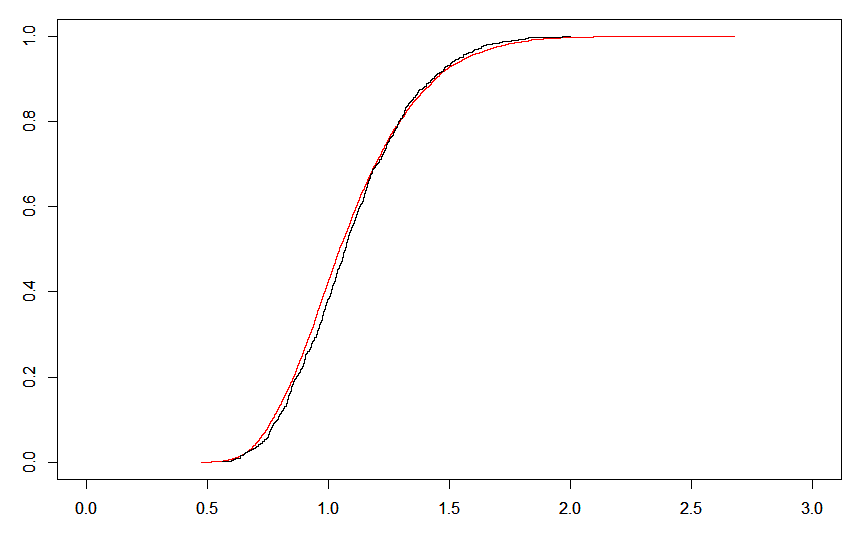}
  \subcaption{
  		EDF of $\TUT$ with $n=10^7$. 
  }
 \end{minipage}\\
 \begin{minipage}[t]{0.46\linewidth}
  \centering
  \includegraphics[keepaspectratio, width=67mm]
  {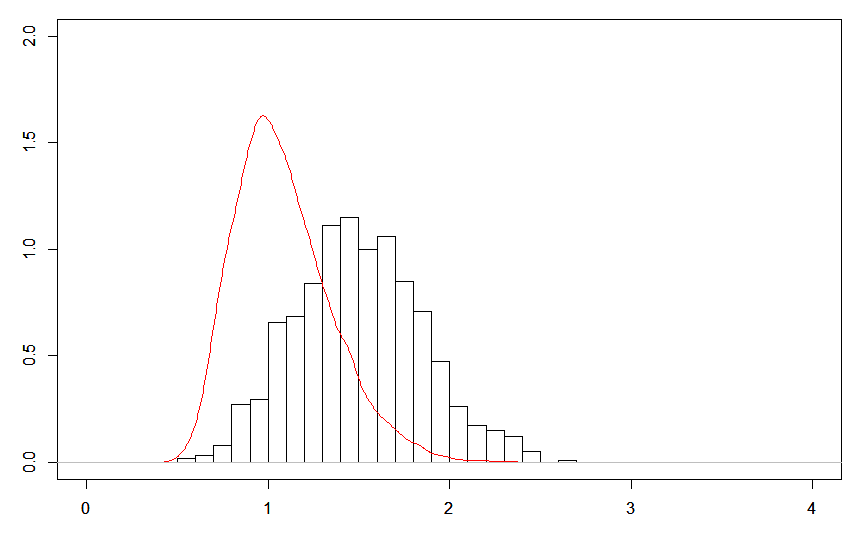}
  \subcaption{
  		Histogram of $\TOT$ with $n=10^6$.
  }
 \end{minipage}
 \begin{minipage}[t]{0.46\linewidth}
  \centering
  \includegraphics[keepaspectratio, width=67mm]
  {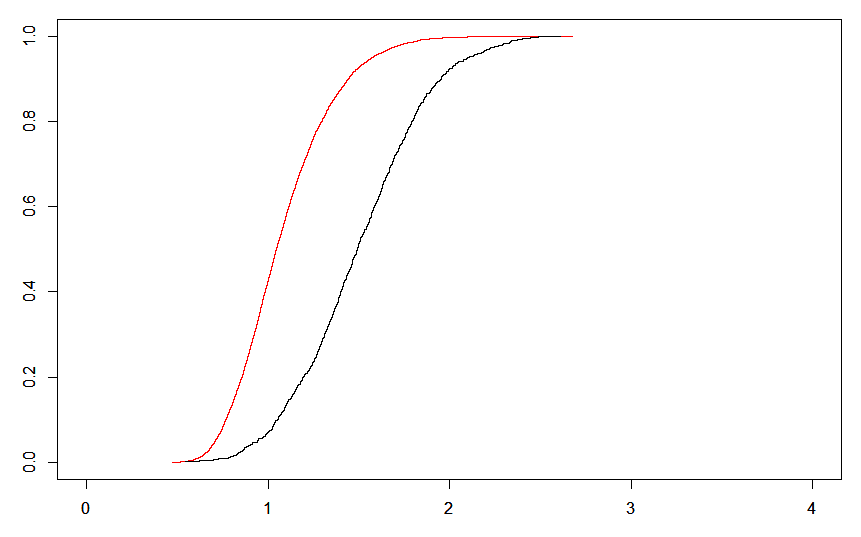}
  \subcaption{
  		EDF of $\TOT$ with $n=10^6$. 
  }
 \end{minipage}\\
 \begin{minipage}[t]{0.46\linewidth}
  \centering
  \includegraphics[keepaspectratio, width=67mm]
  {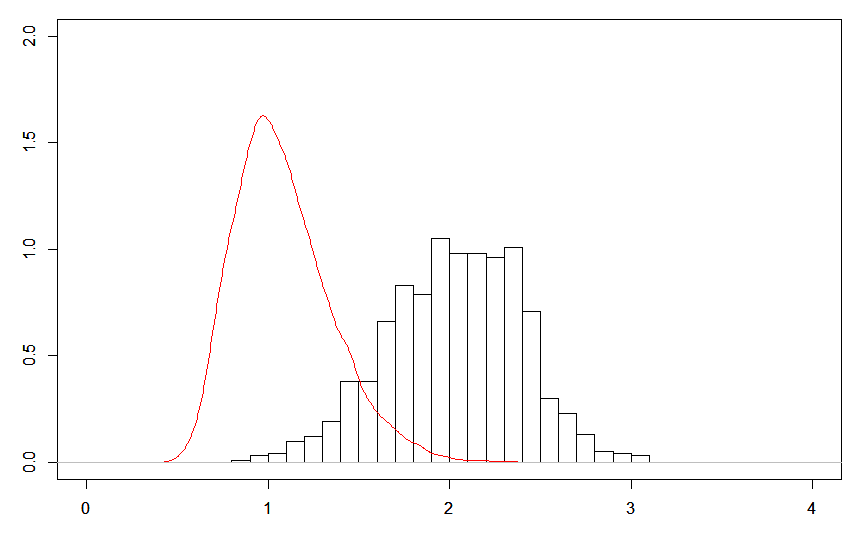}
  \subcaption{
  		Histogram of $\TOT$ with $n=10^7$.
  }
 \end{minipage}
 \begin{minipage}[t]{0.46\linewidth}
  \centering
  \includegraphics[keepaspectratio, width=67mm]
  {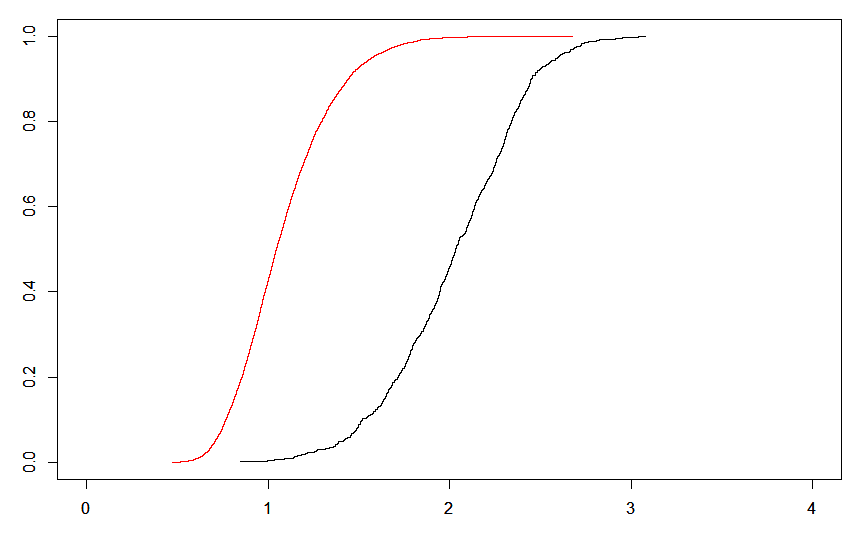}
  \subcaption{
  		EDF of $\TOT$ with $n=10^7$. 
  }
 \end{minipage}
 \caption{
 		Histogram (black line) versus theoretical density function (red line) 
 		and
 		empirical distribution function (black line) versus 
 		theoretical distribution function (red line) in (ii).
 }
 \label{model2ii_test}
\end{figure}

\begin{table}[h]
\captionsetup{margin=5pt}
\caption{
Mean and standard deviation of the estimators in (ii). 
True values: $\beta_1^*=1$, $\beta_2^*=0.5$, $\gamma^*=2$, 
$\tau_*^{\boldsymbol\beta}=0.7$.
}
\begin{center}
\begin{tabular*}{1.0\textwidth}{@{\extracolsep{\fill}}cccccccc}\hline
$n$ & $T$ & $h$ & $\hat\beta_1$ & $\hat\gamma_1$ & $\hat\beta_2$ & 
$\hat\gamma_2$ & $\hat\tau_n^{\boldsymbol\beta}$
\rule[0mm]{0cm}{4mm}\\\hline 
$10^6$ & 177.83 & $1.78\times10^{-4}$ 
& 1.72974 & 3.09513 & 0.52996 & 2.91408 & 0.70629  \rule[0mm]{0cm}{4mm}\\ 
&&& (1.81075) & (2.15126) & (0.84344) & (1.83324) & (0.13175) \\
$10^7$ & 421.70 & $4.22\times10^{-4}$ 
& 1.14360 & 2.25463 & 0.51434 & 2.11352 & 0.69866  \rule[0mm]{0cm}{4mm}\\ 
&&& (0.59646) & (0.82162) & (0.21015) & (0.45589) & (0.06982) \\\hline
\end{tabular*}
\label{tab9}
\end{center}
\end{table}

\begin{figure}[h]
\captionsetup{margin=5pt}
\captionsetup[sub]{margin=5pt}
 \begin{minipage}[t]{0.46\linewidth}
  \centering
  \includegraphics[keepaspectratio, width=67mm]
  {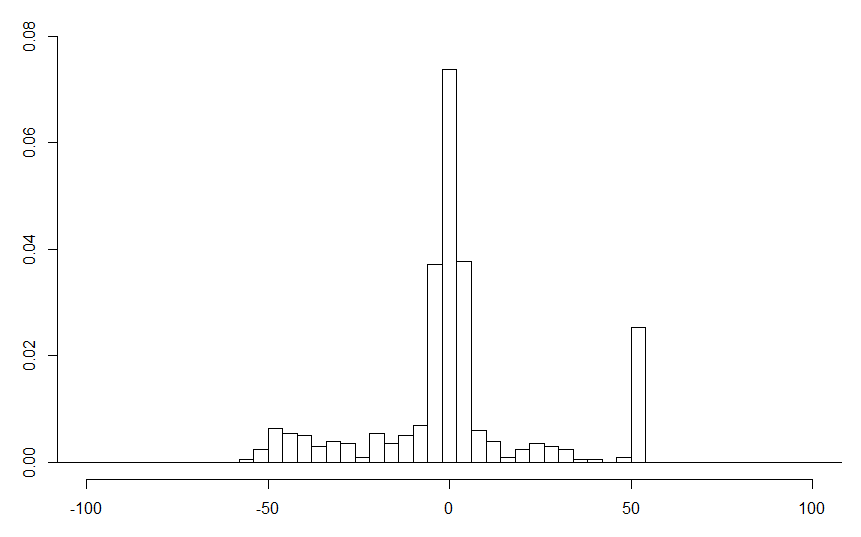}
  \subcaption{
  		$n=10^6$.
  }
 \end{minipage}
 \begin{minipage}[t]{0.46\linewidth}
  \centering
  \includegraphics[keepaspectratio, width=67mm]
  {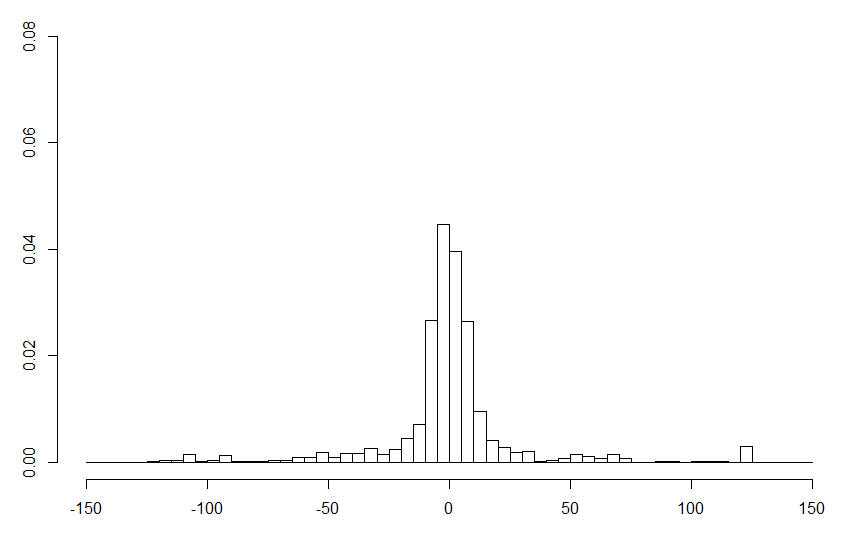}
  \subcaption{
  		$n=10^7$.
  }
 \end{minipage}
 \caption{
  		Histogram of 
  		$T(\hat\tau_n^\beta-\tau_*^\beta)$ in (ii).
 }
 \label{model2ii_estimation}
\end{figure}

\begin{table}[h]
\captionsetup{margin=5pt}
\caption{
Mean and standard deviation of the estimators in (iii). 
True values: $\beta_1^*=1$, $\beta_1^*=0.5$, $\gamma^*=2$.
}
\begin{center}
\begin{tabular*}{1.0\textwidth}{@{\extracolsep{\fill}}ccccccc}\hline
$n$ & $T$ & $h$ & $\check\beta_1$ & $\check\gamma_1$ & $\check\beta_2$ & 
$\check\gamma_2$ 
\rule[0mm]{0cm}{4mm}\\\hline 
$10^6$ & 177.83 & $1.78\times10^{-4}$  
& 1.07202 & 2.15235 & 0.52224 & 2.07371    \rule[0mm]{0cm}{4mm}\\ 
&&& (0.42817) & (0.61731) & (0.27310) & (0.47234)  \\
$10^7$ & 421.70 & $4.22\times10^{-4}$ 
& 1.01424 & 2.03951 & 0.50641 & 2.02274    \rule[0mm]{0cm}{4mm}\\ 
&&& (0.13744) & (0.21292) & (0.07692) & (0.15020)  \\\hline
\end{tabular*}
\label{tab10}
\end{center}
\end{table}

\clearpage

\begin{figure}[h]
\captionsetup{margin=5pt}
\captionsetup[sub]{margin=5pt}
 \begin{minipage}[t]{0.46\linewidth}
  \centering
  \includegraphics[keepaspectratio, width=67mm]
  {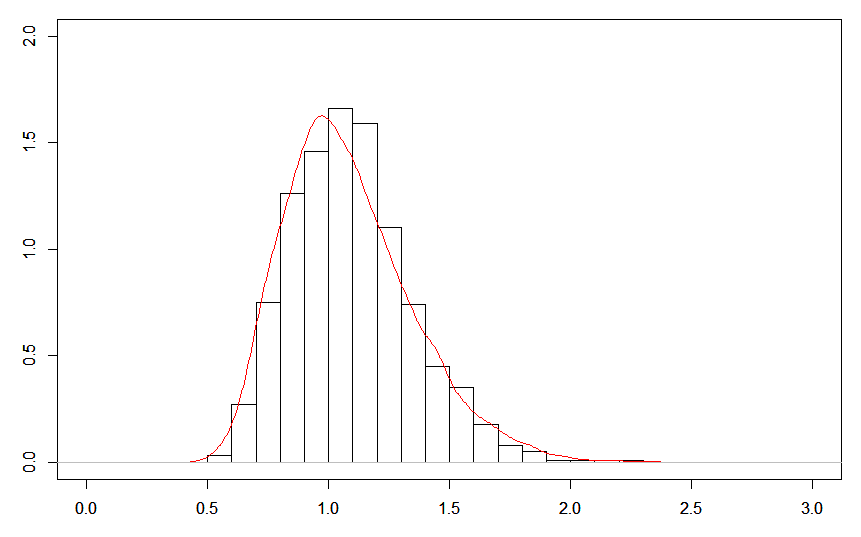}
  \subcaption{
  		Histogram of $\TUT$ with $n=10^7$.
  }
 \end{minipage}
 \begin{minipage}[t]{0.46\linewidth}
  \centering
  \includegraphics[keepaspectratio, width=67mm]
  {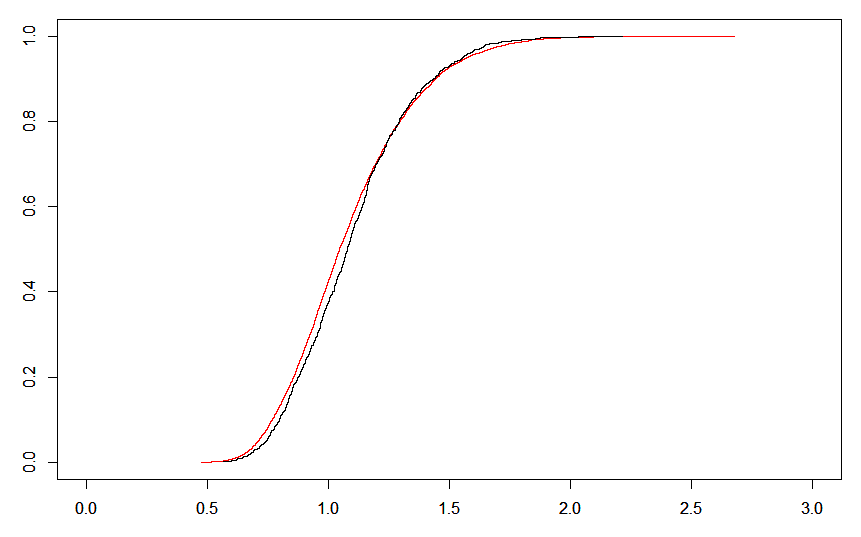}
  \subcaption{
  		EDF of $\TUT$ with $n=10^7$. 
  }
 \end{minipage}\\
 \begin{minipage}[t]{0.46\linewidth}
  \centering
  \includegraphics[keepaspectratio, width=67mm]
  {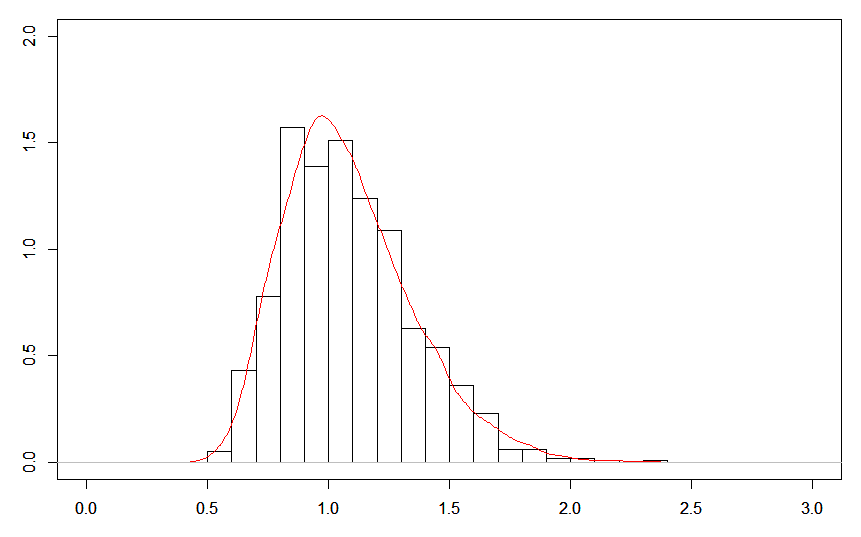}
  \subcaption{
  		Histogram of $\TOT$ with $n=10^7$.
  }
 \end{minipage}
 \begin{minipage}[t]{0.46\linewidth}
  \centering
  \includegraphics[keepaspectratio, width=67mm]
  {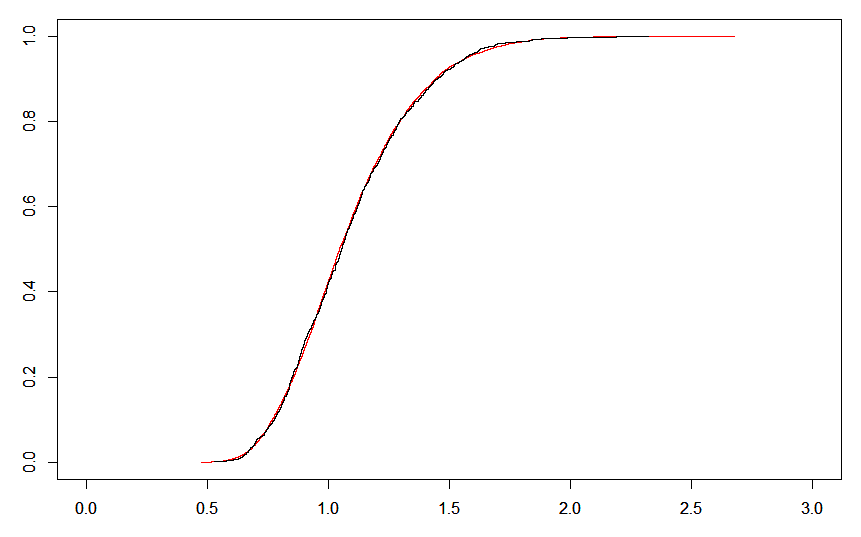}
  \subcaption{
  		EDF of $\TOT$ with $n=10^7$. 
  }
 \end{minipage}
 \caption{
 		Histogram (black line) versus theoretical density function (red line) 
 		and
 		empirical distribution function (black line) versus 
 		theoretical distribution function (red line) in (iii).
 }
 \label{model2iii_test}
\end{figure}


\begin{figure}[h]
\captionsetup{margin=5pt}
\captionsetup[sub]{margin=5pt}
 \begin{minipage}[t]{0.46\linewidth}
  \centering
  \includegraphics[keepaspectratio, width=67mm]
  {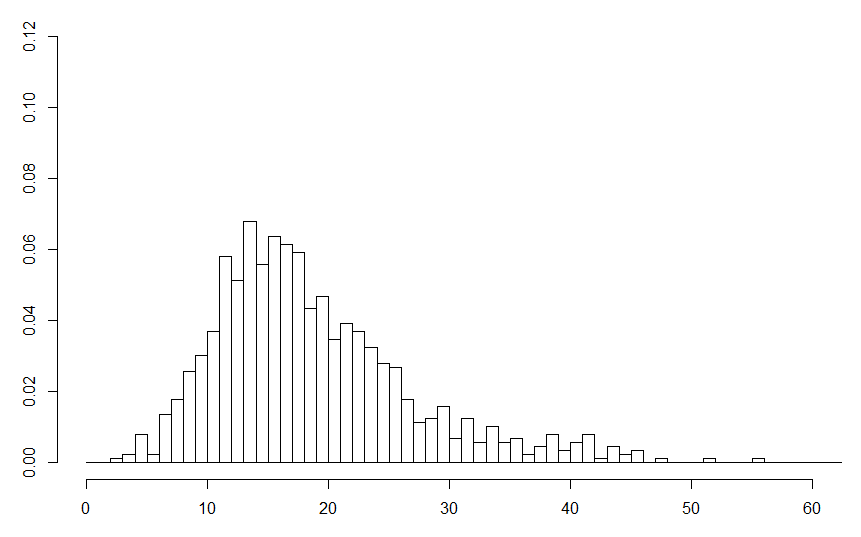}
  \subcaption{
  		$n=10^6$. 
  }
 \end{minipage}
 \begin{minipage}[t]{0.46\linewidth}
  \centering
  \includegraphics[keepaspectratio, width=67mm]
  {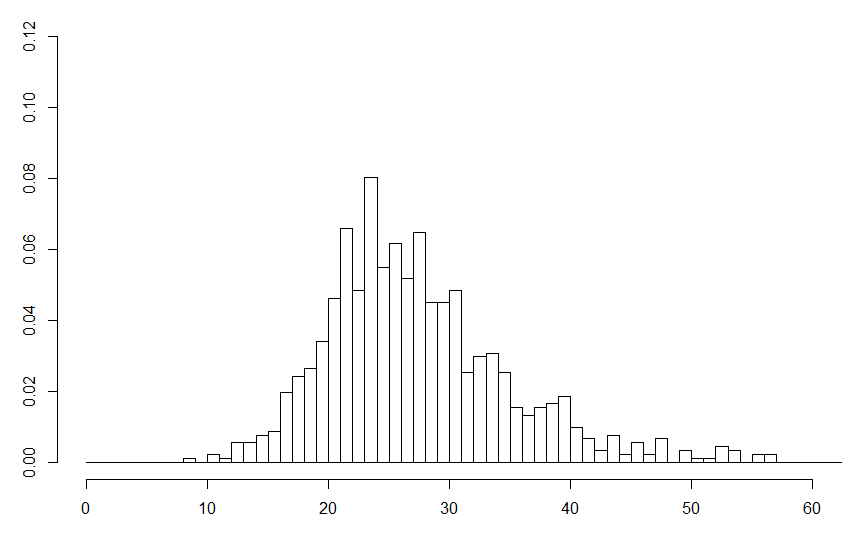}
  \subcaption{
  		$n=10^7$. 
  }
 \end{minipage}
 \caption{
  		Histogram of 
  		$\sqrt T\|\check{\boldsymbol\beta}_1-\check{\boldsymbol\beta}_2\|$  
 		in (iii).
 }
 \label{model2iii_test_M}
\end{figure}



\section{Proofs}\label{sec5}
We set the following notations.
\begin{enumerate}
\item[1.]
$\GG=\sigma\bigl[\{W_s\}_{s\le t_i^n}\bigr]$.

\item[2.]
$C$, $C_1$, $C_2,\ldots>0$ denote universal constants.

\item[3.]
For a measurable set $A$ and an integrable random variable $X$, we define 
\begin{align*}
\EE[X:A]=\int_AX(\omega)\dd P(\omega).
\end{align*}

\item[4.]
For a function $f$ on $\mathbb R^d\times \Theta$,
we define
$f_{i-1}(\theta)=f(\Xs,\theta)$.

\item[5.]
We define 
\begin{align*}
A\otimes x^{\otimes k}=\sum_{\ell_1,\ldots,\ell_k=1}^{d_1}
A^{\ell_1,\ldots,\ell_k}x^{\ell_1}\cdots x^{\ell_k},\quad
\text{ for }
A\in\underbrace{\mathbb R^{d_1}\otimes\cdots\otimes\mathbb R^{d_1}}_{k},\ x\in\mathbb R^{d_1}.
\end{align*} 

\end{enumerate}

\begin{lem}\label{lem1}
Let
$0\le\tau_1<\tau_2\le1$, where $\tau_1,\tau_2$ may depend on $n$.
Let 
$\{r_n\}_{n=1}^\infty$ be a sequence with
$r_n^2([n\tau_2]-[n\tau_1])h\lto0$,
and
$\{\mathcal M_i\}_{i=1}^n$ be a martingale with $\EE[\|\mathcal M_i\|^2]\le Ch$.
If $[n\tau_1]< k_n\le[n\tau_2]$ and $[n\tau_1]\le \ell_n<[n\tau_2]$
on $\Omega_n$ with $P(\Omega_n)\lto1$, then
\begin{align}\label{eq-Lemma1}
r_n\Biggl\|
\sum_{i=[n\tau_1]+1}^{k_n}\mathcal M_i
\Biggr\|
=o_p(1), 
\quad
r_n\Biggl\|
\sum_{i=\ell_n+1}^{[n\tau_2]}\mathcal M_i
\Biggr\|
=o_p(1).
\end{align}
\end{lem}
\begin{proof}
Let
\begin{align*}
\mathcal S_n
=r_n\Biggl\|
\sum_{i=[n\tau_1]+1}^{k_n}\mathcal M_i
\Biggr\|.
\end{align*}
For all $\epsilon>0$, 
\begin{align}\label{eq-Lemma1-001}
P(\mathcal S_n>\epsilon)
\le
P(\mathcal S_n>\epsilon, \Omega_n)
+P(\Omega_n^{\mathrm c})
\le
\frac{1}{\epsilon^2}\EE[\mathcal S_n^2:\Omega_n]
+P(\Omega_n^{\mathrm c}).
\end{align}
From Burkholder inequality, we have
\begin{align}
\EE[\mathcal S_n^2:\Omega_n]
&\le
r_n^2
\EE\left[
\max_{[n\tau_1]< k\le[n\tau_2]}
\Biggl\|
\sum_{i=[n\tau_1]+1}^k\mathcal M_i
\Biggr\|^2
:\Omega_n\right]
\nonumber
\\
&\le
r_n^2
\EE\left[
\max_{[n\tau_1]< k\le[n\tau_2]}
\Biggl\|
\sum_{i=[n\tau_1]+1}^k\mathcal M_i
\Biggr\|^2
\right]
\nonumber
\\
&\le
Cr_n^2
\sum_{i=[n\tau_1]+1}^{[n\tau_2]}
\EE[\|\mathcal M_i\|^2]
\nonumber
\\
&=
O\bigl(r_n^2([n\tau_2]-[n\tau_1])h\bigr)
=o(1).
\label{eq-Lemma1-002}
\end{align}
Therefore, 
from \eqref{eq-Lemma1-001}, \eqref{eq-Lemma1-002} and 
$P(\Omega_n^{\mathrm c})\lto0$, 
we have the first part of \eqref{eq-Lemma1}.
According to
\begin{align*}
\Biggl\|
\sum_{i=\ell_n+1}^{[n\tau_2]}\mathcal M_i
\Biggr\|^2
&\le
2
\left(
\Biggl\|
\sum_{i=[n\tau_1]+1}^{[n\tau_2]}\mathcal M_i
\Biggr\|^2
+
\Biggl\|
\sum_{i=[n\tau_1]+1}^{\ell_n}\mathcal M_i
\Biggr\|^2
\right),
\end{align*}
the second part of \eqref{eq-Lemma1} is obtained in the same way.
\end{proof}

Let
$\theta_k=(\alpha_k^*,\beta_k)$, 
$\tau_n^L=\tau_*^\alpha-2n^{-\epsilon_1}$,
$\tau_n^U=\tau_*^\alpha+2n^{-\epsilon_1}$,
$m_n=[n\tau_n^L]$ 
and
$M_n=[n\tau_n^U]$.

\begin{proof}[\bf{Proof of Theorem \ref{th1}}]
We first prove \eqref{eq-th1-1}. 
Let
\begin{align*}
\TLA
&=
\frac{1}{\sqrt{d\underline{\tau}_n T}}
\max_{1\le k\le m_n}
\left|
\sum_{i=1}^k\check\xi_{1,i}
-\frac{k}{[n\underline{\tau}_n]}\sum_{i=1}^{[n\underline{\tau}_n]}\check\xi_{1,i}
\right|,\\
\TUA
&=
\frac{1}{\sqrt{d\underline{\tau}_n T}}
\max_{1\le k\le [n\tau_*^\alpha]}
\left|
\sum_{i=1}^k\check\xi_{1,i}
-\frac{k}{[n\underline{\tau}_n]}\sum_{i=1}^{[n\underline{\tau}_n]}\check\xi_{1,i}
\right|
\end{align*}
and $D_n=\{n^{\epsilon_1}|\hat\tau_n^\alpha-\tau_*^\alpha|\le 1\}$.
Note that 
the probability of $D_n$ converges to one from \textbf{[E1]},
and
$m_n \le [n\underline{\tau}_n] \le [n\tau_*^\alpha]
\le [n\overline{\tau}_n] \le M_n$
on $D_n$. 
Since $\TLA\le\TUO\le\TUA$ on $D_n$,
if 
\begin{align}
&\TLA\dto\sup_{0\le s\le 1}|\boldsymbol B_1^0(s)|,
\label{eq-000-1}
\\
&\TUA\dto\sup_{0\le s\le 1}|\boldsymbol B_1^0(s)|,
\label{eq-000-2}
\end{align}
then we have
\begin{align*}
&\varlimsup_{n\to\infty}
P(\TUO\le x, D_n)
\le 
\varlimsup_{n\to\infty}
P(\TLA\le x, D_n)
\le 
\varlimsup_{n\to\infty}
P(\TLA\le x)
\le 
P\left(\sup_{0\le s\le 1}|\boldsymbol B_1^0(s)|\le x\right),
\\
&\varlimsup_{n\to\infty}
P(\TUO> x, D_n)
\le 
\varlimsup_{n\to\infty}
P(\TUA> x, D_n)
\le 
\varlimsup_{n\to\infty}
P(\TUA> x)
\le 
P\left(\sup_{0\le s\le 1}|\boldsymbol B_1^0(s)|> x\right).
\end{align*}
Since 
\begin{align*}
\varlimsup_{n\to\infty}
P(\{\TUO> x\}\cup D_n^{\mathrm c})
\le
\varlimsup_{n\to\infty}
P(\TUO> x, D_n),
\end{align*}
we see
\begin{align*}
P\left(\sup_{0\le s\le 1}|\boldsymbol B_1^0(s)|\le x\right)
&\le
\varliminf_{n\to\infty}
P(\TUO\le x, D_n)\\
&\le
\varlimsup_{n\to\infty}
P(\TUO\le x, D_n)
\le
P\left(\sup_{0\le s\le 1}|\boldsymbol B_1^0(s)|\le x\right),
\end{align*}
i.e.,
\begin{align}\label{eq-000-3}
\lim_{n\to\infty}
P(\TUO\le x, D_n)
=
P\left(\sup_{0\le s\le 1}|\boldsymbol B_1^0(s)|\le x\right).
\end{align}
Hence, from \eqref{eq-000-3}, $P(D_n)\lto1$ and
$P(\TUO\le x)
=
P(\TUO\le x, D_n)
+P(\TUO\le x, D_n^{\mathrm c})$, 
we obtain
\begin{align*}
\TUO\dto\sup_{0\le s\le 1}|\boldsymbol B_1^0(s)|.
\end{align*}
From the above, it suffices to show \eqref{eq-000-1} and \eqref{eq-000-2}.


\textit{Proof of \eqref{eq-000-1}. }
We have 
\begin{align}
&\sum_{i=1}^k\check\xi_{1,i}
-\frac{k}{[n\underline{\tau}_n]}\sum_{i=1}^{[n\underline{\tau}_n]}\check\xi_{1,i}
\nonumber
\\
&=
\sum_{i=1}^k\check\xi_{1,i}
-\frac{k}{m_n}\sum_{i=1}^{m_n}\check\xi_{1,i}
+\frac{k}{m_n}\left(1-\frac{m_n}{[n\underline{\tau}_n]}\right)
\sum_{i=1}^{m_n}\check\xi_{1,i}
-\frac{k}{[n\underline{\tau}_n]}\sum_{i=m_n+1}^{[n\underline{\tau}_n]}\check\xi_{1,i}.
\label{eq-000-4}
\end{align}
Let 
$\xi_{k,i}=1_d^\TT a_{i-1}^{-1}(\alpha_k^*)(\DeX-hb_{i-1}(\beta_k))$,
$\mathcal M_{k,i}=\xi_{k,i}-\EE_{\theta_k}[\xi_{k,i}|\GG]$.
Here, noting that 
\begin{align*}
\check\xi_{k,i}
&=1_d^\TT a_{i-1}^{-1}(\hat\alpha_k)(\DeX-hb_{i-1}(\hat\beta_k))\\
&=1_d^\TT a_{i-1}^{-1}(\alpha_k^*)(\DeX-hb_{i-1}(\hat\beta_k))
\\
&\qquad+
\frac{1}{\sqrt n}\int_0^1\left.
\partial_{\alpha}\left(1_d^\TT a_{i-1}^{-1}(\alpha)(\DeX-hb_{i-1}(\hat\beta_k))\right)
\right|_{\alpha=\alpha_k^*+u(\hat\alpha_k-\alpha_k^*)}
\dd u
\sqrt n(\hat\alpha_k-\alpha_k^*)\\
&=1_d^\TT a_{i-1}^{-1}(\alpha_k^*)(\DeX-hb_{i-1}(\beta_k))
-h 1_d^\TT a_{i-1}^{-1}(\alpha_k^*)(b_{i-1}(\hat\beta_k)-b_{i-1}(\beta_k))\\
&\qquad+
\frac{1}{\sqrt n}\int_0^1\left.
\partial_{\alpha}\left(1_d^\TT a_{i-1}^{-1}(\alpha)(\DeX-hb_{i-1}(\hat\beta_k))\right)
\right|_{\alpha=\alpha_k^*+u(\hat\alpha_k-\alpha_k^*)}
\dd u
\sqrt n(\hat\alpha_k-\alpha_k^*)\\
&=
\xi_{k,i}
-\frac{h}{\sqrt T} 
\int_0^1 \partial_{\beta}
\left.
\Bigl(
1_d^\TT a_{i-1}^{-1}(\alpha_k^*) (b_{i-1}(\beta)-b_{i-1}(\beta_k))
\Bigr)
\right|_{\beta=\beta_k+u(\hat\beta_k-\beta_k)}
\dd u
\sqrt T(\hat\beta_k-\beta_k)\\
&\qquad+
\frac{1}{\sqrt n}\int_0^1\left.
\partial_{\alpha}\left(1_d^\TT a_{i-1}^{-1}(\alpha)(\DeX-hb_{i-1}(\hat\beta_k))\right)
\right|_{\alpha=\alpha_k^*+u(\hat\alpha_k-\alpha_k^*)}
\dd u
\sqrt n(\hat\alpha_k-\alpha_k^*)
\\
&=
\mathcal M_{k,i}+\EE_{\theta_k}[\xi_{k,i}|\GG]
+O_p\Biggl(\sqrt{\frac{h}{n}}\Biggr)
\\
&=
\mathcal M_{k,i}
+O_p\Biggl(\sqrt{\frac{h}{n}}\Biggr),
\end{align*}
we have
\begin{align}
&\frac{1}{\sqrt{T}}
\max_{1\le k \le m_n}
\left|
\frac{k}{m_n}\left(1-\frac{m_n}{[n\underline{\tau}_n]}\right)
\sum_{i=1}^{m_n}\check\xi_{1,i}
\right|
\nonumber
\\
&\le
\frac{1}{\sqrt{T}}
\frac{|[n\underline{\tau}_n]-m_n|}{[n\underline{\tau}_n]}
\left|
\sum_{i=1}^{m_n}\check\xi_{1,i}
\right|
\nonumber
\\
&\le
\frac{1}{\sqrt{T}}
\frac{|[n\underline{\tau}_n]-m_n|}{[n\underline{\tau}_n]}
\left(
\left|
\sum_{i=1}^{m_n}\mathcal M_{1,i}
\right|
+
m_n O_p\Biggl(\sqrt{\frac{h}{n}}\Biggr)
\right)
\nonumber
\\
&=
\frac{m_n}{[n\underline{\tau}_n]}
n^{\epsilon_1-1}|[n\underline{\tau}_n]-m_n|
\left(
\frac{n^{1-\epsilon_1}}{\sqrt{T}m_n}
\left|
\sum_{i=1}^{m_n}\mathcal M_{1,i}
\right|
+
\frac{n^{1-\epsilon_1}}{\sqrt T}
O_p\Biggl(\sqrt{\frac{h}{n}}\Biggr)
\right).
\label{eq-000-5}
\end{align}
Since
$\dfrac{m_n}{[n\underline{\tau}_n]}n^{\epsilon_1-1}|[n\underline{\tau}_n]-m_n|=O_p(1)$,
$\dfrac{n^{1-\epsilon_1}}{\sqrt T }\sqrt{\dfrac{h}{n}}=n^{-\epsilon_1}\lto0$, 
$\EE_{\theta_1}[\mathcal M_{1,i}^2]\le Ch$
and
\begin{align*}
\frac{n^{2-2\epsilon_1}m_nh}{Tm_n^2}
=O(n^{-2\epsilon_1})
=o(1),
\end{align*}
we see
\begin{align}\label{eq-000-6}
&\frac{1}{\sqrt{T}}
\max_{1\le k \le m_n}
\left|
\frac{k}{m_n}\left(1-\frac{m_n}{[n\underline{\tau}_n]}\right)
\sum_{i=1}^{m_n}\check\xi_{1,i}
\right|
=o_p(1)
\end{align}
from Lemma \ref{lem1} and \eqref{eq-000-5}.
In the same way as \eqref{eq-000-5}, we have 
\begin{align*}
\frac{1}{\sqrt{T}}
\max_{1\le k\le m_n}
\left|
\frac{k}{[n\underline{\tau}_n]}\sum_{i=m_n+1}^{[n\underline{\tau}_n]}\check\xi_{1,i}
\right|
\le
\frac{m_n}{[n\underline{\tau}_n]}
\frac{1}{\sqrt{T}}
\left|
\sum_{i=m_n+1}^{[n\underline{\tau}_n]}\mathcal M_{1,i}
\right|
+o_p(1).
\end{align*}
Since $[n\underline{\tau}_n]\le [n\tau_*^\alpha]$ on $D_n$ and
\begin{align*}
\frac{([n\tau_*^\alpha]-m_n)h}{T}
=O(n^{-\epsilon_1})
=o(1),
\end{align*}
we obtain
\begin{align}\label{eq-000-7}
\frac{1}{\sqrt{T}}
\max_{1\le k\le m_n}
\left|
\frac{k}{[n\underline{\tau}_n]}\sum_{i=m_n+1}^{[n\underline{\tau}_n]}\check\xi_{1,i}
\right|
=o_p(1)
\end{align}
from Lemma \ref{lem1}.
According to \eqref{eq-000-4}, \eqref{eq-000-6} and \eqref{eq-000-7}, 
we can express that
\begin{align*}
\TLA
&=
\frac{1}{\sqrt{d\underline{\tau}_n T}}
\max_{1\le k\le m_n}
\left|
\sum_{i=1}^k\check\xi_{1,i}
-\frac{k}{[n\underline{\tau}_n]}\sum_{i=1}^{[n\underline{\tau}_n]}\check\xi_{1,i}
\right|
\\
&=
\frac{1}{\sqrt{d\underline{\tau}_n T}}
\max_{1\le k\le m_n}
\left|
\sum_{i=1}^k\check\xi_{1,i}
-\frac{k}{m_n}\sum_{i=1}^{m_n}\check\xi_{1,i}
\right|
+o_p(1).
\end{align*}
In the same proof as Theorem 2 of Tonaki et al. (2020),
we obtain
\begin{align*}
&\frac{1}{\sqrt{d\underline{\tau}_n T}}
\max_{1\le k\le m_n}
\left|
\sum_{i=1}^k\check\xi_{1,i}
-\frac{k}{m_n}\sum_{i=1}^{m_n}\check\xi_{1,i}
\right|
\\
&=
\sqrt{\frac{\tau_n^L}{\underline{\tau}_n}}
\frac{1}{\sqrt{d\tau_n^L T}}
\max_{1\le k\le m_n}
\left|
\sum_{i=1}^k\check\xi_{1,i}
-\frac{k}{m_n}\sum_{i=1}^{m_n}\check\xi_{1,i}
\right|
\\
&
\dto
\sup_{0\le s\le 1}|\boldsymbol B_1^0(s)|,
\end{align*}
which concludes the proof of \eqref{eq-000-1}. 

\textit{Proof of \eqref{eq-000-2}. }
We have
\begin{align*}
&\sum_{i=1}^k\check\xi_{1,i}
-\frac{k}{[n\underline{\tau}_n]}\sum_{i=1}^{[n\underline{\tau}_n]}\check\xi_{1,i}
\nonumber
\\
&=
\sum_{i=1}^k\check\xi_{1,i}
-\frac{k}{[n\tau_*^\alpha]}\sum_{i=1}^{[n\tau_*^\alpha]}\check\xi_{1,i}
+\frac{k}{[n\tau_*^\alpha]}
\left(1-\frac{[n\tau_*^\alpha]}{[n\underline{\tau}_n]}\right)
\sum_{i=1}^{[n\tau_*^\alpha]}\check\xi_{1,i}
+\frac{k}{[n\underline{\tau}_n]}
\sum_{i=[n\underline{\tau}_n]+1}^{[n\tau_*^\alpha]}\check\xi_{1,i}.
\end{align*}
Therefore, by the same argument, we obtain
\begin{align*}
\TUA
&=
\frac{1}{\sqrt{d\underline{\tau}_n T}}
\max_{1\le k\le [n\tau_*^\alpha]}
\left|
\sum_{i=1}^k\check\xi_{1,i}
-\frac{k}{[n\underline{\tau}_n]}\sum_{i=1}^{[n\underline{\tau}_n]}\check\xi_{1,i}
\right|
\\
&=
\frac{1}{\sqrt{d\underline{\tau}_n T}}
\max_{1\le k\le [n\tau_*^\alpha]}
\left|
\sum_{i=1}^k\check\xi_{1,i}
-\frac{k}{[n\tau_*^\alpha]}\sum_{i=1}^{[n\tau_*^\alpha]}\check\xi_{1,i}
\right|
+o_p(1)
\\
&=
\sqrt{\frac{\tau_*^\alpha}{\underline{\tau}_n}}
\frac{1}{\sqrt{d\tau_*^\alpha T}}
\max_{1\le k\le [n\tau_*^\alpha]}
\left|
\sum_{i=1}^k\check\xi_{1,i}
-\frac{k}{[n\tau_*^\alpha]}\sum_{i=1}^{[n\tau_*^\alpha]}\check\xi_{1,i}
\right|
+o_p(1)
\\
&
\dto
\sup_{0\le s\le 1}|\boldsymbol B_1^0(s)|.
\end{align*}


Next, we prove \eqref{eq-th1-2}.
Let
\begin{align*}
\TLB
&=
\frac{1}{\sqrt{d(1-\overline{\tau}_n) T}}
\max_{1\le k\le n-M_n}
\left|
\sum_{i=[n\overline{\tau}_n]+1}^{[n\overline{\tau}_n]+k}\check\xi_{2,i}
-\frac{k}{n-[n\overline{\tau}_n]}
\sum_{i=[n\overline{\tau}_n]+1}^{n}\check\xi_{2,i}
\right|,
\\
\TUB
&=
\frac{1}{\sqrt{d(1-\overline{\tau}_n) T}}
\max_{1\le k\le n-[n\tau_*^\alpha]}
\left|
\sum_{i=[n\overline{\tau}_n]+1}^{[n\overline{\tau}_n]+k}\check\xi_{2,i}
-\frac{k}{n-[n\overline{\tau}_n]}
\sum_{i=[n\overline{\tau}_n]+1}^{n}\check\xi_{2,i}
\right|.
\end{align*}
Since $\TLB\le\TOO\le\TUB$ on $D_n$, 
it is enough to show
\begin{align}
\TLB\dto\sup_{0\le s\le 1}|\boldsymbol B_1^0(s)|,
\label{eq-000-8}
\\
\TUB\dto\sup_{0\le s\le 1}|\boldsymbol B_1^0(s)|
\label{eq-000-9}
\end{align}
similar to the proof of \eqref{eq-th1-1}.

\textit{Proof of \eqref{eq-000-8}. }
We can express that
\begin{align}
&\sum_{i=[n\overline{\tau}_n]+1}^{[n\overline{\tau}_n]+k}\check\xi_{2,i}
-\frac{k}{n-[n\overline{\tau}_n]}
\sum_{i=[n\overline{\tau}_n]+1}^{n}\check\xi_{2,i}
\nonumber
\\
&=
\sum_{i=M_n+1}^{M_n+k}\check\xi_{2,i}
+\sum_{i=[n\overline{\tau}_n]+1}^{M_n}\check\xi_{2,i}
-\sum_{i=[n\overline{\tau}_n]+k+1}^{M_n+k}\check\xi_{2,i}
\nonumber
\\
&
\qquad-
\frac{k}{n-M_n}
\sum_{i=M_n+1}^{n}\check\xi_{2,i}
+\frac{k}{n-M_n}
\left(1-\frac{n-M_n}{n-[n\overline{\tau}_n]}\right)
\sum_{i=M_n+1}^{n}\check\xi_{2,i}
-\frac{k}{n-[n\overline{\tau}_n]}
\sum_{i=[n\overline{\tau}_n]+1}^{M_n}\check\xi_{2,i}
\nonumber
\\
&=
\sum_{i=M_n+1}^{M_n+k}\check\xi_{2,i}
-\frac{k}{n-M_n}
\sum_{i=M_n+1}^{n}\check\xi_{2,i}
-\sum_{i=[n\overline{\tau}_n]+k+1}^{M_n+k}\check\xi_{2,i}
\nonumber
\\
&
\qquad+
\frac{k}{n-M_n}
\left(1-\frac{n-M_n}{n-[n\overline{\tau}_n]}\right)
\sum_{i=M_n+1}^{n}\check\xi_{2,i}
+\left(1-\frac{k}{n-[n\overline{\tau}_n]}\right)
\sum_{i=[n\overline{\tau}_n]+1}^{M_n}\check\xi_{2,i}.
\label{eq-000-9'}
\end{align}
Here, we have
\begin{align}
&\frac{1}{\sqrt{T}}
\max_{1\le k \le n-M_n}
\left|
\frac{k}{n-M_n}\left(1-\frac{n-M_n}{n-[n\overline{\tau}_n]}\right)
\sum_{i=M_n+1}^{n}\check\xi_{2,i}
\right|
\nonumber
\\
&\le
\frac{1}{\sqrt{T}}
\frac{|[n\overline{\tau}_n]-M_n|}{n-[n\overline{\tau}_n]}
\left|
\sum_{i=M_n+1}^{n}\check\xi_{2,i}
\right|
\nonumber
\\
&\le
\frac{1}{\sqrt{T}}
\frac{|[n\overline{\tau}_n]-M_n|}{n-[n\overline{\tau}_n]}
\left(
\left|
\sum_{i=M_n+1}^{n}\mathcal M_{2,i}
\right|
+
(n-M_n)O_p\Biggl(\sqrt\frac{h}{n}\Biggr)
\right)
\nonumber
\\
&=
\frac{n-M_n}{n-[n\overline{\tau}_n]}
n^{\epsilon_1-1}|[n\overline{\tau}_n]-M_n|
\left(
\frac{n^{1-\epsilon_1}}{\sqrt{T}(n-M_n)}
\left|
\sum_{i=M_n+1}^n\mathcal M_{2,i}
\right|
+
\frac{n^{1-\epsilon_1}}{\sqrt{T}}
O_p\Biggl(\sqrt{\frac{h}{n}}\Biggr)
\right).
\label{eq-000-10}
\end{align}
Since
$\dfrac{n-M_n}{n-[n\overline{\tau}_n]}n^{\epsilon_1-1}
|[n\overline{\tau}_n]-M_n|=O_p(1)$,
$\EE_{\theta_2}[\mathcal M_{2,i}^2]\le Ch$ 
and
\begin{align*}
\frac{n^{2-2\epsilon_1}h(n-M_n)}{T(n-M_n)^2}
=O(n^{-2\epsilon_1})
=o(1),
\end{align*}
we have
\begin{align}\label{eq-000-11}
\frac{1}{\sqrt{T}}
\max_{1\le k \le n-M_n}
\left|
\frac{k}{n-M_n}\left(1-\frac{n-M_n}{n-[n\overline{\tau}_n]}\right)
\sum_{i=M_n+1}^{n}\check\xi_{2,i}
\right|
=o_p(1)
\end{align}
from Lemma \ref{lem1} and \eqref{eq-000-10}.
Similarly, we see
\begin{align*}
\frac{1}{\sqrt{T}}
\max_{1\le k\le n-M_n}
\left|
\left(1-\frac{k}{n-[n\overline{\tau}_n]}\right)
\sum_{i=[n\overline{\tau}_n]+1}^{M_n}\check\xi_{2,i}
\right|
\le
\frac{1}{\sqrt{T}}
\left|
\sum_{i=[n\overline{\tau}_n]+1}^{M_n}\mathcal M_{2,i}
\right|
+o_p(1).
\end{align*}
Since $[n\tau_*^\alpha]\le [n\overline{\tau}_n]$ on $D_n$ and
\begin{align*}
\frac{h(M_n-[n\tau_*^\alpha])}{T}
=O(n^{-\epsilon_1})
=o(1),
\end{align*}
we obtain
\begin{align}\label{eq-000-12}
\frac{1}{\sqrt{T}}
\max_{1\le k\le n-M_n}
\left|
\left(1-\frac{k}{n-[n\overline{\tau}_n]}\right)
\sum_{i=[n\overline{\tau}_n]+1}^{M_n}\check\xi_{2,i}
\right|
=o_p(1)
\end{align}
from Lemma \ref{lem1}.
In the same way as \eqref{eq-000-10}, we have
\begin{align}
\frac{1}{\sqrt{T}}
\max_{1\le k\le n-M_n}
\left|
\sum_{i=[n\overline{\tau}_n]+k+1}^{M_n+k}\check\xi_{2,i}
\right|
&\le
\frac{1}{\sqrt{T}}
\max_{1\le k\le n-M_n}
\left|
\sum_{i=[n\overline{\tau}_n]+k+1}^{M_n+k}\mathcal M_{2,i}
\right|
+o_p(1)
\nonumber
\\
&=:
\Qn+o_p(1).
\label{eq-000-13}
\end{align}
For all $\epsilon>0$, 
\begin{align}\label{eq-000-14}
P(\Qn>2\epsilon)
\le
P(\Qn>2\epsilon, D_n)+P(D_n^{\mathrm c}).
\end{align}
Here, the first term on the right hand side can be transformed as follows.
\begin{align}
&P(\Qn>2\epsilon,D_n)
\nonumber
\\
&\le
P\left(
\frac{1}{\sqrt{T}}
\max_{[n\tau_*^\alpha]\le \ell< M_n}
\max_{1\le k\le n-M_n}
\left|
\sum_{i=\ell+k+1}^{M_n+k}\Mi
\right|>2\epsilon, D_n
\right)
\nonumber
\\
&\le
P\left(
\frac{1}{\sqrt{T}}
\max_{[n\tau_*^\alpha]\le \ell< M_n}
\max_{1\le k\le n-M_n}
\left|
\sum_{i=\ell+k+1}^{M_n+k}\Mi
\right|>2\epsilon
\right)
\nonumber
\\
&=
P\left(
\frac{1}{\sqrt{T}}
\max_{[n\tau_*^\alpha]< \ell< M_n}
\max_{1\le k\le n-M_n}
\left|
\sum_{i=[n\tau_*^\alpha]+k+1}^{M_n+k}\Mi
-
\sum_{i=[n\tau_*^\alpha]+k+1}^{\ell+k}\Mi
\right|>2\epsilon
\right)
\nonumber
\\
&\le
P\left(
\frac{1}{\sqrt{T}}
\max_{1\le k\le n-M_n}
\left|
\sum_{i=[n\tau_*^\alpha]+k+1}^{M_n+k}\Mi
\right|
+
\frac{1}{\sqrt{T}}
\max_{[n\tau_*^\alpha]< \ell\le M_n}
\max_{1\le k\le n-M_n}
\left|
\sum_{i=[n\tau_*^\alpha]+k+1}^{\ell+k}\Mi
\right|
>2\epsilon
\right)
\nonumber
\\
&\le
P\left(
\frac{1}{\sqrt{T}}
\max_{1\le k\le n-M_n}
\left|
\sum_{i=[n\tau_*^\alpha]+k+1}^{M_n+k}\Mi
\right|
>\epsilon
\right)
\nonumber
\\
&\qquad+
P\left(
\frac{1}{\sqrt{T}}
\max_{[n\tau_*^\alpha]< \ell\le M_n}
\max_{1\le k\le n-M_n}
\left|
\sum_{i=[n\tau_*^\alpha]+k+1}^{\ell+k}\Mi
\right|
>\epsilon
\right).
\label{eq-000-15}
\end{align}
We choose $r>\frac{2-\epsilon_1}{\epsilon_1}$.
Noting that
$\epsilon_1r>2-\epsilon_1>1$,
we see,
from Theorem 2.11 of Hall and Heyde (1980),
convex inequality and 
$\EE_{\theta_2}[\mathcal M_{2,i}^{2r}]\le Ch^r$,
\begin{align}
&P\left(
\frac{1}{\sqrt{T}}
\max_{1\le k\le n-M_n}
\left|
\sum_{i=[n\tau_*^\alpha]+k+1}^{M_n+k}\Mi
\right|
>\epsilon
\right)
\nonumber
\\
&=
P\Biggl(
\bigcup_{k=1}^{n-M_n}
\Biggl\{
\frac{1}{\sqrt{T}}
\Biggl|
\sum_{i=[n\tau_*^\alpha]+k+1}^{M_n+k}\Mi
\Biggr|
>\epsilon
\Biggr\}
\Biggr)
\nonumber
\\
&\le
\sum_{k=1}^{n-M_n}
P\Biggl(
\frac{1}{\sqrt{T}}
\Biggl|
\sum_{i=[n\tau_*^\alpha]+k+1}^{M_n+k}\Mi
\Biggr|
>\epsilon
\Biggr)
\nonumber
\\
&\le
\sum_{k=1}^{n-M_n}
\frac{1}{T^{r}\epsilon^{2r}}
\EE_{\theta_2}\Biggl[
\Biggl|
\sum_{i=[n\tau_*^\alpha]+k+1}^{M_n+k}\Mi
\Biggr|^{2r}
\Biggr]
\nonumber
\\
&\le
\sum_{k=1}^{n-M_n}
\frac{C_1}{T^{r}\epsilon^{2r}}
\EE_{\theta_2}\Biggl[
\Biggl(
\sum_{i=[n\tau_*^\alpha]+k+1}^{M_n+k}\Mi^2
\Biggr)^r
\Biggr]
\nonumber
\\
&\le
\sum_{k=1}^{n-M_n}
\frac{C_2}{T^{r}\epsilon^{2r}}
[n^{1-\epsilon_1}]^{r}
\frac{1}{[n^{1-\epsilon_1}]}
\sum_{i=[n\tau_*^\alpha]+k+1}^{M_n+k}
\EE_{\theta_2}[\Mi^{2r}]
\nonumber
\\
&=
O\left(
\frac{n^{1+r-\epsilon_1r}h^r}{T^r}
\right)
=
O(n^{1-\epsilon_1 r})
=o(1)
\label{eq-000-16}
\end{align}
and
\begin{align}
&P\left(
\frac{1}{\sqrt{T}}
\max_{[n\tau_*^\alpha]< \ell\le M_n}
\max_{1\le k\le n-M_n}
\left|
\sum_{i=[n\tau_*^\alpha]+k+1}^{\ell+k}\Mi
\right|
>\epsilon
\right)
\nonumber
\\
&=
P\Biggl(
\bigcup_{\ell=[n\tau_*^\alpha]+1}^{M_n}
\bigcup_{k=1}^{n-M_n}
\Biggl\{
\frac{1}{\sqrt{T}}
\Biggl|
\sum_{i=[n\tau_*^\alpha]+k+1}^{\ell+k}\Mi
\Biggr|
>\epsilon
\Biggr\}
\Biggr)
\nonumber
\\
&\le
\sum_{\ell=[n\tau_*^\alpha]+1}^{M_n}
\sum_{k=1}^{n-M_n}
P\Biggl(
\frac{1}{\sqrt{T}}
\Biggl|
\sum_{i=[n\tau_*^\alpha]+k+1}^{\ell+k}\Mi
\Biggr|
>\epsilon
\Biggr)
\nonumber
\\
&\le
\sum_{\ell=[n\tau_*^\alpha]+1}^{M_n}
\sum_{k=1}^{n-M_n}
\frac{1}{T^r\epsilon^{2r}}
\EE_{\theta_2}\Biggl[
\Biggl|
\sum_{i=[n\tau_*^\alpha]+k+1}^{\ell+k}\Mi
\Biggr|^{2r}
\Biggr]
\nonumber
\\
&\le
\sum_{\ell=[n\tau_*^\alpha]+1}^{M_n}
\sum_{k=1}^{n-M_n}
\frac{C_1}{T^r\epsilon^{2r}}
\EE_{\theta_2}\Biggl[
\Biggl(
\sum_{i=[n\tau_*^\alpha]+k+1}^{M_n+k}\Mi^2
\Biggr)^r
\Biggr]
\nonumber
\\
&\le
\sum_{\ell=[n\tau_*^\alpha]+1}^{M_n}
\sum_{k=1}^{n-M_n}
\frac{C_2}{T^r\epsilon^{2r}}
[n^{1-\epsilon_1}]^p
\frac{1}{[n^{1-\epsilon_1}]}
\sum_{i=[n\tau_*^\alpha]+k+1}^{M_n+k}
\EE_{\theta_2}[\Mi^{2r}]
\nonumber
\\
&\le
\sum_{\ell=[n\tau_*^\alpha]+1}^{M_n}
\sum_{k=1}^{n-M_n}
\frac{C_3}{T^r\epsilon^{2r}}
[n^{1-\epsilon_1}]^r h^r
\nonumber
\\
&=O(n^{1-\epsilon_1+1-\epsilon_1 r})
=O(n^{2-\epsilon_1-\epsilon_1 r})
=o(1).
\label{eq-000-17}
\end{align}
According to \eqref{eq-000-14}-\eqref{eq-000-17} and $P(D_n^{\mathrm c})\lto0$,
we have $\Qn=o_p(1)$.
Therefore, from \eqref{eq-000-13},
\begin{align}\label{eq-000-18}
\frac{1}{\sqrt{T}}
\max_{1\le k\le n-M_n}
\left|
\sum_{i=[n\overline{\tau}_n]+k+1}^{M_n+k}\check\xi_{2,i}
\right|
=o_p(1).
\end{align}
From \eqref{eq-000-9'}, \eqref{eq-000-11}, \eqref{eq-000-12} and \eqref{eq-000-18}, 
we obtain
\begin{align*}
\TLB
&=
\frac{1}{\sqrt{d(1-\overline{\tau}_n) T}}
\max_{1\le k\le n-M_n}
\left|
\sum_{i=[n\overline{\tau}_n]+1}^{[n\overline{\tau}_n]+k}\check\xi_{2,i}
-\frac{k}{n-[n\overline{\tau}_n]}
\sum_{i=[n\overline{\tau}_n]+1}^{n}\check\xi_{2,i}
\right|
\\
&=
\frac{1}{\sqrt{d(1-\overline{\tau}_n) T}}
\max_{1\le k\le n-M_n}
\left|
\sum_{i=M_n+1}^{M_n+k}\check\xi_{2,i}
-\frac{k}{n-M_n}
\sum_{i=M_n+1}^{n}\check\xi_{2,i}
\right|
+o_p(1)
\\
&=
\sqrt\frac{1-\tau_n^U}{1-\overline{\tau}_n}
\frac{1}{\sqrt{d(1-\tau_n^U) T}}
\max_{1\le k\le n-M_n}
\left|
\sum_{i=M_n+1}^{M_n+k}\check\xi_{2,i}
-\frac{k}{n-M_n}
\sum_{i=M_n+1}^{n}\check\xi_{2,i}
\right|
+o_p(1)
\\
&\dto\sup_{0\le s\le 1}|\boldsymbol B_1^0(s)|.
\end{align*}

\textit{Proof of \eqref{eq-000-9}. }
Since it follows that
\begin{align*}
&\sum_{i=[n\overline{\tau}_n]+1}^{[n\overline{\tau}_n]+k}\check\xi_{2,i}
-\frac{k}{n-[n\overline{\tau}_n]}
\sum_{i=[n\overline{\tau}_n]+1}^{n}\check\xi_{2,i}
\\
&=
\sum_{i=[n\tau_*^\alpha]+1}^{[n\tau_*^\alpha]+k}\check\xi_{2,i}
+\sum_{i=[n\tau_*^\alpha]+k+1}^{[n\overline{\tau}_n]+k}\check\xi_{2,i}
-\sum_{i=[n\tau_*^\alpha]+1}^{[n\overline{\tau}_n]}\check\xi_{2,i}
\\
&\qquad
-\frac{k}{n-[n\tau_*^\alpha]}
\sum_{i=[n\tau_*^\alpha]+1}^{n}\check\xi_{2,i}
+\frac{k}{n-[n\tau_*^\alpha]}
\left(1-\frac{n-[n\tau_*^\alpha]}{n-[n\overline{\tau}_n]}\right)
\sum_{i=m_n+1}^{n}\check\xi_{2,i}
+\frac{k}{n-[n\overline{\tau}_n]}
\sum_{i=[n\tau_*^\alpha]+1}^{[n\overline{\tau}_n]}\check\xi_{2,i}
\\
&=
\sum_{i=[n\tau_*^\alpha]+1}^{[n\tau_*^\alpha]+k}\check\xi_{2,i}
-\frac{k}{n-[n\tau_*^\alpha]}
\sum_{i=[n\tau_*^\alpha]+1}^{n}\check\xi_{2,i}
+\sum_{i=[n\tau_*^\alpha]+k+1}^{[n\overline{\tau}_n]+k}\check\xi_{2,i}
\\
&\qquad
+\frac{k}{n-[n\tau_*^\alpha]}
\left(1-\frac{n-[n\tau_*^\alpha]}{n-[n\overline\tau_n]}\right)
\sum_{i=[n\tau_*^\alpha]+1}^{n}\check\xi_{2,i}
-\left(1-\frac{k}{n-[n\overline{\tau}_n]}\right)
\sum_{i=[n\tau_*^\alpha]+1}^{[n\overline{\tau}_n]}\check\xi_{2,i},
\end{align*}
we obtain, by the same argument, 
\begin{align*}
\TUB
&=
\frac{1}{\sqrt{d(1-\overline{\tau}_n) T}}
\max_{1\le k\le n-[n\tau_*^\alpha]}
\left|
\sum_{i=[n\overline{\tau}_n]+1}^{[n\overline{\tau}_n]+k}\check\xi_{2,i}
-\frac{k}{n-[n\overline{\tau}_n]}
\sum_{i=[n\overline{\tau}_n]+1}^{n}\check\xi_{2,i}
\right|
\\
&=
\frac{1}{\sqrt{d(1-\overline{\tau}_n) T}}
\max_{1\le k\le n-[n\tau_*^\alpha]}
\left|
\sum_{i=[n\tau_*^\alpha]+1}^{[n\tau_*^\alpha]+k}\check\xi_{2,i}
-\frac{k}{n-[n\tau_*^\alpha]}
\sum_{i=[n\tau_*^\alpha]+1}^{n}\check\xi_{2,i}
\right|
+o_p(1)
\\
&=
\sqrt\frac{1-\tau_*^\alpha}{1-\overline{\tau}_n}
\frac{1}{\sqrt{d(1-\tau_*^\alpha) T}}
\max_{1\le k\le n-[n\tau_*^\alpha]}
\left|
\sum_{i=[n\tau_*^\alpha]+1}^{[n\tau_*^\alpha]+k}\check\xi_{2,i}
-\frac{k}{n-[n\tau_*^\alpha]}
\sum_{i=[n\tau_*^\alpha]+1}^{n}\check\xi_{2,i}
\right|
+o_p(1)
\\
&\dto\sup_{0\le s\le 1}|\boldsymbol B_1^0(s)|.
\end{align*}
\end{proof}
\begin{proof}[\bf{Proof of Theorem \ref{th2}}]
Let
\begin{align*}
\TLC
&=
\frac{1}{\sqrt{\underline{\tau}_n T}}
\max_{1\le k\le m_n}
\left\|
\mathcal I_{1,n}^{-1/2}
\left(
\sum_{i=1}^k\check\zeta_{1,i}
-\frac{k}{[n\underline{\tau}_n]}\sum_{i=1}^{[n\underline{\tau}_n]}\check\zeta_{1,i}
\right)
\right\|,
\\
\TUC
&=
\frac{1}{\sqrt{\underline{\tau}_n T}}
\max_{1\le k\le [n\tau_*^\alpha]}
\left\|
\mathcal I_{1,n}^{-1/2}
\left(
\sum_{i=1}^k\check\zeta_{1,i}
-\frac{k}{[n\underline{\tau}_n]}\sum_{i=1}^{[n\underline{\tau}_n]}\check\zeta_{1,i}
\right)
\right\|.
\end{align*}
On $D_n$, we have $\TLC\le\TUT\le\TUC$.
Similar to the proof of \eqref{eq-th1-1}, it is enough to show
\begin{align}
\TLC\dto\sup_{0\le s\le 1}\|\boldsymbol B_q^0(s)\|,
\label{eq-001-1}
\\
\TUC\dto\sup_{0\le s\le 1}\|\boldsymbol B_q^0(s)\|.
\label{eq-001-2}
\end{align}

\textit{Proof of \eqref{eq-001-1}. }
We have
\begin{align}
&\sum_{i=1}^k\check\zeta_{1,i}
-\frac{k}{[n\underline{\tau}_n]}\sum_{i=1}^{[n\underline{\tau}_n]}\check\zeta_{1,i}
\nonumber
\\
&=
\sum_{i=1}^k\check\zeta_{1,i}
-\frac{k}{m_n}\sum_{i=1}^{m_n}\check\zeta_{1,i}
+\frac{k}{m_n}\left(1-\frac{m_n}{[n\underline{\tau}_n]}\right)\sum_{i=1}^{m_n}\check\zeta_{1,i}
-\frac{k}{[n\underline{\tau}_n]}\sum_{i=m_n+1}^{[n\underline{\tau}_n]}\check\zeta_{1,i}.
\label{eq-001-3}
\end{align}
Let
\begin{align*}
\mathcal Z_{k,i}^{[j]}
&=
\frac{1}{j!}
\left.\partial_\beta^j
\Bigl(\partial_\beta b_{i-1}(\beta)^\TT 
A_{i-1}^{-1}(\alpha_k^*)(\DeX-hb_{i-1}(\beta_k))
\Bigr)
\right|_{\beta=\beta_k},
\\
\mathcal N_{k,i}^{[j]}
&=
\mathcal Z_{k,i}^{[j]}
-\EE_{\theta_k}[\mathcal Z_{k,i}^{[j]}|\GG].
\end{align*}
Here, taking into account that 
\begin{align*}
\check\zeta_{k,i}
&=\partial_\beta b_{i-1}(\hat\beta_k)^\TT 
A_{i-1}^{-1}(\hat\alpha_k)(\DeX-hb_{i-1}(\hat\beta_k))\\
&=\partial_\beta b_{i-1}(\hat\beta_k)^\TT 
A_{i-1}^{-1}(\alpha_k^*)(\DeX-hb_{i-1}(\hat\beta_k))
\\
&\qquad+
\frac{1}{\sqrt n}\int_0^1\left.
\partial_{\alpha}\left(
\partial_\beta b_{i-1}(\hat\beta_k)^\TT 
A_{i-1}^{-1}(\alpha)(\DeX-hb_{i-1}(\hat\beta_k))\right)
\right|_{\alpha=\alpha_k^*+u(\hat\alpha_k-\alpha_k^*)}
\dd u
\sqrt n(\hat\alpha_k-\alpha_k^*)\\
&=\partial_\beta b_{i-1}(\hat\beta_k)^\TT 
A_{i-1}^{-1}(\alpha_k^*)(\DeX-hb_{i-1}(\beta_k))
-h \partial_\beta b_{i-1}(\hat\beta_k)^\TT 
A_{i-1}^{-1}(\alpha_k^*)(b_{i-1}(\hat\beta_k)-b_{i-1}(\beta_k))\\
&\qquad+
\frac{1}{\sqrt n}\int_0^1\left.
\partial_{\alpha}\left(
\partial_\beta b_{i-1}(\hat\beta_k)^\TT 
A_{i-1}^{-1}(\alpha)(\DeX-hb_{i-1}(\hat\beta_k))\right)
\right|_{\alpha=\alpha_k^*+u(\hat\alpha_k-\alpha_k^*)}
\dd u
\sqrt n(\hat\alpha_k-\alpha_k^*)\\
&=
\sum_{j=0}^{\mma-1}
\frac{1}{T^{j/2}}
\mathcal Z_{k,i}^{[j]}
\otimes\bigl(\sqrt T(\hat\beta_k-\beta_k)\bigr)^{\otimes j}
\\
&\qquad+
\frac{1}{T^{\mma/2}}
\int_0^1\frac{(1-u)^{\mma-1}}{(\mma-1)!}\left.\partial_\beta^{\mma}
\Bigl(\partial_\beta b_{i-1}(\beta)^\TT 
A_{i-1}^{-1}(\alpha_k^*)(\DeX-hb_{i-1}(\beta_k))
\Bigr)
\right|_{\beta=\beta_k+u(\hat\beta_k-\beta_k)}
\\
&\qquad\qquad
\otimes\bigl(\sqrt T(\hat\beta_k-\beta_k)\bigr)^{\otimes \mma}
\\
&\qquad
-\frac{h}{\sqrt T} 
\int_0^1 \partial_{\beta}
\left.
\Bigl(
\partial_\beta b_{i-1}(\hat\beta_k)^\TT 
A_{i-1}^{-1}(\alpha_k^*) (b_{i-1}(\beta)-b_{i-1}(\beta_k))
\Bigr)
\right|_{\beta=\beta_k+u(\hat\beta_k-\beta_k)}
\dd u
\sqrt T(\hat\beta_k-\beta_k)\\
&\qquad+
\frac{1}{\sqrt n}\int_0^1\left.
\partial_{\alpha}
\left(
\partial_\beta b_{i-1}(\hat\beta_k)^\TT 
A_{i-1}^{-1}(\alpha)(\DeX-hb_{i-1}(\hat\beta_k))\right)
\right|_{\alpha=\alpha_k^*+u(\hat\alpha_k-\alpha_k^*)}
\dd u
\sqrt n(\hat\alpha_k-\alpha_k^*)\\
&=
\sum_{j=0}^{\mma-1}
\frac{1}{T^{j/2}}\mathcal Z_{k,i}^{[j]}\otimes
\bigl(\sqrt T(\hat\beta_k-\beta_k)\bigr)^{\otimes j}
+
O_p\Biggl(\frac{\sqrt h}{T^{\mma/2}}
\lor\sqrt{\frac{h}{n}}\Biggr)
\\
&=
\sum_{j=0}^{\mma-1}
\frac{1}{T^{j/2}}\mathcal N_{k,i}^{[j]}\otimes
\bigl(\sqrt T(\hat\beta_k-\beta_k)\bigr)^{\otimes j}
+
\sum_{j=0}^{\mma-1}
\frac{1}{T^{j/2}}\EE_{\theta_k}[\mathcal Z_{k,i}^{[j]}|\GG]\otimes
\bigl(\sqrt T(\hat\beta_k-\beta_k)\bigr)^{\otimes j}
\\
&\qquad+
O_p\Biggl(\frac{\sqrt h}{T^{\mma/2}}
\lor\sqrt{\frac{h}{n}}\Biggr)
\\
&=
\sum_{j=0}^{\mma-1}
\frac{1}{T^{j/2}}\mathcal N_{k,i}^{[j]}\otimes
\bigl(\sqrt T(\hat\beta_k-\beta_k)\bigr)^{\otimes j}
+
O_p\Biggl(\frac{\sqrt h}{T^{\mma/2}}
\lor\sqrt{\frac{h}{n}}\lor h^2\Biggr)
\\
&=
\sum_{j=0}^{\mma-1}
\frac{1}{T^{j/2}}\mathcal N_{k,i}^{[j]}\otimes
\bigl(\sqrt T(\hat\beta_k-\beta_k)\bigr)^{\otimes j}
+
O_p\Biggl(\sqrt{\frac{h}{n}}\Biggr),
\end{align*}
we have
\begin{align}
&\frac{1}{\sqrt{T}}
\max_{1\le k \le m_n}
\left\|
\frac{k}{m_n}\left(1-\frac{m_n}{[n\underline{\tau}_n]}\right)
\sum_{i=1}^{m_n}\check\zeta_{1,i}
\right\|
\nonumber
\\
&\le
\frac{1}{\sqrt{T}}
\frac{|[n\underline{\tau}_n]-m_n|}{[n\underline{\tau}_n]}
\left\|
\sum_{i=1}^{m_n}\check\zeta_{1,i}
\right\|
\nonumber
\\
&\le
\frac{1}{\sqrt{T}}
\frac{|[n\underline{\tau}_n]-m_n|}{[n\underline{\tau}_n]}
\Biggl(
\sum_{j=0}^{m-1}
\frac{1}{T^{j/2}}
\left\|
\sum_{i=1}^{m_n}\mathcal N_{1,i}^{[j]}
\right\|
\bigl(\sqrt T|\hat\beta_1-\beta_1|\bigr)^j
+m_n O_p\biggl(\sqrt{\frac{h}{n}}\biggr)
\Biggr)
\nonumber
\\
&\le
\frac{m_n}{[n\underline{\tau}_n]}
n^{\epsilon_1-1}|[n\underline{\tau}_n]-m_n|
\Biggl(
\sum_{j=0}^{m-1}
\frac{n^{1-\epsilon_1}}{T^{(j+1)/2}m_n}
\left\|
\sum_{i=1}^{m_n}\mathcal N_{1,i}^{[j]}
\right\|
\bigl(\sqrt T|\hat\beta_1-\beta_1|\bigr)^j
+O_p(n^{-\epsilon_1})
\Biggr).
\label{eq-001-4}
\end{align}
Since
$\dfrac{m_n}{[n\underline{\tau}_n]}n^{\epsilon_1-1}|[n\underline{\tau}_n]-m_n|=O_p(1)$,
$\EE_{\theta_1}[\|\mathcal N_{1,i}^{[j]}\|^2]\le Ch$
and 
\begin{align*}
\frac{n^{2-2\epsilon_1}m_n h}{T^{(j+1)}m_n^2}
=O(n^{-2\epsilon_1})
=o(1)
\end{align*}
for $0\le j\le \mma-1$, we have
\begin{align}
\frac{1}{\sqrt{T}}
\max_{1\le k \le m_n}
\left\|
\frac{k}{m_n}\left(1-\frac{m_n}{[n\underline{\tau}_n]}\right)
\sum_{i=1}^{m_n}\check\zeta_{1,i}
\right\|
=o_p(1)
\label{eq-001-5}
\end{align}
from Lemma \ref{lem1} and \eqref{eq-001-4}.
By the same argument, we also have
\begin{align}
\frac{1}{\sqrt{T}}
\max_{1\le k \le m_n}
\left\|
\frac{k}{[n\underline{\tau}_n]}
\sum_{i=1}^{m_n}\check\zeta_{1,i}
\right\|
=o_p(1).
\label{eq-001-6}
\end{align}
Therefore, according to 
\eqref{eq-001-3}, \eqref{eq-001-5}, \eqref{eq-001-6} and
Theorem 3 of Tonaki et al. (2020),
we obtain
\begin{align*}
\TLC
&=
\frac{1}{\sqrt{\underline{\tau}_n T}}
\max_{1\le k\le m_n}
\left\|
{\mathcal I}_{1,n}^{-1/2}
\left(
\sum_{i=1}^k\check\zeta_{1,i}
-\frac{k}{[n\underline{\tau}_n]}
\sum_{i=1}^{[n\underline{\tau}_n]}\check\zeta_{1,i}
\right)\right\|
\\
&=
\frac{1}{\sqrt{\underline{\tau}_n T}}
\max_{1\le k\le m_n}
\left\|
{\mathcal I}_{1,n}^{-1/2}
\left(
\sum_{i=1}^k\check\zeta_{1,i}
-\frac{k}{m_n}\sum_{i=1}^{m_n}\check\zeta_{1,i}
\right)\right\|
+o_p(1)
\\
&=
\sqrt{\frac{\tau_n^L}{\underline{\tau}_n}}
\frac{1}{\sqrt{\tau_n^L T}}
\max_{1\le k\le m_n}
\left\|
{\mathcal I}_{1,n}^{-1/2}
\left(
\sum_{i=1}^k\check\zeta_{1,i}
-\frac{k}{m_n}\sum_{i=1}^{m_n}\check\zeta_{1,i}
\right)\right\|
+o_p(1)
\\
&
\dto
\sup_{0\le s\le 1}\|\boldsymbol{B}_q^0(s)\|.
\end{align*}

\textit{Proof of \eqref{eq-001-2}. }
We have
\begin{align*}
&\sum_{i=1}^k\check\zeta_{1,i}
-\frac{k}{[n\underline{\tau}_n]}
\sum_{i=1}^{[n\underline{\tau}_n]}\check\zeta_{1,i}
\nonumber
\\
&=
\sum_{i=1}^k\check\zeta_{1,i}
-\frac{k}{[n\tau_*^\alpha]}\sum_{i=1}^{[n\tau_*^\alpha]}\check\zeta_{1,i}
+\frac{k}{[n\tau_*^\alpha]}
\left(1-\frac{[n\tau_*^\alpha]}{[n\underline{\tau}_n]}\right)
\sum_{i=1}^{[n\tau_*^\alpha]}\check\zeta_{1,i}
+\frac{k}{[n\underline{\tau}_n]}
\sum_{i=[n\underline{\tau}_n]+1}^{[n\tau_*^\alpha]}\check\zeta_{1,i}.
\end{align*}
By the same argument, we obtain
\begin{align*}
\TUC
&=
\frac{1}{\sqrt{\underline{\tau}_n T}}
\max_{1\le k\le [n\tau_*^\alpha]}
\left\|
{\mathcal I}_{1,n}^{-1/2}
\left(
\sum_{i=1}^k\check\zeta_{1,i}
-\frac{k}{[n\underline{\tau}_n]}
\sum_{i=1}^{[n\underline{\tau}_n]}\check\zeta_{1,i}
\right)\right\|
\\
&=
\frac{1}{\sqrt{\underline{\tau}_n T}}
\max_{1\le k\le [n\tau_*^\alpha]}
\left\|
{\mathcal I}_{1,n}^{-1/2}
\left(
\sum_{i=1}^k\check\zeta_{1,i}
-\frac{k}{[n\tau_*^\alpha]}\sum_{i=1}^{[n\tau_*^\alpha]}\check\zeta_{1,i}
\right)\right\|
+o_p(1)
\\
&=
\sqrt{\frac{\tau_*^\alpha}{\underline{\tau}_n}}
\frac{1}{\sqrt{\tau_*^\alpha T}}
\max_{1\le k\le [n\tau_*^\alpha]}
\left\|
{\mathcal I}_{1,n}^{-1/2}
\left(
\sum_{i=1}^k\check\zeta_{1,i}
-\frac{k}{[n\tau_*^\alpha]}\sum_{i=1}^{[n\tau_*^\alpha]}\check\zeta_{1,i}
\right)\right\|
+o_p(1)
\\
&
\dto
\sup_{0\le s\le 1}\|\boldsymbol{B}_q^0(s)\|.
\end{align*}

Next, we show \eqref{eq-th1-4}.
Let
\begin{align*}
\TLD
&=
\frac{1}{\sqrt{(1-\overline{\tau}_n) T}}
\max_{1\le k\le n-M_n}
\left\|
{\mathcal I}_{2,n}^{-1/2}
\left(
\sum_{i=[n\overline{\tau}_n]+1}^{[n\overline{\tau}_n]+k}
\check\zeta_{2,i}
-\frac{k}{n-[n\overline{\tau}_n]}
\sum_{i=[n\overline{\tau}_n]+1}^{n}
\check\zeta_{2,i}
\right)\right\|,
\\
\TUD
&=
\frac{1}{\sqrt{(1-\overline{\tau}_n) T}}
\max_{1\le k\le n-[n\tau_*^\alpha]}
\left\|
{\mathcal I}_{2,n}^{-1/2}
\left(
\sum_{i=[n\overline{\tau}_n]+1}^{[n\overline{\tau}_n]+k}
\check\zeta_{2,i}
-\frac{k}{n-[n\overline{\tau}_n]}
\sum_{i=[n\overline{\tau}_n]+1}^{n}
\check\zeta_{2,i}
\right)\right\|.
\end{align*}
Since $\TLD\le\TOT\le\TUD$ on $D_n$,
it is sufficient to prove
\begin{align*}
\TLD\dto\sup_{0\le s\le 1}\|\boldsymbol B_q^0(s)\|,
\quad
\TUD\dto\sup_{0\le s\le 1}\|\boldsymbol B_q^0(s)\|.
\end{align*}
These can be proved in the same way as \eqref{eq-000-8} and \eqref{eq-000-9}.
\end{proof}

\begin{proof}[\bf{Proof of Theorem \ref{th3}}]
\textit{Step 1.}
We first prove that
$P(\TUO>w_1(\epsilon))$ converges to one as $n\to\infty$ under $H_1^{(1)}$.

(a) If we prove
\begin{align}
&\frac{1}{\tau_*^\beta T}\sum_{i=1}^{[n\tau_*^\beta]}\check\xi_{1,i}
\pto
\mathcal G_{1,1},
\label{eq-A-1}
\\
&\frac{1}{(\underline\tau_n-\tau_*^\beta)T}
\sum_{i=[n\tau_*^\beta]+1}^{[n\underline\tau_n]}\check\xi_{1,i}
\pto
\mathcal G_{1,2},
\label{eq-A-2}
\end{align}
then
\begin{align*}
\frac{1}{\underline\tau_n T}
\sum_{i=1}^{[n\underline\tau_n]}\check\xi_{1,i}
&=
\frac{\tau_*^\beta}{\underline\tau_n}
\frac{1}{\tau_*^\beta T}\sum_{i=1}^{[n\tau_*^\beta]}\check\xi_{1,i}
+
\frac{\underline\tau_n-\tau_*^\beta}{\underline\tau_n}
\frac{1}{(\underline\tau_n-\tau_*^\beta)T}
\sum_{i=[n\tau_*^\beta]+1}^{[n\underline\tau_n]}\check\xi_{1,i}
\\
&\pto
\frac{\tau_*^\beta}{\tau_*^\alpha}\mathcal G_{1,1}
+
\biggl(1-\frac{\tau_*^\beta}{\tau_*^\alpha}\biggr)
\mathcal G_{1,2},
\end{align*}
\begin{align*}
\frac{1}{\underline\tau_n T}
\sum_{i=1}^{[n\tau_*^\beta]}\check\xi_{1,i}
-
\frac{[n\tau_*^\beta]}{[n\underline\tau_n]}
\frac{1}{\underline\tau_n T}
\sum_{i=1}^{[n\underline\tau_n]}\check\xi_{1,i}
&=
\frac{\tau_*^\beta}{\underline\tau_n}
\frac{1}{\tau_*^\beta T}
\sum_{i=1}^{[n\tau_*^\beta]}\check\xi_{1,i}
-
\frac{[n\tau_*^\beta]}{[n\underline\tau_n]}
\frac{1}{\underline\tau_n T}
\sum_{i=1}^{[n\underline\tau_n]}\check\xi_{1,i}
\\
&\pto
\frac{\tau_*^\beta}{\tau_*^\alpha}
\Biggl(
\mathcal G_{1,1}
-\biggl(
\frac{\tau_*^\beta}{\tau_*^\alpha}\mathcal G_{1,1}
+\biggl(1-\frac{\tau_*^\beta}{\tau_*^\alpha}\biggr)
\mathcal G_{1,2}
\biggr)
\Biggr)
\\
&=
\frac{\tau_*^\beta}{\tau_*^\alpha}
\biggl(1-\frac{\tau_*^\beta}{\tau_*^\alpha}\biggr)
\bigl(\mathcal G_{1,1}-\mathcal G_{1,2}\bigr)
\neq0.
\end{align*}
Therefore, we have
\begin{align*}
\TUO
&=
\frac{1}{\sqrt{d\underline{\tau}_n T}}
\max_{1\le k\le [n\underline{\tau}_n]}
\left|
\sum_{i=1}^k\check\xi_{1,i}
-\frac{k}{[n\underline{\tau}_n]}\sum_{i=1}^{[n\underline{\tau}_n]}\check\xi_{1,i}
\right|
\\
&\ge
\frac{1}{\sqrt{d\underline{\tau}_n T}}
\left|
\sum_{i=1}^{[n\tau_*^\beta]}\check\xi_{1,i}
-\frac{[n\tau_*^\beta]}{[n\underline{\tau}_n]}
\sum_{i=1}^{[n\underline{\tau}_n]}\check\xi_{1,i}
\right|
\\
&=
\sqrt{\frac{\underline{\tau}_n T}{d}}
\left|
\frac{1}{\underline\tau_n T}
\sum_{i=1}^{[n\tau_*^\beta]}\check\xi_{1,i}
-
\frac{[n\tau_*^\beta]}{[n\underline\tau_n]}
\frac{1}{\underline\tau_n T}
\sum_{i=1}^{[n\underline\tau_n]}\check\xi_{1,i}
\right|
\\
&\pto\infty,
\end{align*}
which implies $P(\TUO>w_1(\epsilon))\lto1$.

\eqref{eq-A-1} can be shown similarly to (4.54) of Tonaki et al. (2020).

\textit{Proof of \eqref{eq-A-2}. }
We can prove 
\begin{align}
\frac{1}{(\tau_*^\alpha-\tau_*^\beta)T}
\sum_{i=[n\tau_*^\beta]+1}^{[n\tau_*^\alpha]}\check\xi_{1,i}
\pto
\mathcal G_{1,2}
\label{eq-A-3}
\end{align}
with the same argument as (4.55) of Tonaki et al. (2020). 
We have
\begin{align*}
\Delta_n
&=
\left|
\frac{1}{(\underline\tau_n-\tau_*^\beta) T}
\sum_{i=[n\tau_*^\beta]+1}^{[n\underline\tau_n]}\check\xi_{1,i}
-\frac{1}{(\tau_*^\alpha-\tau_*^\beta) T}
\sum_{i=[n\tau_*^\beta]+1}^{[n\tau_*^\alpha]}\check\xi_{1,i}
\right|
\\
&=
\frac{1}{T}
\left|
\left(
\frac{1}{\underline\tau_n-\tau_*^\beta}
-\frac{1}{\tau_*^\alpha-\tau_*^\beta}
\right)
\sum_{i=[n\tau_*^\beta]+1}^{[n\underline\tau_n]}\check\xi_{1,i}
+
\frac{1}{\tau_*^\alpha-\tau_*^\beta}
\left(
\sum_{i=[n\tau_*^\beta]+1}^{[n\underline\tau_n]}\check\xi_{1,i}
-\sum_{i=[n\tau_*^\beta]+1}^{[n\tau_*^\alpha]}\check\xi_{1,i}
\right)
\right|
\\
&\le
\left|
\frac{\tau_*^\alpha-\underline\tau_n}
{(\underline\tau_n-\tau_*^\beta)(\tau_*^\alpha-\tau_*^\beta)}
\right|
\frac{1}{T}
\left|
\sum_{i=[n\tau_*^\beta]+1}^{[n\underline\tau_n]}\check\xi_{1,i}
\right|
+
\frac{1}{(\tau_*^\alpha-\tau_*^\beta) T}
\left|
\sum_{i=[n\tau_*^\beta]+1}^{[n\underline\tau_n]}\check\xi_{1,i}
-\sum_{i=[n\tau_*^\beta]+1}^{[n\tau_*^\alpha]}\check\xi_{1,i}
\right|
\\
&=
\left|
\frac{n^{\epsilon_1}(\tau_*^\alpha-\underline\tau_n)}
{(\underline\tau_n-\tau_*^\beta)(\tau_*^\alpha-\tau_*^\beta)}
\right|
\frac{n^{-\epsilon_1}}{T}
\left|
\sum_{i=[n\tau_*^\beta]+1}^{[n\underline\tau_n]}\check\xi_{1,i}
\right|
+
\frac{1}{(\tau_*^\alpha-\tau_*^\beta) T}
\left|
\sum_{i=[n\tau_*^\beta]+1}^{[n\underline\tau_n]}\check\xi_{1,i}
-\sum_{i=[n\tau_*^\beta]+1}^{[n\tau_*^\alpha]}\check\xi_{1,i}
\right|
\\
&=:
\left|
\frac{n^{\epsilon_1}(\tau_*^\alpha-\underline\tau_n)}
{(\underline\tau_n-\tau_*^\beta)(\tau_*^\alpha-\tau_*^\beta)}
\right|
\check{\mathcal S}_n
+
\frac{1}{\tau_*^\alpha-\tau_*^\beta}
\check{\mathcal Q}_n.
\end{align*}
If we prove $\check{\mathcal S}_n\pto0$ and $\check{\mathcal Q}_n\pto0$,
then, from $n^{\epsilon_1}(\tau_*^\alpha-\underline\tau_n)=O_p(1)$
and \eqref{eq-A-3}, 
we have $\Delta_n\pto0$ and \eqref{eq-A-2}.  
In the following, we prove them.

Set 
$\mathcal Y_{k,i}=1_d^\TT a_{i-1}^{-1}(\alpha_k^*)\DeX$, 
$\mathcal M_{k,i}=\mathcal Y_{k,i}-\EE_{\theta_k}[\mathcal Y_{k,i}|\GG]$.
Notice that
\begin{align*}
\check\xi_{k,i}
&=1_d^\TT a_{i-1}^{-1}(\hat\alpha_k)(\DeX-hb_{i-1}(\hat\beta_k))\\
&=1_d^\TT a_{i-1}^{-1}(\alpha_k^*)(\DeX-hb_{i-1}(\hat\beta_k))
\\
&\qquad+
\frac{1}{\sqrt n}\int_0^1\left.
\partial_{\alpha}\left(1_d^\TT a_{i-1}^{-1}(\alpha)(\DeX-hb_{i-1}(\hat\beta_k))\right)
\right|_{\alpha=\alpha_k^*+u(\hat\alpha_k-\alpha_k^*)}
\dd u
\sqrt n(\hat\alpha_k-\alpha_k^*)\\
&=1_d^\TT a_{i-1}^{-1}(\alpha_k^*)\DeX
-h 1_d^\TT a_{i-1}^{-1}(\alpha_k^*)b_{i-1}(\hat\beta_k)\\
&\qquad+
\frac{1}{\sqrt n}\int_0^1\left.
\partial_{\alpha}\left(1_d^\TT a_{i-1}^{-1}(\alpha)(\DeX-hb_{i-1}(\hat\beta_k))\right)
\right|_{\alpha=\alpha_k^*+u(\hat\alpha_k-\alpha_k^*)}
\dd u
\sqrt n(\hat\alpha_k-\alpha_k^*)\\
&=1_d^\TT a_{i-1}^{-1}(\alpha_k^*)\DeX
+O_p\Biggl(h\lor\sqrt{\frac{h}{n}}\Biggr)
\\
&=\mathcal M_{k,i}+\EE_{\theta_k}[\mathcal Y_{k,i}|\GG]+O_p(h)
\\
&=\mathcal M_{k,i}+O_p(h).
\end{align*}
We see
\begin{align*}
\check{\mathcal S}_n
&=
\frac{n^{-\epsilon_1}}{T}
\left|
\sum_{i=[n\tau_*^\beta]+1}^{[n\underline\tau_n]}
\mathcal M_{1,i}
\right|
+O_p\left(
\frac{n^{1-\epsilon_1}h}{T}
\right)
=:\mathcal S_n+o_p(1),
\end{align*}
and 
\begin{align*}
\check{\mathcal Q}_n
=
\frac{1}{T}
\left|
\sum_{i=[n\underline\tau_n]+1}^{[n\tau_*^\alpha]}\check\xi_{1,i}
\right|
=
\frac{1}{T}
\left|
\sum_{i=[n\underline\tau_n]+1}^{[n\tau_*^\alpha]}\mathcal M_{1,i}
\right|
+O_p(n^{-\epsilon_1})
=:
\Qn+o_p(1)
\end{align*}
on $D_n$. 
Since $[n\underline{\tau}_n]\le [n\tau_*^\alpha]$, 
$m_n\le[n\underline{\tau}_n]$
on $D_n$, 
\begin{align*}
\frac{n^{-2\epsilon_1}([n\tau_*^\alpha]-[n\tau_*^\beta])h}{T^2}
=O\left(\frac{n^{-2\epsilon_1}}{T}\right)=o(1),
\quad
\frac{([n\tau_*^\alpha]-m_n)h}{T^2}
=O\left(\frac{n^{-\epsilon_1}}{T}\right)=o(1),
\end{align*}
we have
$\mathcal S_n\pto0$ and $\mathcal Q_n\pto0$
from Lemma \ref{lem1}.
Hence, we obtain the desired results. 

(b) According to
\begin{align*}
\TUO
&=
\frac{1}{\sqrt{d\underline{\tau}_n T}}
\max_{1\le k\le [n\underline{\tau}_n]}
\left|
\sum_{i=1}^k\check\xi_{1,i}
-\frac{k}{[n\underline{\tau}_n]}\sum_{i=1}^{[n\underline{\tau}_n]}\check\xi_{1,i}
\right|
\\
&\ge
\frac{1}{\sqrt{d\underline{\tau}_n T}}
\left|
\sum_{i=1}^{[n\tau_*^\beta]}\check\xi_{1,i}
-\frac{[n\tau_*^\beta]}{[n\underline{\tau}_n]}
\sum_{i=1}^{[n\underline{\tau}_n]}\check\xi_{1,i}
\right|
\\
&=
\sqrt{\frac{ T\Deko^2}{d\underline{\tau}_n }}
\left|
\frac{1}{T\Deko}
\sum_{i=1}^{[n\tau_*^\beta]}\check\xi_{1,i}
-
\frac{[n\tau_*^\beta]}{[n\underline\tau_n]}
\frac{1}{T\Deko}
\sum_{i=1}^{[n\underline\tau_n]}\check\xi_{1,i}
\right|,
\end{align*}
it is enough to prove that there exists $c\neq0$ such that
\begin{align}
\mathcal K_n^{(1)}
=
\frac{1}{T\Deko}
\sum_{i=1}^{[n\tau_*^\beta]}\check\xi_{1,i}
-
\frac{[n\tau_*^\beta]}{[n\underline\tau_n]}
\frac{1}{T\Deko}
\sum_{i=1}^{[n\underline\tau_n]}\check\xi_{1,i}
\pto c.
\label{eq-A-4}
\end{align}
Note that there exists $c'\neq0$ such that
\begin{align}
\mathcal K_n^{(2)}
=
\frac{1}{T\Deko}
\sum_{i=1}^{[n\tau_*^\beta]}\check\xi_{1,i}
-
\frac{[n\tau_*^\beta]}{[n\tau_*^\alpha]}
\frac{1}{T\Deko}
\sum_{i=1}^{[n\tau_*^\alpha]}\check\xi_{1,i}
\pto c'.
\label{eq-A-5}
\end{align}
as in the proof of Proposition 2 of Tonaki et al. (2021).
Meanwhile, we see
\begin{align*}
\Delta_n
&=
|\mathcal K_n^{(1)}-\mathcal K_n^{(2)}|
\\
&=
\frac{1}{T\Deko}
\left|
\frac{[n\tau_*^\beta]}{[n\underline\tau_n]}
\sum_{i=1}^{[n\underline\tau_n]}\check\xi_{1,i}
-
\frac{[n\tau_*^\beta]}{[n\tau_*^\alpha]}
\sum_{i=1}^{[n\tau_*^\alpha]}\check\xi_{1,i}
\right|
\\
&\le
\frac{[n\tau_*^\beta]}{[n\underline\tau_n]}
\frac{1}{T\Deko}
\left|
\sum_{i=1}^{[n\underline\tau_n]}\check\xi_{1,i}
-
\sum_{i=1}^{[n\tau_*^\alpha]}\check\xi_{1,i}
\right|
+
n^{\epsilon_1}
\left|
\frac{[n\tau_*^\beta]}{[n\underline\tau_n]}
-\frac{[n\tau_*^\beta]}{[n\tau_*^\alpha]}
\right|
\frac{n^{-\epsilon_1}}{T\Deko}
\left|
\sum_{i=1}^{[n\tau_*^\alpha]}\check\xi_{1,i}
\right|
\\
&=:
\frac{[n\tau_*^\beta]}{[n\underline\tau_n]}
\check{\mathcal Q}_n
+
n^{\epsilon_1}
\left|
\frac{[n\tau_*^\beta]}{[n\underline\tau_n]}
-\frac{[n\tau_*^\beta]}{[n\tau_*^\alpha]}
\right|
\check{\mathcal S}_n.
\end{align*}
Here, we have, from \textbf{[G5]},
\begin{align*}
\check{\mathcal Q}_n
=
\frac{1}{T\Deko}
\left|
\sum_{i=[n\underline\tau_n]+1}^{[n\tau_*^\alpha]}\check\xi_{1,i}
\right|
=
\frac{1}{T\Deko}
\left|
\sum_{i=[n\underline\tau_n]+1}^{[n\tau_*^\alpha]}\mathcal M_{1,i}
\right|
+O_p\left(
\frac{1}{n^{\epsilon_1}\Deko}
\right)
=:
\Qn+o_p(1)
\end{align*}
and
\begin{align*}
\check{\mathcal S}_n
&=
\frac{n^{-\epsilon_1}}{T\Deko}
\left|
\sum_{i=1}^{[n\tau_*^\alpha]}\mathcal M_{1,i}
\right|
+O_p\left(
\frac{1}{n^{\epsilon_1}\Deko}
\right)
=:\mathcal S_n+o_p(1)
\end{align*}
on $D_n$.
Since $m_n\le[n\underline{\tau}_n]$ on $D_n$,
\begin{align*}
\frac{n^{-2\epsilon_1}[n\tau_*^\alpha]h}{T^2\Deko^2}
=O\left(\frac{n^{-2\epsilon_1}}{T\Deko^2}\right)=o(1),
\quad
\frac{([n\tau_*^\alpha]-m_n)h}{T^2\Deko^2}
=O\left(\frac{n^{-\epsilon_1}}{T\Deko^2}\right)=o(1),
\end{align*}
we have
$\mathcal S_n\pto0$ and $\mathcal Q_n\pto0$
from Lemma \ref{lem1},
that is, we obtain $\check{\mathcal S}_n\pto0$ and $\check{\mathcal Q}_n\pto0$.
Therefore, from
\begin{align*}
n^{\epsilon_1}
\left|
\frac{[n\tau_*^\beta]}{[n\underline\tau_n]}
-\frac{[n\tau_*^\beta]}{[n\tau_*^\alpha]}
\right|
=O_p(1),
\end{align*}
we have $\Delta_n\pto0$.
This and \eqref{eq-A-5} lead to \eqref{eq-A-4}.


\textit{Step 2.}
Next, we prove $P(\TOO>w_1(\epsilon))$ converges to one as $n\to\infty$
under $H_1^{(2)}$.

(a) If we prove
\begin{align}
\frac{1}{(\tau_*^\beta-{\overline\tau}_n) T}
\sum_{i=[n\overline\tau_n]+1}^{[n\tau_*^\beta]}\check\xi_{2,i}
\pto
\mathcal G_{2,1},
\quad
\frac{1}{(1-\tau_*^\beta)T}
\sum_{i=[n\tau_*^\beta]+1}^n\check\xi_{2,i}
\pto
\mathcal G_{2,2},
\label{eq-B-2}
\end{align}
then
\begin{align*}
\frac{1}{(1-\overline{\tau}_n) T}
\sum_{i=[n\overline{\tau}_n]+1}^{n}\check\xi_{2,i}
&=
\frac{\tau_*^\beta-\overline{\tau}_n}{1-\overline{\tau}_n}
\frac{1}{(\tau_*^\beta-\overline{\tau}_n) T}
\sum_{i=[n\overline{\tau}_n]+1}^{[n\tau_*^\beta]}\check\xi_{2,i}
+
\frac{1-\tau_*^\beta}{1-\overline{\tau}_n}
\frac{1}{(1-\tau_*^\beta) T}
\sum_{i=[n\tau_*^\beta]+1}^{n}\check\xi_{2,i}
\\
&\pto
\frac{\tau_*^\beta-\tau_*^\alpha}{1-\tau_*^\alpha}
\mathcal G_{2,1}
+
\biggl(1-\frac{\tau_*^\beta-\tau_*^\alpha}{1-\tau_*^\alpha}\biggr)
\mathcal G_{2,2},
\end{align*}
\begin{align*}
&\frac{1}{(1-\overline{\tau}_n) T}
\sum_{i=[n\overline{\tau}_n]+1}^{[n\tau_*^\beta]}\check\xi_{2,i}
-\frac{[n\tau_*^\beta]-[n\overline{\tau}_n]}{n-[n\overline{\tau}_n]}
\frac{1}{(1-\overline{\tau}_n) T}
\sum_{i=[n\overline{\tau}_n]+1}^{n}\check\xi_{2,i}
\\
&=
\frac{\tau_*^\beta-\overline{\tau}_n}{1-\overline{\tau}_n}
\frac{1}{(\tau_*^\beta-\overline{\tau}_n) T}
\sum_{i=[n\overline{\tau}_n]+1}^{[n\tau_*^\beta]}\check\xi_{2,i}
-\frac{[n\tau_*^\beta]-[n\overline{\tau}_n]}{n-[n\overline{\tau}_n]}
\frac{1}{(1-\overline{\tau}_n) T}
\sum_{i=[n\overline{\tau}_n]+1}^{n}\check\xi_{2,i}
\\
&\pto
\frac{\tau_*^\beta-\tau_*^\alpha}{1-\tau_*^\alpha}
\Biggl(
\mathcal G_{2,1}
-\biggl(
\frac{\tau_*^\beta-\tau_*^\alpha}{1-\tau_*^\alpha}
\mathcal G_{2,1}
+\biggl(1-\frac{\tau_*^\beta-\tau_*^\alpha}{1-\tau_*^\alpha}\biggr)
\mathcal G_{2,2}
\biggr)
\Biggr)
\\
&=
\frac{\tau_*^\beta-\tau_*^\alpha}{1-\tau_*^\alpha}
\biggl(1-\frac{\tau_*^\beta-\tau_*^\alpha}{1-\tau_*^\alpha}\biggr)
\bigl(\mathcal G_{2,1}-\mathcal G_{2,2}\bigr)
\neq0,
\end{align*}
and
\begin{align*}
\TOO
&=
\frac{1}{\sqrt{d(1-\overline{\tau}_n) T}}
\max_{1\le k\le n-[n\overline{\tau}_n]}
\left|
\sum_{i=[n\overline{\tau}_n]+1}^{[n\overline{\tau}_n]+k}\check\xi_{2,i}
-\frac{k}{n-[n\overline{\tau}_n]}
\sum_{i=[n\overline{\tau}_n]+1}^{n}\check\xi_{2,i}
\right|
\\
&\ge
\frac{1}{\sqrt{d(1-\overline{\tau}_n) T}}
\left|
\sum_{i=[n\overline{\tau}_n]+1}^{[n\tau_*^\beta]}\check\xi_{2,i}
-\frac{[n\tau_*^\beta]-[n\overline{\tau}_n]}{n-[n\overline{\tau}_n]}
\sum_{i=[n\overline{\tau}_n]+1}^{n}\check\xi_{2,i}
\right|
\\
&=
\sqrt{\frac{(1-\overline{\tau}_n) T}{d}}
\left|
\frac{1}{(1-\overline{\tau}_n) T}
\sum_{i=[n\overline{\tau}_n]+1}^{[n\tau_*^\beta]}\check\xi_{2,i}
-\frac{[n\tau_*^\beta]-[n\overline{\tau}_n]}{n-[n\overline{\tau}_n]}
\frac{1}{(1-\overline{\tau}_n) T}
\sum_{i=[n\overline{\tau}_n]+1}^{n}\check\xi_{2,i}
\right|
\\
&\pto\infty.
\end{align*}
\eqref{eq-B-2} can be proved in the same way as
\eqref{eq-A-2} and (4.55) of Tonaki et al. (2020).

(b) According to 
\begin{align*}
\TOO
&=
\frac{1}{\sqrt{d(1-\overline{\tau}_n) T}}
\max_{1\le k\le n-[n\overline{\tau}_n]}
\left|
\sum_{i=[n\overline{\tau}_n]+1}^{[n\overline{\tau}_n]+k}\check\xi_{2,i}
-\frac{k}{n-[n\overline{\tau}_n]}
\sum_{i=[n\overline{\tau}_n]+1}^{n}\check\xi_{2,i}
\right|
\\
&\ge
\frac{1}{\sqrt{d(1-\overline{\tau}_n) T}}
\left|
\sum_{i=[n\overline{\tau}_n]+1}^{[n\tau_*^\beta]}\check\xi_{2,i}
-\frac{[n\tau_*^\beta]-[n\overline{\tau}_n]}{n-[n\overline{\tau}_n]}
\sum_{i=[n\overline{\tau}_n]+1}^{n}\check\xi_{2,i}
\right|
\\
&=
\sqrt{\frac{T\Dekt^2}{d(1-\overline{\tau}_n)}}
\left|
\frac{1}{T\Dekt}
\sum_{i=[n\overline{\tau}_n]+1}^{[n\tau_*^\beta]}\check\xi_{2,i}
-\frac{[n\tau_*^\beta]-[n\overline{\tau}_n]}{n-[n\overline{\tau}_n]}
\frac{1}{T\Dekt}
\sum_{i=[n\overline{\tau}_n]+1}^{n}\check\xi_{2,i}
\right|,
\end{align*}
it suffices to prove that 
there exists $c\neq0$ such that
\begin{align*}
\frac{1}{T\Dekt}
\sum_{i=[n\overline{\tau}_n]+1}^{[n\tau_*^\beta]}\check\xi_{2,i}
-\frac{[n\tau_*^\beta]-[n\overline{\tau}_n]}{n-[n\overline{\tau}_n]}
\frac{1}{T\Dekt}
\sum_{i=[n\overline{\tau}_n]+1}^{n}\check\xi_{2,i}
\pto c.
\end{align*}
This can be derived in the same way as \eqref{eq-A-4}.
\end{proof}

\begin{proof}[\bf{Proof of Theorem \ref{th4}}]
\textit{Step 1.}
First, we prove $P(\TUT>w_1(\epsilon))\lto1$ under $H_1^{(1)}$.

(a) If we prove
\begin{align}
&\frac{1}{\tau_*^\beta T}\sum_{i=1}^{[n\tau_*^\beta]}\check\zeta_{1,i}
\pto
\mathcal H_{1,1},
\label{eq-C-1}
\\
&\frac{1}{(\underline\tau_n-\tau_*^\beta)T}
\sum_{i=[n\tau_*^\beta]+1}^{[n\underline\tau_n]}\check\zeta_{1,i}
\pto
\mathcal H_{1,2},
\label{eq-C-2}
\end{align}
then
\begin{align*}
\frac{1}{\underline\tau_n T}
\sum_{i=1}^{[n\underline\tau_n]}\check\zeta_{1,i}
&=
\frac{\tau_*^\beta}{\underline\tau_n}
\frac{1}{\tau_*^\beta T}\sum_{i=1}^{[n\tau_*^\beta]}\check\zeta_{1,i}
+
\frac{\underline\tau_n-\tau_*^\beta}{\underline\tau_n}
\frac{1}{(\underline\tau_n-\tau_*^\beta)T}
\sum_{i=[n\tau_*^\beta]+1}^{[n\underline\tau_n]}\check\zeta_{1,i}
\\
&\pto
\frac{\tau_*^\beta}{\tau_*^\alpha}\mathcal H_{1,1}
+\biggl(1-\frac{\tau_*^\beta}{\tau_*^\alpha}\biggr)\mathcal H_{1,2},
\end{align*}
\begin{align*}
\frac{1}{\underline\tau_n T}
\sum_{i=1}^{[n\tau_*^\beta]}\check\zeta_{1,i}
-
\frac{[n\tau_*^\beta]}{[n\underline\tau_n]}
\frac{1}{\underline\tau_n T}
\sum_{i=1}^{[n\underline\tau_n]}\check\zeta_{1,i}
&=
\frac{\tau_*^\beta}{\underline\tau_n}
\frac{1}{\tau_*^\beta T}
\sum_{i=1}^{[n\tau_*^\beta]}\check\zeta_{1,i}
-
\frac{[n\tau_*^\beta]}{[n\underline\tau_n]}
\frac{1}{\underline\tau_n T}
\sum_{i=1}^{[n\underline\tau_n]}\check\zeta_{1,i}
\\
&\pto
\frac{\tau_*^\beta}{\tau_*^\alpha}
\Biggl(
\mathcal H_{1,1}
-\biggl(
\frac{\tau_*^\beta}{\tau_*^\alpha}\mathcal H_{1,1}
+\biggl(1-\frac{\tau_*^\beta}{\tau_*^\alpha}
\biggr)\mathcal H_{1,2}
\biggr)
\Biggr)
\\
&=
\frac{\tau_*^\beta}{\tau_*^\alpha}
\biggl(1-\frac{\tau_*^\beta}{\tau_*^\alpha}\biggr)
\bigl(\mathcal H_{1,1}-\mathcal H_{1,2}\bigr)
\neq0,
\end{align*}
and
\begin{align*}
\TUT
&=
\frac{1}{\sqrt{\underline{\tau}_n T}}
\max_{1\le k\le [n\underline{\tau}_n]}
\left\|
\mathcal I_{1,n}^{-1/2}
\left(
\sum_{i=1}^k\check\zeta_{1,i}
-\frac{k}{[n\underline{\tau}_n]}\sum_{i=1}^{[n\underline{\tau}_n]}\check\zeta_{1,i}
\right)
\right\|
\\
&\ge
\frac{1}{\sqrt{\underline{\tau}_n T}}
\left\|
\mathcal I_{1,n}^{-1/2}
\left(
\sum_{i=1}^{[n\tau_*^\beta]}\check\zeta_{1,i}
-\frac{[n\tau_*^\beta]}{[n\underline{\tau}_n]}
\sum_{i=1}^{[n\underline{\tau}_n]}\check\zeta_{1,i}
\right)
\right\|
\\
&=
\sqrt{\underline{\tau}_n T}
\left\|
\mathcal I_{1,n}^{-1/2}
\left(
\frac{1}{\underline\tau_n T}
\sum_{i=1}^{[n\tau_*^\beta]}\check\zeta_{1,i}
-\frac{[n\tau_*^\beta]}{[n\underline{\tau}_n]}
\frac{1}{\underline\tau_n T}
\sum_{i=1}^{[n\underline{\tau}_n]}\check\zeta_{1,i}
\right)
\right\|
\\
&\pto\infty.
\end{align*}
\eqref{eq-C-1} can be shown similarly to (4.59) of Tonaki et al. (2020).

\textit{Proof of \eqref{eq-C-2}. }
We can prove 
\begin{align}
\frac{1}{(\tau_*^\alpha-\tau_*^\beta)T}
\sum_{i=[n\tau_*^\beta]+1}^{[n\tau_*^\alpha]}\check\zeta_{1,i}
\pto
\mathcal H_{1,2}
\label{eq-C-3}
\end{align}
with the same argument as (4.60) of Tonaki et al. (2020).
We have
\begin{align*}
\Delta_n
&=
\left\|
\frac{1}{(\underline\tau_n-\tau_*^\beta) T}
\sum_{i=[n\tau_*^\beta]+1}^{[n\underline\tau_n]}\check\zeta_{1,i}
-\frac{1}{(\tau_*^\alpha-\tau_*^\beta) T}
\sum_{i=[n\tau_*^\beta]+1}^{[n\tau_*^\alpha]}\check\zeta_{1,i}
\right\|
\\
&=
\frac{1}{T}
\left\|
\left(
\frac{1}{\underline\tau_n-\tau_*^\beta}
-\frac{1}{\tau_*^\alpha-\tau_*^\beta}
\right)
\sum_{i=[n\tau_*^\beta]+1}^{[n\underline\tau_n]}\check\zeta_{1,i}
+
\frac{1}{\tau_*^\alpha-\tau_*^\beta}
\left(
\sum_{i=[n\tau_*^\beta]+1}^{[n\underline\tau_n]}\check\zeta_{1,i}
-\sum_{i=[n\tau_*^\beta]+1}^{[n\tau_*^\alpha]}\check\zeta_{1,i}
\right)
\right\|
\\
&\le
\left|
\frac{\tau_*^\alpha-\underline\tau_n}
{(\underline\tau_n-\tau_*^\beta)(\tau_*^\alpha-\tau_*^\beta)}
\right|
\frac{1}{T}
\left\|
\sum_{i=[n\tau_*^\beta]+1}^{[n\underline\tau_n]}\check\zeta_{1,i}
\right\|
+
\frac{1}{(\tau_*^\alpha-\tau_*^\beta) T}
\left\|
\sum_{i=[n\tau_*^\beta]+1}^{[n\underline\tau_n]}\check\zeta_{1,i}
-\sum_{i=[n\tau_*^\beta]+1}^{[n\tau_*^\alpha]}\check\zeta_{1,i}
\right\|
\\
&=
\left|
\frac{n^{\epsilon_1}(\tau_*^\alpha-\underline\tau_n)}
{(\underline\tau_n-\tau_*^\beta)(\tau_*^\alpha-\tau_*^\beta)}
\right|
\frac{n^{-\epsilon_1}}{T}
\left\|
\sum_{i=[n\tau_*^\beta]+1}^{[n\underline\tau_n]}\check\zeta_{1,i}
\right\|
+
\frac{1}{(\tau_*^\alpha-\tau_*^\beta) T}
\left\|
\sum_{i=[n\tau_*^\beta]+1}^{[n\underline\tau_n]}\check\zeta_{1,i}
-\sum_{i=[n\tau_*^\beta]+1}^{[n\tau_*^\alpha]}\check\zeta_{1,i}
\right\|
\\
&=:
\left|
\frac{n^{\epsilon_1}(\tau_*^\alpha-\underline\tau_n)}
{(\underline\tau_n-\tau_*^\beta)(\tau_*^\alpha-\tau_*^\beta)}
\right|
\check{\mathcal R}_n
+
\frac{1}{\tau_*^\alpha-\tau_*^\beta}
\check{\mathcal V}_n.
\end{align*}
If we prove $\check{\mathcal R}_n\pto0$ and $\check{\mathcal V}_n\pto0$,
then, from $n^{\epsilon_1}(\tau_*^\alpha-\underline\tau_n)=O_p(1)$ 
and \eqref{eq-C-3}, 
we have $\Delta_n\pto0$ and \eqref{eq-C-2}.  
In the following, we prove them.

Set 
\begin{align*}
\mathcal Z_{k,i}^{[j]}
=\frac{1}{j!}
\partial_\beta^{j}
\Bigl(
\partial_\beta b_{i-1}(\beta)^\TT A_{i-1}^{-1}(\alpha_k^*)\DeX
\Bigr)\Bigl|_{\beta=\beta_k'},
\quad
\mathcal N_{k,i}^{[j]}
=\mathcal Z_{k,i}^{[j]}
-\EE_{\theta_k}[\mathcal Z_{k,i}^{[j]}|\GG].
\end{align*} 
Note that
\begin{align*}
\check\zeta_{k,i}
&=
\partial_\beta b_{i-1}(\hat\beta_k)^\TT 
A_{i-1}^{-1}(\hat\alpha_k)(\DeX-hb_{i-1}(\hat\beta_k))\\
&=\partial_\beta b_{i-1}(\hat\beta_k)^\TT 
A_{i-1}^{-1}(\alpha_k^*)(\DeX-hb_{i-1}(\hat\beta_k))
\\
&\qquad+
\frac{1}{\sqrt n}\int_0^1\left.
\partial_{\alpha}\left(
\partial_\beta b_{i-1}(\hat\beta_k)^\TT 
A_{i-1}^{-1}(\alpha)(\DeX-hb_{i-1}(\hat\beta_k))\right)
\right|_{\alpha=\alpha_k^*+u(\hat\alpha_k-\alpha_k^*)}
\dd u
\sqrt n(\hat\alpha_k-\alpha_k^*)\\
&=\partial_\beta b_{i-1}(\hat\beta_k)^\TT 
A_{i-1}^{-1}(\alpha_k^*)\DeX
-h \partial_\beta b_{i-1}(\hat\beta_k)^\TT 
A_{i-1}^{-1}(\alpha_k^*)b_{i-1}(\hat\beta_k)\\
&\qquad+
\frac{1}{\sqrt n}\int_0^1\left.
\partial_{\alpha}\left(
\partial_\beta b_{i-1}(\hat\beta_k)^\TT 
A_{i-1}^{-1}(\alpha)(\DeX-hb_{i-1}(\hat\beta_k))\right)
\right|_{\alpha=\alpha_k^*+u(\hat\alpha_k-\alpha_k^*)}
\dd u
\sqrt n(\hat\alpha_k-\alpha_k^*)\\
&=\partial_\beta b_{i-1}(\hat\beta_k)^\TT 
A_{i-1}^{-1}(\alpha_k^*)\DeX
+O_p(h)
\\
&=
\sum_{j=0}^{\mma-1}
\mathcal Z_{k,i}^{[j]}
\otimes(\hat\beta_k-\beta_k')^{\otimes j}
+O_p\left(h\lor\frac{\sqrt h}{T^{\mma/2}}\right)
\\
&=
\sum_{j=0}^{\mma-1}
(\mathcal N_{k,i}^{[j]}+\EE_{\theta_k}[\mathcal Z_{k,i}^{[j]}|\GG])
\otimes(\hat\beta_k-\beta_k')^{\otimes j}
+O_p\left(h\lor\frac{\sqrt h}{T^{\mma/2}}\right)
\\
&=
\sum_{j=0}^{\mma-1}
\mathcal N_{k,i}^{[j]}
\otimes(\hat\beta_k-\beta_k')^{\otimes j}
+O_p(h).
\end{align*}
Let
\begin{align*}
\mathcal R_n^{[j]}
=
\frac{n^{-\epsilon_1}}{T}
\left\|
\sum_{i=[n\tau_*^\beta]+1}^{[n\underline\tau_n]}
\mathcal N_{1,i}^{[j]}
\right\|,
\quad
\mathcal V_n^{[j]}
=
\frac{1}{T}
\left\|
\sum_{i=[n\underline\tau_n]+1}^{[n\tau_*^\alpha]}
\mathcal N_{1,i}^{[j]}
\right\|.
\end{align*}
We have
\begin{align*}
\check{\mathcal R}_n
&\le
\sum_{j=0}^{\mma-1}
\mathcal R_n^{[j]}
|\hat\beta_1-\beta_1'|^j
+O_p\left(
\frac{n^{1-\epsilon_1}h}{T}
\right)
=\sum_{j=0}^{\mma-1}\mathcal R_n^{[j]}
|\hat\beta_1-\beta_1'|^j
+o_p(1),
\end{align*}
and
\begin{align*}
\check{\mathcal V}_n
=
\frac{1}{T}
\left\|
\sum_{i=[n\underline\tau_n]+1}^{[n\tau_*^\alpha]}\check\zeta_{1,i}
\right\|
\le
\sum_{j=0}^{\mma-1}
\mathcal V_n^{[j]}
|\hat\beta_1-\beta_1'|^j
+O_p(n^{-\epsilon_1})
=
\sum_{j=0}^{\mma-1}\mathcal V_n^{[j]}
|\hat\beta_1-\beta_1'|^j
+o_p(1)
\end{align*}
on $D_n$.
Let $E_n=\{|\hat\beta_1-\beta_1'|\le1\}$.
Noting that 
$P(E_n^{\mathrm c})\lto0$ 
from \textbf{[F1]}, 
for all $\epsilon>0$, 
\begin{align*}
P\bigl(\check{\mathcal R}_n>(\mma+1)\epsilon\bigr)
&\le
P\Biggl(\sum_{j=0}^{\mma-1}
\mathcal R_n^{[j]}|\hat\beta_1-\beta_1'|^j>\mma\epsilon\Biggr)
+o(1)
\nonumber
\\
&\le
P\Biggl(
\sum_{j=0}^{\mma-1}\mathcal R_n^{[j]}|\hat\beta_1-\beta_1'|^j>\mma\epsilon, 
D_n\cap E_n
\Biggr)
+P(D_n^{\mathrm c})+P(E_n^{\mathrm c})+o(1)
\nonumber
\\
&\le
P\Biggl(
\sum_{j=0}^{\mma-1}\mathcal R_n^{[j]}>\mma\epsilon, 
D_n\cap E_n
\Biggr)
+o(1)
\nonumber
\\
&\le
\sum_{j=0}^{\mma-1}
P\bigl(
\mathcal R_n^{[j]}>\epsilon, D_n
\bigr)
+o(1)
\nonumber
\\
&\le
\frac{1}{\epsilon^2}
\sum_{j=0}^{\mma-1}
\EE_{\theta_1}[(\mathcal R_n^{[j]})^2: D_n]
+o(1),
\\
P\bigl(\check{\mathcal V}_n>(\mma+1)\epsilon\bigr)
&\le
P\bigl(\check{\mathcal V}_n>(\mma+1)\epsilon, D_n\bigr)
+P(D_n^{\mathrm c})
\nonumber
\\
&\le
P\Biggl(\sum_{j=0}^{\mma-1}
\mathcal V_n^{[j]}|\hat\beta_1-\beta_1'|^j>\mma\epsilon, D_n\Biggl)
+o(1)
\nonumber
\\
&\le
P\Biggl(\sum_{j=0}^{\mma-1}
\mathcal V_n^{[j]}|\hat\beta_1-\beta_1'|^j>m\epsilon, D_n\cap E_n\Biggl)
+P(E_n^{\mathrm c})+o(1)
\nonumber
\\
&\le
P\Biggl(\sum_{j=0}^{\mma-1}
\mathcal V_n^{[j]}>\mma\epsilon, D_n\cap E_n\Biggl)
+o(1)
\nonumber
\\
&\le
\sum_{j=0}^{\mma-1}
P\bigl(
\mathcal V_n^{[j]}>\epsilon, D_n
\bigr)
+o(1)
\nonumber
\\
&\le
\frac{1}{\epsilon^2}
\sum_{j=0}^{\mma-1}
\EE_{\theta_1}[(\mathcal V_n^{[j]})^2: D_n]
+o(1).
\end{align*}
Since $\EE_{\theta_1}[\|\mathcal N_{1,i}^{[j]}\|^2]\le Ch$,
\begin{align*}
\frac{n^{-2\epsilon_1}([n\tau_*^\alpha]-[n\tau_*^\beta])h}{T^2}
=O\left(\frac{n^{-2\epsilon_1}}{T}\right)=o(1),
\quad
\frac{([n\tau_*^\alpha]-m_n)h}{T^2}
=O\left(\frac{n^{-\epsilon_1}}{T}\right)=o(1),
\end{align*}
we have
$\EE_{\theta_1}[(\mathcal R_n^{[j]})^2:D_n]=o(1)$ and 
$\EE_{\theta_1}[(\mathcal V_n^{[j]})^2:D_n]=o(1)$
for $0\le j\le \mma-1$ as in Lemma \ref{lem1}.
Hence, we get the desired results. 


(b) Notice that 
\begin{align*}
\check\zeta_{k,i}
&=
\sum_{j=0}^{\mmb-1}
\mathcal N_{k,i}^{[j]}
\otimes(\hat\beta_k-\beta_k')^{\otimes j}
+O_p(h)
\end{align*}
from \textbf{[E4]}.
Then, the desired result can be obtained in the same way as in (a).


(c) According to
\begin{align*}
\TUT
&=
\frac{1}{\sqrt{\underline{\tau}_n T}}
\max_{1\le k\le [n\underline{\tau}_n]}
\left\|
\mathcal I_{1,n}^{-1/2}
\left(
\sum_{i=1}^k\check\zeta_{1,i}
-\frac{k}{[n\underline{\tau}_n]}\sum_{i=1}^{[n\underline{\tau}_n]}\check\zeta_{1,i}
\right)
\right\|
\\
&\ge
\frac{1}{\sqrt{\underline{\tau}_n T}}
\left\|
\mathcal I_{1,n}^{-1/2}
\left(
\sum_{i=1}^{[n\tau_*^\beta]}\check\zeta_{1,i}
-\frac{[n\tau_*^\beta]}{[n\underline{\tau}_n]}
\sum_{i=1}^{[n\underline{\tau}_n]}\check\zeta_{1,i}
\right)
\right\|
\\
&=
\sqrt{\frac{T\Deko^2}{\underline{\tau}_n}}
\left\|
\mathcal I_{1,n}^{-1/2}
\left(
\frac{1}{T\Deko}
\sum_{i=1}^{[n\tau_*^\beta]}\check\zeta_{1,i}
-\frac{[n\tau_*^\beta]}{[n\underline{\tau}_n]}
\frac{1}{T\Deko}
\sum_{i=1}^{[n\underline{\tau}_n]}\check\zeta_{1,i}
\right)
\right\|,
\end{align*}
it is sufficient to prove that there exists $c\neq0$ such that
\begin{align}
\mathcal K_n^{(1)}
=
\frac{1}{T\Deko}
\sum_{i=1}^{[n\tau_*^\beta]}\check\zeta_{1,i}
-
\frac{[n\tau_*^\beta]}{[n\underline\tau_n]}
\frac{1}{T\Deko}
\sum_{i=1}^{[n\underline\tau_n]}\check\zeta_{1,i}
\pto c.
\label{eq-C-6}
\end{align}
Notice that there exists $c'\neq0$ such that
\begin{align}
\mathcal K_n^{(2)}
=
\frac{1}{T\Deko}
\sum_{i=1}^{[n\tau_*^\beta]}\check\zeta_{1,i}
-
\frac{[n\tau_*^\beta]}{[n\tau_*^\alpha]}
\frac{1}{T\Deko}
\sum_{i=1}^{[n\tau_*^\alpha]}\check\zeta_{1,i}
\pto c'.
\label{eq-C-7}
\end{align}
as in the proof of Proposition 3 of Tonaki et al. (2021). 
Meanwhile, we see
\begin{align*}
\Delta_n
&=\|\mathcal K_n^{(1)}-\mathcal K_n^{(2)}\|
\\
&=
\left\|
\frac{[n\tau_*^\beta]}{[n\underline{\tau}_n]}
\frac{1}{T\Deko}
\sum_{i=1}^{[n\underline{\tau}_n]}\check\zeta_{1,i}
-\frac{[n\tau_*^\beta]}{[n\tau_*^\alpha]}
\frac{1}{T\Deko}
\sum_{i=1}^{[n\tau_*^\alpha]}\check\zeta_{1,i}
\right\|
\\
&\le
\frac{[n\tau_*^\beta]}{[n\tau_*^\alpha]}
\frac{1}{T\Deko}
\left\|
\sum_{i=1}^{[n\underline{\tau}_n]}\check\zeta_{1,i}
-
\sum_{i=1}^{[n\tau_*^\alpha]}\check\zeta_{1,i}
\right\|
+
n^{\epsilon_1}
\left|
\frac{[n\tau_*^\beta]}{[n\tau_*^\alpha]}
-\frac{[n\tau_*^\beta]}{[n\underline{\tau}_n]}
\right|
\frac{n^{-\epsilon_1}}{T\Deko}
\left\|
\sum_{i=1}^{[n\tau_*^\alpha]}\check\zeta_{1,i}
\right\|
\\
&=:
\frac{[n\tau_*^\beta]}{[n\tau_*^\alpha]}
\check{\mathcal V}_n
+
n^{\epsilon_1}
\left|
\frac{[n\tau_*^\beta]}{[n\tau_*^\alpha]}
-\frac{[n\tau_*^\beta]}{[n\underline{\tau}_n]}
\right|
\check{\mathcal R}_n
\end{align*}
and
\begin{align*}
\check\zeta_{1,i}
&=
\sum_{j=0}^{\mmc-1}
\mathcal N_{1,i}^{[j]}
\otimes(\hat\beta_1-\beta_1')^{\otimes j}
+O_p(h)
\end{align*}
from \textbf{[G3]}.
Here, we set
\begin{align*}
\mathcal V_n^{[j]}
=
\frac{1}{T\Deko}
\left\|
\sum_{i=[n\underline\tau_n]+1}^{[n\tau_*^\alpha]}
\mathcal N_{1,i}^{[j]}
\right\|,
\quad
\mathcal R_n^{[j]}
=
\frac{n^{-\epsilon_1}}{T\Deko}
\left\|
\sum_{i=1}^{[n\tau_*^\alpha]}
\mathcal N_{1,i}^{[j]}
\right\|,
\end{align*}
and have, from \textbf{[G5]},
\begin{align*}
\check{\mathcal V}_n
=
\frac{1}{T\Deko}
\left\|
\sum_{i=[n\underline\tau_n]+1}^{[n\tau_*^\alpha]}\check\zeta_{1,i}
\right\|
\le
\sum_{j=0}^{\mmc-1}
\mathcal V_n^{[j]}
|\hat\beta_1-\beta_1'|^j
+O_p\left(
\frac{n^{1-\epsilon_1}h}{T\Deko}
\right)
=
\sum_{j=0}^{\mmc-1}\mathcal V_n^{[j]}
|\hat\beta_1-\beta_1'|^j
+o_p(1)
\end{align*}
and 
\begin{align*}
\check{\mathcal R}_n
&\le
\sum_{j=0}^{\mmc-1}
\mathcal R_n^{[j]}
|\hat\beta_1-\beta_1'|^j
+O_p\left(
\frac{n^{1-\epsilon_1}h}{T\Deko}
\right)
=\sum_{j=0}^{\mmc-1}\mathcal R_n^{[j]}
|\hat\beta_1-\beta_1'|^j
+o_p(1)
\end{align*}
on $D_n$.
Thus, for all $\epsilon>0$, 
\begin{align*}
P\bigl(\check{\mathcal V}_n>(\mmc+1)\epsilon\bigr)
&\le
P\bigl(\check{\mathcal V}_n>(\mmc+1)\epsilon, D_n\bigr)
+P(D_n^{\mathrm c})
\nonumber
\\
&\le
P\Biggl(\sum_{j=0}^{\mmc-1}
\mathcal V_n^{[j]}|\hat\beta_1-\beta_1'|^j>\mmc\epsilon, D_n\Biggl)
+o(1)
\nonumber
\\
&\le
P\Biggl(\sum_{j=0}^{\mmc-1}
\mathcal V_n^{[j]}|\hat\beta_1-\beta_1'|^j>\mmc\epsilon, D_n\cap E_n\Biggl)
+P(E_n^{\mathrm c})+o(1)
\nonumber
\\
&\le
P\Biggl(\sum_{j=0}^{\mmc-1}
\mathcal V_n^{[j]}>\mmc\epsilon, D_n\cap E_n\Biggl)
+o(1)
\nonumber
\\
&\le
\sum_{j=0}^{\mmc-1}
P\bigl(
\mathcal V_n^{[j]}>\epsilon, D_n
\bigr)
+o(1)
\nonumber
\\
&\le
\frac{1}{\epsilon^2}
\sum_{j=0}^{\mmc-1}
\EE_{\theta_1}[(\mathcal V_n^{[j]})^2: D_n]
+o(1),
\\
P\bigl(\check{\mathcal R}_n>(\mmc+1)\epsilon\bigr)
&\le
P\Biggl(
\sum_{j=0}^{\mmc-1}\mathcal R_n^{[j]}|\hat\beta_1-\beta_1'|^j>\mmc\epsilon
\Biggr)
+o(1)
\nonumber
\\
&\le
P\Biggl(
\sum_{j=0}^{\mmc-1}\mathcal R_n^{[j]}|\hat\beta_1-\beta_1'|^j>\mmc\epsilon, E_n
\Biggr)
+P(E_n^{\mathrm c})+o(1)
\nonumber
\\
&\le
P\Biggl(
\sum_{j=0}^{\mmc-1}\mathcal R_n^{[j]}>\mmc\epsilon, E_n
\Biggr)
+o(1)
\nonumber
\\
&\le
\sum_{j=0}^{\mmc-1}
P\bigl(
\mathcal R_n^{[j]}>\epsilon
\bigr)
+o(1)
\nonumber
\\
&\le
\frac{1}{\epsilon^2}
\sum_{j=0}^{\mmc-1}
\EE_{\theta_1}[(\mathcal R_n^{[j]})^2]
+o(1).
\end{align*}
Since
\begin{align*}
\frac{([n\tau_*^\alpha]-m_n)h}{T^2\Deko^2}
=O\left(\frac{n^{-\epsilon_1}}{T\Deko^2}\right)=o(1),
\quad
\frac{n^{-2\epsilon_1}[n\tau_*^\alpha]h}{T^2\Deko^2}
=O\left(\frac{n^{-2\epsilon_1}}{T\Deko^2}\right)=o(1),
\end{align*}
we have
$\EE_{\theta_1}[(\mathcal V_n^{[j]})^2:D_n]=o(1)$ and
$\EE_{\theta_1}[(\mathcal R_n^{[j]})^2]=o(1)$
for $0\le j \le \mmc-1$ as in Lemma \ref{lem1}, that is,
we obtain $\check{\mathcal V}_n\pto0$ and $\check{\mathcal R}_n\pto0$,
which implies $\Delta_n\pto0$. 
This and \eqref{eq-C-7} yield \eqref{eq-C-6}.


\textit{Step 2.}
Next, we show $P(\TOT>w_1(\epsilon))\lto1$ under $H_1^{(2)}$.

(a) If we prove
\begin{align}
\frac{1}{(\tau_*^\beta-{\overline\tau}_n) T}
\sum_{i=[n\overline\tau_n]+1}^{[n\tau_*^\beta]}\check\zeta_{2,i}
\pto
\mathcal H_{2,1},
\quad
\frac{1}{(1-\tau_*^\beta)T}
\sum_{i=[n\tau_*^\beta]+1}^n\check\zeta_{2,i}
\pto
\mathcal H_{2,2},
\label{eq-D-2}
\end{align}
then
\begin{align*}
\frac{1}{(1-\overline{\tau}_n) T}
\sum_{i=[n\overline{\tau}_n]+1}^{n}\check\zeta_{2,i}
&=
\frac{\tau_*^\beta-\overline{\tau}_n}{1-\overline{\tau}_n}
\frac{1}{(\tau_*^\beta-\overline{\tau}_n) T}
\sum_{i=[n\overline{\tau}_n]+1}^{[n\tau_*^\beta]}\check\zeta_{2,i}
+
\frac{1-\tau_*^\beta}{1-\overline{\tau}_n}
\frac{1}{(1-\tau_*^\beta) T}
\sum_{i=[n\tau_*^\beta]+1}^{n}\check\zeta_{2,i}
\\
&\pto
\frac{\tau_*^\beta-\tau_*^\alpha}{1-\tau_*^\alpha}
\mathcal H_{2,1}
+\biggl(1-\frac{\tau_*^\beta-\tau_*^\alpha}{1-\tau_*^\alpha}\biggr)
\mathcal H_{2,2},
\end{align*}
\begin{align*}
&\frac{1}{(1-\overline{\tau}_n) T}
\sum_{i=[n\overline{\tau}_n]+1}^{[n\tau_*^\beta]}\check\zeta_{2,i}
-\frac{[n\tau_*^\beta]-[n\overline{\tau}_n]}{n-[n\overline{\tau}_n]}
\frac{1}{(1-\overline{\tau}_n) T}
\sum_{i=[n\overline{\tau}_n]+1}^{n}\check\zeta_{2,i}
\\
&=
\frac{\tau_*^\beta-\overline{\tau}_n}{1-\overline{\tau}_n}
\frac{1}{(\tau_*^\beta-\overline{\tau}_n) T}
\sum_{i=[n\overline{\tau}_n]+1}^{[n\tau_*^\beta]}\check\zeta_{2,i}
-\frac{[n\tau_*^\beta]-[n\overline{\tau}_n]}{n-[n\overline{\tau}_n]}
\frac{1}{(1-\overline{\tau}_n) T}
\sum_{i=[n\overline{\tau}_n]+1}^{n}\check\zeta_{2,i}
\\
&\pto
\frac{\tau_*^\beta-\tau_*^\alpha}{1-\tau_*^\alpha}
\Biggl(
\mathcal H_{2,1}
-\biggl(
\frac{\tau_*^\beta-\tau_*^\alpha}{1-\tau_*^\alpha}
\mathcal H_{2,1}
+\biggl(1-\frac{\tau_*^\beta-\tau_*^\alpha}{1-\tau_*^\alpha}\biggr)
\mathcal H_{2,2}
\biggr)
\Biggr)
\\
&=
\frac{\tau_*^\beta-\tau_*^\alpha}{1-\tau_*^\alpha}
\biggl(1-\frac{\tau_*^\beta-\tau_*^\alpha}{1-\tau_*^\alpha}\biggr)
\bigl(\mathcal H_{2,1}-\mathcal H_{2,2}\bigr)
\neq0,
\end{align*}
and
\begin{align*}
\TOT
&=
\frac{1}{\sqrt{(1-\overline{\tau}_n) T}}
\max_{1\le k\le n-[n\overline{\tau}_n]}
\left\|
\mathcal I_{2,n}^{-1/2}
\left(
\sum_{i=[n\overline{\tau}_n]+1}^{[n\overline{\tau}_n]+k}\check\zeta_{2,i}
-\frac{k}{n-[n\overline{\tau}_n]}
\sum_{i=[n\overline{\tau}_n]+1}^{n}\check\zeta_{2,i}
\right)
\right\|
\\
&\ge
\frac{1}{\sqrt{(1-\overline{\tau}_n) T}}
\left\|
\mathcal I_{2,n}^{-1/2}
\left(
\sum_{i=[n\overline{\tau}_n]+1}^{[n\tau_*^\beta]}\check\zeta_{2,i}
-\frac{[n\tau_*^\beta]-[n\overline{\tau}_n]}{n-[n\overline{\tau}_n]}
\sum_{i=[n\overline{\tau}_n]+1}^{n}\check\zeta_{2,i}
\right)
\right\|
\\
&=
\sqrt{(1-\overline{\tau}_n) T}
\left\|
\mathcal I_{2,n}^{-1/2}
\left(
\frac{1}{(1-\overline{\tau}_n) T}
\sum_{i=[n\overline{\tau}_n]+1}^{[n\tau_*^\beta]}\check\zeta_{2,i}
-\frac{[n\tau_*^\beta]-[n\overline{\tau}_n]}{n-[n\overline{\tau}_n]}
\frac{1}{(1-\overline{\tau}_n) T}
\sum_{i=[n\overline{\tau}_n]+1}^{n}\check\zeta_{2,i}
\right)
\right\|
\\
&\pto\infty.
\end{align*}
\eqref{eq-D-2} can be proved in the same way as
\eqref{eq-C-2} and (4.60) of Tonaki et al. (2020). 


(c) According to 
\begin{align*}
\TOT
&=
\frac{1}{\sqrt{(1-\overline{\tau}_n) T}}
\max_{1\le k\le n-[n\overline{\tau}_n]}
\left\|
\mathcal I_{2,n}^{-1/2}
\left(
\sum_{i=[n\overline{\tau}_n]+1}^{[n\overline{\tau}_n]+k}\check\zeta_{2,i}
-\frac{k}{n-[n\overline{\tau}_n]}
\sum_{i=[n\overline{\tau}_n]+1}^{n}\check\zeta_{2,i}
\right)
\right\|
\\
&\ge
\frac{1}{\sqrt{(1-\overline{\tau}_n) T}}
\left\|
\mathcal I_{2,n}^{-1/2}
\left(
\sum_{i=[n\overline{\tau}_n]+1}^{[n\tau_*^\beta]}\check\zeta_{2,i}
-\frac{[n\tau_*^\beta]-[n\overline{\tau}_n]}{n-[n\overline{\tau}_n]}
\sum_{i=[n\overline{\tau}_n]+1}^{n}\check\zeta_{2,i}
\right)
\right\|
\\
&=
\sqrt{\frac{T\Dekt^2}{1-\overline{\tau}_n}}
\left\|
\mathcal I_{2,n}^{-1/2}
\left(
\frac{1}{T\Dekt}
\sum_{i=[n\overline{\tau}_n]+1}^{[n\tau_*^\beta]}\check\zeta_{2,i}
-\frac{[n\tau_*^\beta]-[n\overline{\tau}_n]}{n-[n\overline{\tau}_n]}
\frac{1}{T\Dekt}
\sum_{i=[n\overline{\tau}_n]+1}^{n}\check\zeta_{2,i}
\right)
\right\|,
\end{align*}
it is sufficient to show that there exists $c\neq0$ such that
\begin{align*}
\frac{1}{T\Dekt}
\sum_{i=[n\overline{\tau}_n]+1}^{[n\tau_*^\beta]}\check\zeta_{2,i}
-\frac{[n\tau_*^\beta]-[n\overline{\tau}_n]}{n-[n\overline{\tau}_n]}
\frac{1}{T\Dekt}
\sum_{i=[n\overline{\tau}_n]+1}^{n}\check\zeta_{2,i}
\pto c,
\end{align*}
which can be derived in the same way as \eqref{eq-C-6}.
\end{proof}


\begin{proof}[\bf{Proof of Theorem \ref{th5}}]
We have
\begin{align*}
&\Psi_{1,n}(\tau:\beta_1,\beta_2|\alpha)
-\Psi_{1,n}(\tau_*^\beta:\beta_1,\beta_2|\alpha)
\nonumber
\\
&=\sum_{i=1}^{[n\tau]}G_i(\beta_1|\alpha)
+\sum_{i=[n\tau]+1}^{[n\overline{\tau}_n]} G_i(\beta_2|\alpha)
-\sum_{i=1}^{[n\tau_*^\beta]}G_i(\beta_1|\alpha)
-\sum_{i=[n\tau_*^\beta]+1}^{[n\overline{\tau}_n]} G_i(\beta_2|\alpha)
\nonumber
\\
&=\sum_{i=[n\tau_*^\beta]+1}^{[n\tau]}
\Bigl(G_i(\beta_1|\alpha)-G_i(\beta_2|\alpha)\Bigr)
\end{align*}
for $\tau_*^\beta<\tau<\tau_*^\alpha$, and
\begin{align*}
\Psi_{1,n}(\tau:\beta_1,\beta_2|\alpha)
-\Psi_{1,n}(\tau_*^\beta:\beta_1,\beta_2|\alpha)
=\sum_{i=[n\tau]+1}^{[n\tau_*^\beta]}
\Bigl(G_i(\beta_2|\alpha)-G_i(\beta_1|\alpha)\Bigr)
\end{align*}
for $\tau<\tau_*^\beta$.
Therefore, 
by the same proof of Theorems 3 and 4 of Tonaki et al. (2021),
we obtain
\begin{align*}
T\Deb^2(\hat\tau_{1,n}^\beta-\tau_*^\beta)
\dto
\underset{v\in\mathbb R}{\mathrm{argmin}}\, \mathbb G_1(v)
\end{align*}
in Case A, 
and 
$T(\hat\tau_{1,n}^\beta-\tau_*^\beta)=O_p(1)$ 
in Case B.
\end{proof}
Theorem \ref{th6} can be proved in the same way as the proof of Theorem \ref{th5}.

\section*{Acknowledgements}
This work was partially supported by JST CREST Grant Number JPMJCR14D7 
and JSPS KAKENHI Grant Number JP17H01100.





\end{document}